\documentclass[a4paper,12pt]{article}
\usepackage{fullpage}
\usepackage{pinlabel}
\usepackage{mathtools}
\usepackage{amsthm,amssymb}
\usepackage{dsfont,mathrsfs}
\usepackage{color}
\usepackage{float}
\usepackage[T1]{fontenc}
\usepackage{hyperref}
\hypersetup{
 pdftitle={Graph morphisms and exhaustion of curve graphs of low-genus surfaces},
 pdfauthor={Jesús Hernández Hernández},
}

\title{Graph morphisms and exhaustion of curve graphs of low-genus surfaces}
\author{Jes\'{u}s Hern\'{a}ndez Hern\'{a}ndez} 
\date{}

\newtheorem{Lema}{Lemma}[section]
\newtheorem{Teo}[Lema]{Theorem}
\newtheorem{Prop}[Lema]{Proposition}
\newtheorem{Cor}[Lema]{Corollary}
\newtheorem*{Teono}{Theorem}
\newtheorem*{Corno}{Corollary}
\newtheorem{Theo}{Theorem}

\theoremstyle{definition}

\newtheorem{Rem}[Lema]{Remark}

\newcommand{\dsty}{\displaystyle}
\newcommand{\fun}[3]{#1: #2 \rightarrow #3}
\newcommand{\ccomp}[1]{\mathcal{C}(#1)}
\newcommand{\acomp}[1]{\mathcal{A}(#1)}
\newcommand{\Mod}[1]{\mathrm{Mod}(#1)}
\newcommand{\EMod}[1]{\mathrm{Mod}^{\pm}(#1)}
\newcommand{\ColonEqq}{\dsty \mathrel{\mathop:}=}

\newcommand{\veps}{\varepsilon}

\newcommand{\Aut}[1]{\mathrm{Aut}(#1)}
\newcommand{\Int}{\mathrm{int}}

\newcommand{\Of}{\mathscr{O}}
\newcommand{\Yf}[1]{\mathscr{Y}(#1)}
\newcommand{\Xf}[1]{\mathfrak{X}(#1)}
\newcommand{\Gf}{\mathscr{G}}
\newcommand{\Df}{\mathscr{D}}
\newcommand{\Dfone}{\mathscr{D}_{1}}
\newcommand{\Dftwo}{\mathscr{D}_{2}}
\newcommand{\Cf}{\mathscr{C}}
\newcommand{\Af}{\mathscr{A}}
\newcommand{\Bf}{\mathscr{B}}
\newcommand{\Ef}{\mathscr{E}}
\newcommand{\out}[1]{\delta_{#1}}
\newcommand{\vout}[1]{\veps_{#1}}
\newcommand{\Zf}{\mathscr{Z}}

\begin{document}
\maketitle
\begin{abstract}
 This work is the extension of the results by the author in \cite{JHH1} and \cite{JHH2} for low-genus surfaces. Let $S$ be an orientable, connected surface of finite topological type, with genus $g \leq 2$, empty boundary, and complexity at least $2$; as a complement of the results of \cite{JHH2}, we prove that any graph endomorphism of the curve graph of $S$ is actually an automorphism. Also, as a complement of the results in \cite{JHH2} we prove that under mild conditions on the complexity of the underlying surfaces any graph morphism between curve graphs is induced by a homeomorphism of the surfaces.
 
 To prove these results, we construct a finite subgraph whose union of iterated rigid expansions is the curve graph $\ccomp{S}$. The sets constructed, and the method of rigid expansion, are closely related to Aramayona and Leiniger's finite rigid sets in \cite{Ara2}. Similarly to \cite{JHH1}, a consequence of our proof is that Aramayona and Leininger's rigid set also exhausts the curve graph via rigid expansions, and the combinatorial rigidity results follow as an immediate consequence, based on the results in \cite{JHH2}.\\[0.5cm]
\end{abstract}
\section*{Introduction}
 
 In this work we suppose $S_{g,n}$ is an orientable surface of finte topological type, with genus $g \geq 0$, $n \geq 0$ punctures, and empty boundary. The \textit{(extended) mapping class group of} $S_{g,n}$, denoted by $\EMod{S_{g,n}}$ is the group of isotopy classes of self-homeomorphisms of $S_{g,n}$.
 
 In 1979 (see \cite{Harvey}) Harvey introduced what is now one of the main tools for the study of $\EMod{S_{g,n}}$: the \textit{curve graph of} $S_{g,n}$, denoted by $\ccomp{S_{g,n}}$, is  the simplicial graph whose vertices are isotopy classes of (essential simple closed) curves on $S_{g,n}$, where two vertices span an edge if the corresponding isotopy classes are different and have disjoint representatives.
 
 There is a natural action of $\EMod{S_{g,n}}$ on $\ccomp{S_{g,n}}$ by simplicial automorphisms. It is a well-known result by Ivanov (see \cite{Ivanov}), Korkmaz (see \cite{Korkmaz}), and Luo (see \cite{Luo}), that if $(g,n) \neq (1,2)$, then all the simplicial automorphisms of $\ccomp{S_{g,n}}$ are induced by homeomorphisms of $S_{g,n}$.
 
 This result was later extended by Schackleton (see \cite{Shack}) to simplicial maps that are locally injective (recall that a simplicial map is locally injective if its restriction to the star of any vertex, is injective).
 
 Afterwards, Aramayona and Leininger proved in \cite{Ara1} that there exist finite subgraphs of $\ccomp{S_{g,n}}$ with this same property, i.e. there exists a finite subgraph $\Xf{S_{g,n}} < \ccomp{S_{g,n}}$, such that all locally injective maps $\Xf{S_{g,n}} \to \ccomp{S_{g,n}}$ are restrictions to $\Xf{S_{g,n}}$ of an automorphism of $\ccomp{S_{g,n}}$. Additionally, in \cite{Ara2} they call subgraphs with this property \textit{rigid}, and prove (using geometric methods) that there exists an (set-theoretically) increasing sequence of rigid subgraphs of $\ccomp{S_{g,n}}$ whose union is $\ccomp{S_{g,n}}$.
 
 Note that this last result is \textbf{not trivial} since not any supergraph of a rigid subgraph is rigid.
 
 Later on, we reproved this result (see \cite{JHH1}) using combinatorial methods (first introduced in \cite{Ara2}). The main technique used in \cite{JHH1} is called \textit{rigid expansions}:
 
 In the curve graph, a curve $\alpha$ is \textit{uniquely determined} by a set of curves $A$ if $\alpha \notin A$ and $\alpha$ is the unique curve on $S_{g,n}$ that is disjoint from every element in $A$. Given a subgraph $Y < \ccomp{S_{g,n}}$, the \textit{first rigid expansion of} $Y$, denoted by $Y^{1}$, is the subgraph induced by all the curves in $Y$ along with all the curves uniquely determined by subsets of the vertex set of $Y$. The $n$\textit{-th rigid expansion of} $Y$ (denoted by $Y^{n}$) is defined recursively, and we define $Y^{\omega}$ as the union of $Y^{n}$ for all $n \geq 1$. See Section \ref{Prelim} for more details.
 
 In \cite{Ara2} it is proved that if $X$ is rigid, then $X^{n}$ is rigid for all $n \geq 1$, and in \cite{JHH1} we proved that if $g \geq 3$, then $\Xf{S_{g,n}}^{\omega} = \ccomp{S_{g,n}}$. The majority of this work is about extending this result to all surfaces with $\kappa(S_{g,n}) \ColonEqq 3g-3+n \geq 2$. To do this, we first prove there exists a convenient finite subgraph $\Yf{S_{g,n}}$ that exhausts the curve graph via rigid expansions.
 
  \begin{Theo}\label{TeoA}
  Let $S_{g,n}$ be an orientable, connected surface of finite topological type, with genus $g \leq 2$, $n \geq 0$ punctures, and empty boundary. If $\kappa(S_{g,n}) \geq 2$, then there exists a finite subgraph $\Yf{S_{g,n}}$ of $\ccomp{S_{g,n}}$ whose union of iterated rigid expansions is equal to $\ccomp{S_{g,n}}$.
 \end{Theo}
 
 Afterwards, we use this result to prove that $\Xf{S_{g,n}}$ also exhausts the curve graph, with the methods varying depending on the genus of the surface.
 
 \begin{Theo}\label{TeoB}
  Let $S_{g,n}$ be an orientable, connected surface of finite topological type and empty boundary. If $\kappa(S_{g,n}) \geq 2$, then $\Xf{S_{g,n}}^{\omega} = \ccomp{S_{g,n}}$.
 \end{Theo}
 
 These results for the case $g \geq 3$, were used by the author in \cite{JHH2} 
 to prove that if $S_{1} = S_{g_{1},n_{1}}$, $S_{2} = S_{g_{2},n_{2}}$ with $g_{1} \geq 3$ and the complexity of $S_{2}$ bounded above by the complexity of $S_{1}$, then any graph morphism between $\ccomp{S_{1}}$ and $\ccomp{S_{2}}$ is induced by a homeomorphism $S_{1} \rightarrow S_{2}$.
 
 Recently, in \cite{IrmakEP} and \cite{IrmakComp2}, Irmak extended this result for the case of endomorphisms, and proved that any endomorphism of the curve graph of $S_{g,n}$ with $\kappa(S_{g,n}) \geq 2$ is induced by a homeomorphism of $S_{g,n}$.
 
 
 In Section \ref{sec5} we give a short proof of this result using Theorem \ref{TeoB} and the results from \cite{JHH1} and \cite{JHH2}.
 
 \begin{Theo}\label{TeoD}
  Let $S_{g,n}$ be an orientable, connected surface of finite topological type, with empty boundary, and $\kappa(S_{g,n}) \geq 2$. Let also $\fun{\varphi}{\ccomp{S_{g,n}}}{\ccomp{S_{g,n}}}$ be a graph (endo)morphism. Then $\varphi$ is an automorphism of $\ccomp{S_{g,n}}$.
 \end{Theo}
 
 Also in Section \ref{sec5} we give an extension of the results in \cite{Shack} and \cite{JHH2}, proving that with a restriction on the complexity of surfaces $S_{1}$ and $S_{2}$, and with $S_{1} \neq S_{0,7}$, then any graph morphism between the respective curve graphs, is induced by a homeomorphism between $S_{1}$ and $S_{2}$.
 
 \begin{Theo}\label{TeoE}
  Let $S_{1}$ and $S_{2}$ be two orientable, connected surfaces of finite topological type, with empty boundary, such that $\kappa(S_{1}) \geq \kappa(S_{2}) \geq 2$, and $S_{1}$ is not homeomorphic to a 7-punctured sphere. Let also $\fun{\varphi}{\ccomp{S_{1}}}{\ccomp{S_{2}}}$ be a graph morphism. Then, $S_{1}$ is homeomorphic to $S_{2}$ and $\varphi$ is induced by a homeomorphism $S_{1} \to S_{2}$.
 \end{Theo}
 
 Note that, due to Theorem \ref{TeoD}, it is sufficient to prove topological rigidity, i.e. under the hypothesis of Theorem \ref{TeoE}, $S_{1}$ is homeomorphic to $S_{2}$ (see Theorem \ref{TopRig}).

 
 
 
 Also, the lower bound on the complexity of Theorem \ref{TeoD} is sharp. Indeed, if $\kappa(S_{g,n}) = 1$, there exist infinitely many graph endomorphisms of $\ccomp{S_{g,n}}$ that are surjective, but are NOT injective; thus, they cannot be induced by an automorphism.
 
 As well, the lower bound on the complexity of Theorem \ref{TeoE} is also sharp. Indeed, $\ccomp{S_{0,5}}$ and $\ccomp{S_{0,6}}$ are isomorphic to $\ccomp{S_{1,2}}$ and $\ccomp{S_{2,0}}$, respectively.

 \textbf{Reader's guide:} This work is divided as follows: in Section \ref{Prelim} we give the necessary preliminaries for this work; in Section \ref{sec2} we prove Theorems \ref{TeoA} and \ref{TeoB} for the case of genus zero; in Section \ref{sec3} we prove Theorems \ref{TeoA} and \ref{TeoB} for the case of genus one; in Section \ref{sec4} we prove Theorems \ref{TeoA} and \ref{TeoB} for the case of genus two; in Section \ref{sec5} we prove Theorems \ref{TeoD} and \ref{TeoE}.\\[0.3cm]
 \textbf{Acknowledgements:} 
  The author was supported during the creation of this article by the UNAM-PAPIIT research grants IA104620 and IN114323. The author was also supported by the CONAHCYT research grant Ciencia de Frontera 2019 CF 217392.
\section{Preliminaries}\label{Prelim}
 Suppose $S_{g,n}$ is an orientable, connected surface of finite topological type with empty boundary, genus $g \leq 2$, $n \geq 0$ punctures, and such that its complexity $\kappa(S_{g,n}) = 3g-3+n$ is at least $3$. The \textit{mapping class group of} $S_{g,n}$, denoted by $\Mod{S_{g,n}}$, is the group of all orientation-preserving self-homeomorphisms of $S_{g,n}$.
 
 In this work, a \textit{curve on} $S_{g,n}$ is a topological embedding of $\mathbb{S}^{1}$ into the surface. We often abuse notation and call ``curve'' the embedding, its image on $S_{g,n}$, or its isotopy class. The meaning is clear with the context.
 
 A curve on $S_{g,n}$ is \textit{essential} if it is neither homotopic to a point, nor to the boundary curve of a neighbourhood of a puncture. Hereafter we assume every curve to be essential unless otherwise stated. See Figure \ref{fig:PrelimFig1} for examples.
 
 The curves on $S_{g,n}$ can be classified as follows: A curve $\alpha$ on $S_{g,n}$ is \textit{separating} if $S_{g,n} \backslash \{\alpha\}$ is disconnected, and it is \textit{non-separating} otherwise. A special case of a separating curve is when $S_{g,n} \backslash \{\alpha\}$ has a connected component homeomorphic to $S_{0,3}$, in which case we say $\alpha$ is an \textit{outer curve}. See Figure \ref{fig:PrelimFig1} for examples.
 
 \begin{figure}[ht]
     \centering
     \includegraphics[height=4cm]{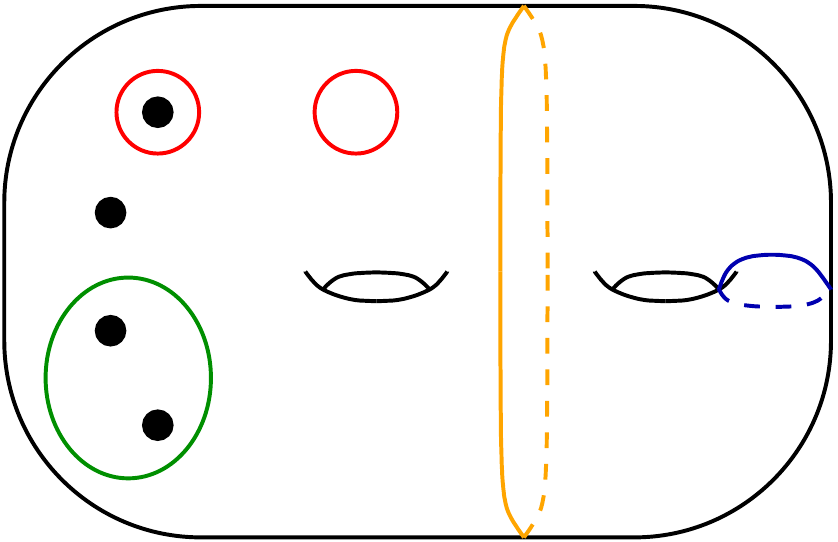}\hspace{1cm}
     \includegraphics[height=4cm]{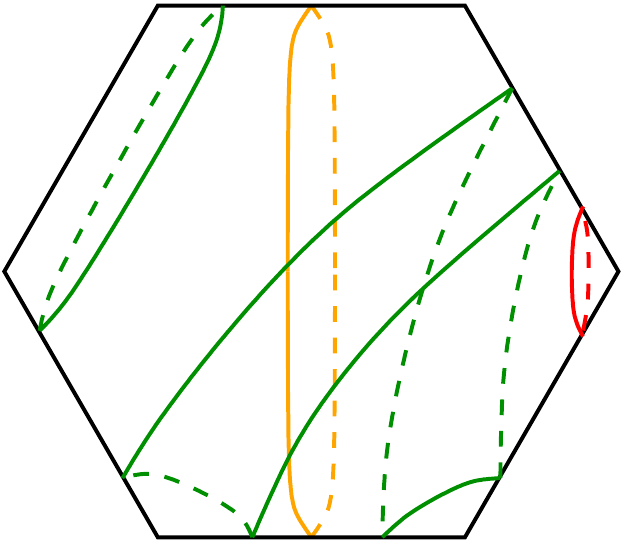}
     \caption{Examples of curves on surfaces. The red curves are non-essential, while the rest are essential. The green curves are outer curves, the orange curves are separating non-outer curves, and the blue curve is non-separating.}
     \label{fig:PrelimFig1}
 \end{figure}
 
 Note that there exist non-separating curves on $S_{g,n}$ if and only if $S_{g,n}$ has positive genus.
 
 Let $\alpha$ and $\beta$ two (isotopy classes of) curves on $S_{g,n}$. The \textit{(geometric) intersection number of} $\alpha$ \textit{and} $\beta$, denoted by $i(\alpha,\beta)$, is defined as follows: $$i(\alpha,\beta) = \min \{|a \cap b|: a \in \alpha, b \in \beta\}.$$
 
 Let $\alpha$ and $\beta$ be two curves on $S_{g,n}$. In this work, we say that $\alpha$ is \textbf{disjoint} from $\beta$ if $i(\alpha,\beta) = 0$ \textbf{and} $\alpha \neq \beta$.
 
 The \textit{curve graph of} $S_{g,n}$, denoted by $\ccomp{S_{g,n}}$, is the abstract simplicial graph whose vertices are the isotopy classes of essential curves on $S_{g,n}$, and two vertices span an edge if the corresponding curves are disjoint. We often abuse notation and denote by $\ccomp{S_{g,n}}$ both the curve graph and its vertex set.
 
\subsection{Rigid expansions}\label{subsec1-1}
 In \cite{JHH2} we gave a general description for the concept of a rigid subgraph and rigid expansions of abstract simplicial graphs. Here however, we give the definitions focused on the case of the curve graph.
 
 Given a curve $\alpha$ on $S_{g,n}$, we define the set $\mathrm{adj}(\alpha)$ as the set of all curves on $S$ that are disjoint from $\alpha$. Note that $\alpha \notin \mathrm{adj}(\alpha)$.
 
 We say a curve $\alpha \in \ccomp{S_{g,n}}$ is \textit{uniquely determined by} $A \subset \ccomp{S_{g,n}}$, denoted by $\alpha = \langle A \rangle$, if $\alpha$ is the unique curve on $S_{g,n}$ that is disjoint from every element in $A$, i.e. $$\{\alpha\} = \bigcap_{\gamma \in A} \mathrm{adj}(\gamma).$$ See Figure \ref{fig:PrelimFig2} for examples.
 
 \begin{figure}[ht]
     \centering
     \includegraphics[height=4cm]{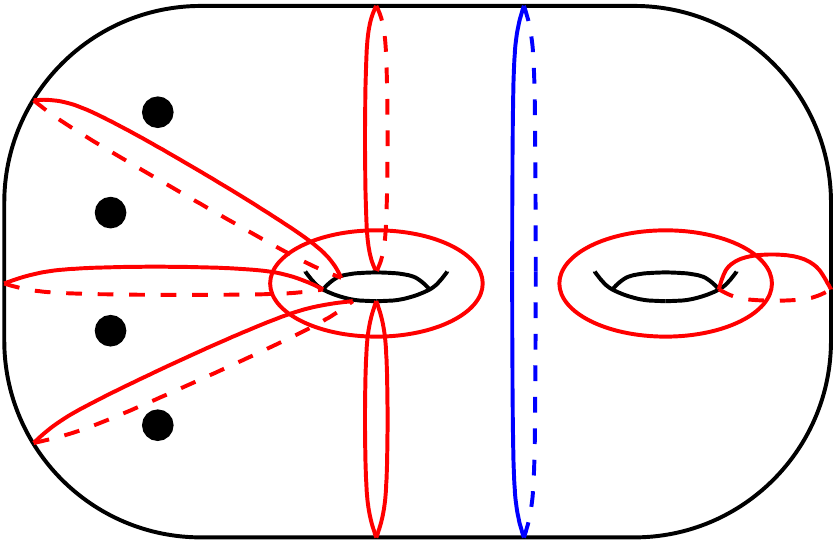}\hspace{1cm}
     \includegraphics[height=4cm]{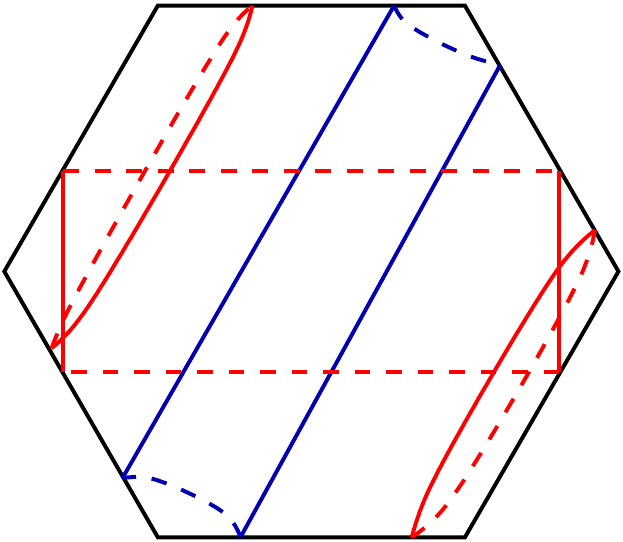}
     \caption{In both examples the blue curve is uniquely determined by the set of red curves.}
     \label{fig:PrelimFig2}
 \end{figure}
 
 Let $Y < \ccomp{S_{g,n}}$ be a induced subgraph of $\ccomp{S_{g,n}}$, and let $V(Y)$ denote its vertex set. We define the \textit{first rigid expansion of} $Y$, denoted by $Y^{1}$ as the induced subgraph of $\ccomp{S_{g,n}}$ whose vertex set if the following: $$V(Y^{1}) \ColonEqq V(Y) \cup \{\alpha \in \ccomp{S_{g,n}} : \alpha = \langle A \rangle,\text{ for some } A \subset V(Y)\}.$$
 
 We then define the $(n+1)$-th rigid expansion as $(Y^{n})^{1}$, and $Y^{\omega}$ as the induced subgraph of $\ccomp{S_{g,n}}$ induced by the following vertex set: $$V(Y^{\omega}) \ColonEqq \bigcup _{n \geq 1} V(Y^{n}).$$
 
 Recall that a graph morphism between simplicial graphs is a map between the corresponding vertex sets that also maps edges to edges.
 
 The following theorem is Theorem B in \cite{JHH2} for the special case of the curve graph, using the terminology here presented:
 
 \begin{Teo}[B in \cite{JHH2}]\label{ThmBJHH2}
  Let $Y < \ccomp{S_{g,n}}$ be an induced subgraph, and $\fun{\varphi}{Y^{\omega}}{\ccomp{S_{g,n}}}$ be a graph morphism such that there exists an automorphism of $\ccomp{S_{g,n}}$ (if $(g,n) \neq (1,2)$ it is a possibly orientation-reversing homeomorphism of $S_{g,n}$) $\fun{h}{\ccomp{S_{g,n}}}{\ccomp{S_{g,n}}}$ with $h|_{Y} = \varphi|_{Y}$. Then, $h|_{Y^{\omega}} = \varphi$.
 \end{Teo}
 
 Now, given a curve $\alpha$ on $S_{g,n}$, we define the star of $\alpha$, denoted by $\mathrm{star}(\alpha)$, as the subgraph whose vertex set is $\{\alpha\} \cup \mathrm{adj}(\alpha)$, and the edges are those from $\ccomp{S_{g,n}}$ that have $\alpha$ as one of its endpoints.
 
 A locally injective map $\fun{\varphi}{X}{\ccomp{S_{g,n}}}$ is a simplicial map whose restriction to $\mathrm{star}(\alpha)$ is injective for every $\alpha \in \ccomp{S_{g,n}}$.
 
 Now, a induced subgraph $X < \ccomp{S_{g,n}}$ is called \textit{rigid} if every locally injective map from $X$ to $\ccomp{S_{g,n}}$ is induced by an automorphism of $\ccomp{S_{g,n}}$ (again, if $(g,n) \neq (1,2)$ it becomes a homeomorphism of $S_{g,n}$).
 
 Note that not every supergraph of a rigid subgraph is rigid (see Proposition 3.2 in \cite{Ara2}). However, any rigid expansion of a rigid subgraph is rigid (see Proposition 3.5 in \cite{Ara2}).
 
 The following corollary is Corollary C in \cite{JHH2} for the special case of the curve graph, and it comes as a direct consequence of Theorem \ref{ThmBJHH2} and the definition of rigidity:
 
 \begin{Cor}[C in \cite{JHH2}]\label{CorCJHH2}
  Let $X < \ccomp{S_{g,n}}$ be a rigid subgraph, and $\fun{\varphi}{X^{\omega}}{\ccomp{S_{g,n}}}$ be an edge-preserving map such that $\varphi|_{X}$ is locally injective. Then $\varphi$ is the restriction to $X^{\omega}$ of an automorphism of $\ccomp{S_{g,n}}$ (if $(g,n) \neq (1,2)$, it is a homeomorphism of $S_{g,n}$).
 \end{Cor}
\section{Genus zero case}\label{sec2}
 In this section, we assume that $S = S_{0,n}$ with $n \geq 5$ (so that $\kappa(S) \geq 2$), unless otherwise stated.
 
 The structure of this section is as follows: In Section \ref{subsec2-1} we define $\Of$ which is the basis upon which we construct $\Yf{S}$; in Section \ref{subsec2-2} we define the sets of auxiliary curves $\Dfone$ and $\Dftwo$, and we also define $\Yf{S}$; in Section \ref{subsec2-3} we give the proof of Theorem \ref{TeoA} pending the proof of a key lemma (Lemma \ref{KeyLemag0}); in Sections \ref{subsec2-lc}, \ref{subsec2-4}, \ref{subsec2-5} and \ref{subsec2-6} we give the proof of the key lemma; finally, in Section \ref{subsec2-7} we recall from \cite{Ara2} the definition of $\Xf{S}$ and prove Theorem \ref{TeoB}.
\subsection{The basis of $\Yf{S}$}\label{subsec2-1} 
 For $k\geq 2$, let $O = \{\alpha_{0}, \ldots, \alpha_{k}\}$ be a set of outer curves. We say $O$ is an \textit{outer chain} if $i(\alpha_{i},\alpha_{j}) = 2$ if $|i-j|=1$ and $\alpha_{i}$ is disjoint from $\alpha_{j}$ otherwise. This notation comes as an analogue for a chain in higher genus. We say $O$ is a \textit{closed outer chain} if (counting the subindices modulo $k+1$) $i(\alpha_{i},\alpha_{j}) = 2$ if $|i-j|=1$ and $\alpha_{i}$ is disjoint from $\alpha_{j}$ otherwise. See Figure \ref{fig:Sec2-1Fig1} for examples. In both cases, we set the \textit{length} of $O$ as its cardinality.
 
 \begin{figure}[ht]
     \centering
     \includegraphics[height=4cm]{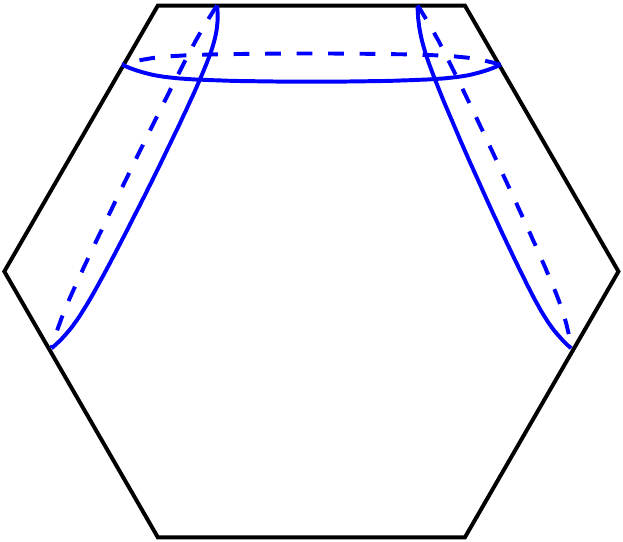}\hspace{1cm}
     \includegraphics[height=4cm]{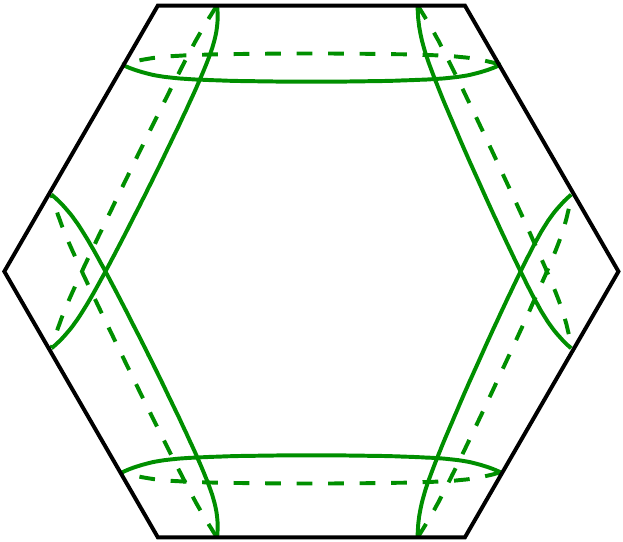}
     \caption{On the left, an example of an outer chain. On the right, an example of a closed outer chain.}
     \label{fig:Sec2-1Fig1}
 \end{figure}
 
 Note that a closed outer chain of maximal length has cardinality $n$, and its regular neighbourhood is homeomorphic to a $n$-punctured annulus.
 
 Let $\Of = \{\alpha_{0}, \ldots, \alpha_{n-1}\}$ be a closed outer chain of maximal length. While these curves seem to be a good candidate for $\Yf{S}$, they are not enough, i.e. they eventually stabilize. So, we have to add some auxiliary curves.
\subsection{Auxiliary curves}\label{subsec2-2}
 To define the auxiliary curves, we must first number the punctures of $S$ in such a way that $\alpha_{i} \in \Of$ bounds the $i$-th and $(i+1)$-th punctures (with indices modulo $n$). See Figure \ref{fig:Sec2-2Fig1} for an example.
 
 \begin{figure}[ht]
     \labellist
     \small \hair 2pt
     \pinlabel $\alpha_{i}$ [bl] at 80 165
     \pinlabel $i+1$ at 45 258
     \pinlabel $i$ [br] at 0 130
     \endlabellist
     \centering
     \includegraphics[height=4cm]{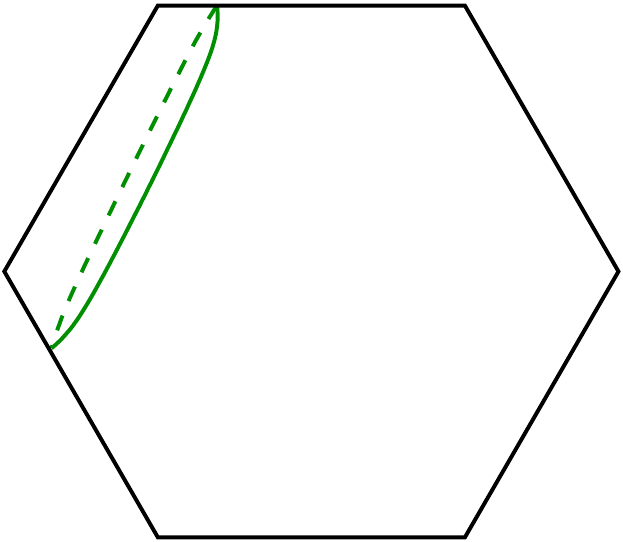}
     \caption{The $i$-th and $(i+1)$-th puncture along with the curve $\alpha_{i}$.}
     \label{fig:Sec2-2Fig1}
 \end{figure}
 
 Now, let $\alpha$ be an outer curve on $S$. We call the \textit{arc corresponding to} $\alpha$, to the unique arc ``joining'' the punctures bounded by $\alpha$. Similarly, let $a$ be an arc on $S$ with different endpoints; we call the \textit{outer curve corresponding to} $a$, to the unique outer curve that bounds the endpoints of $a$ and contains $a$ in its interior. See Figure \ref{fig:Sec2-2Fig2} for examples.
 
 \begin{figure}[ht]
     \centering
     \includegraphics[height=4cm]{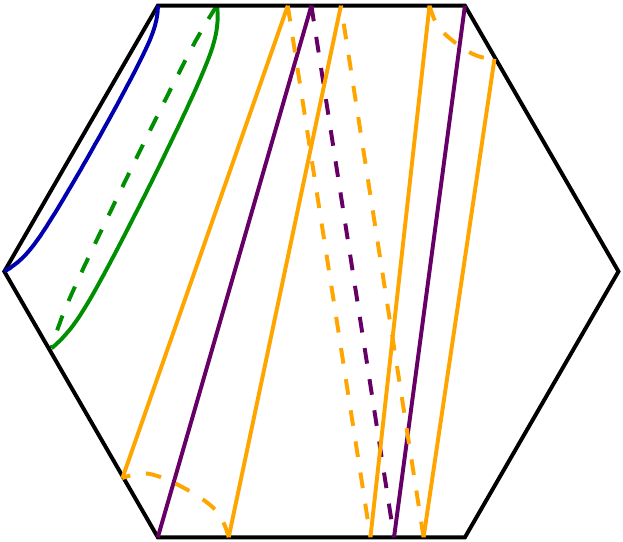}
     \caption{The blue and purple arcs have the green and orange curves as their (respective) corresponding outer curves.}
     \label{fig:Sec2-2Fig2}
 \end{figure}
 
 Note that there is a bijective correspondance between the set of all outer curves on $S$, and the set of all arcs on $S$ with differing endpoints.
 
 Let $A = \{a_{i}\}$ be the set of arcs $a_{i}$ corresponding to $\alpha_{i} \in \Of$. Let also $H_{1}$ and $H_{2}$ be two identical $n$-gons with numbered vertices such that $S$ can be identified with the result of gluing $H_{1}$ and $H_{2}$ together by their sides (and eliminating the vertices), and the labeling of the vertices in $H_{1}$ and $H_{2}$ coincides with the labeling of the punctures on $S$. Moreover, we particularly choose the identification, so that the (glued together) sides of $H_{1}$ and $H_{2}$ are identified with the elements of $A$. See Figure \ref{Sec2-2Fig3}.
 \vspace{2mm}
 \begin{figure}[ht]
     \labellist
     \pinlabel $H_{1}$ at 146 133
     \pinlabel $H_{2}$ at 518 133
     \pinlabel $1$ [br] at 70 255
     \pinlabel $2$ [bl] at 229 255
     \pinlabel $3$ [l] at 295 130
     \pinlabel $4$ [tl] at 225 5
     \pinlabel $5$ [tr] at 72 5
     \pinlabel $6$ [r] at 5 130
     \pinlabel $1$ [br] at 435 255
     \pinlabel $2$ [bl] at 594 255
     \pinlabel $3$ [l] at  665 130
     \pinlabel $4$ [tl] at 590 5
     \pinlabel $5$ [tr] at 437 5
     \pinlabel $6$ [r] at 370 130
     \pinlabel $\rightsquigarrow$ at 730 133
     \pinlabel $\alpha_{1}$ at 924 205
     \pinlabel $\alpha_{6}$ [bl] at 852 160
     \endlabellist
     \centering
     \includegraphics[height=35mm]{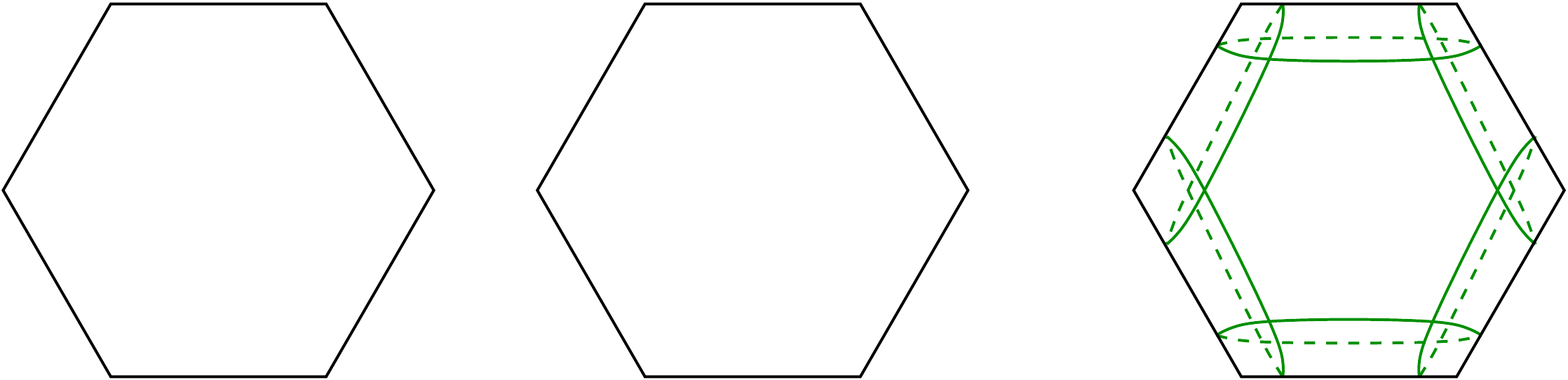}
     \caption{The $n$-gons $H_{1}$ and $H_{2}$ (left), glued together to obtain the surface $S$ (right).}
     \label{Sec2-2Fig3}
 \end{figure}
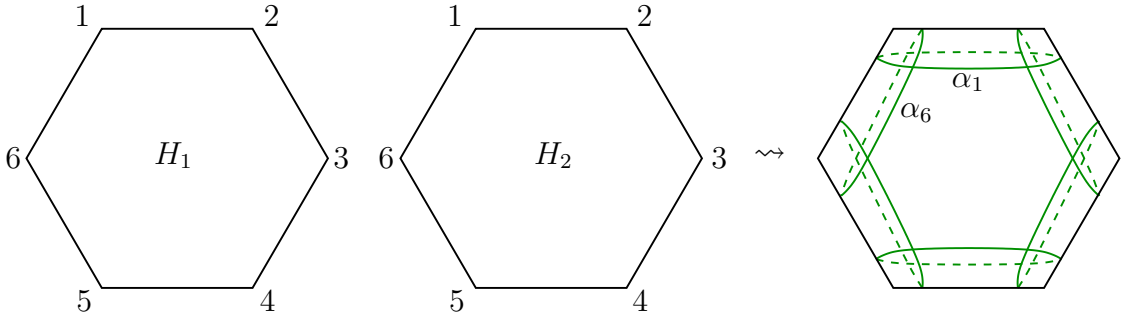
 
 For $0 \leq i < j <n$, let $b_{i,j}$ be the arc on $S$ joining the $i$-th and $j$-th punctures of $S$ that is identified with the diagonal on $H_{1}$ joining the $i$-th and $j$-th vertices. Similarly, for $0 \leq i <j <n$, let $c_{i,j}$ be the arc on $S$ joining the $i$-th and $j$-th punctures of $S$ that is identified with the diagonal on $H_{2}$ joining the $i$-th and $j$-th vertices. See Figure \ref{fig:bBetacGamma} for an example.
 
 \begin{figure}[ht]
     \labellist
     \pinlabel $b_{1,5}$ [l] at 80 130
     \pinlabel $c_{1,6}$ [b] at 517 135
     \pinlabel $\beta_{1,5}$ [bl] at 885 200
     \pinlabel $\gamma_{1,6}$ [bl] at 760 45
     \endlabellist
     \centering
     \includegraphics[height=35mm]{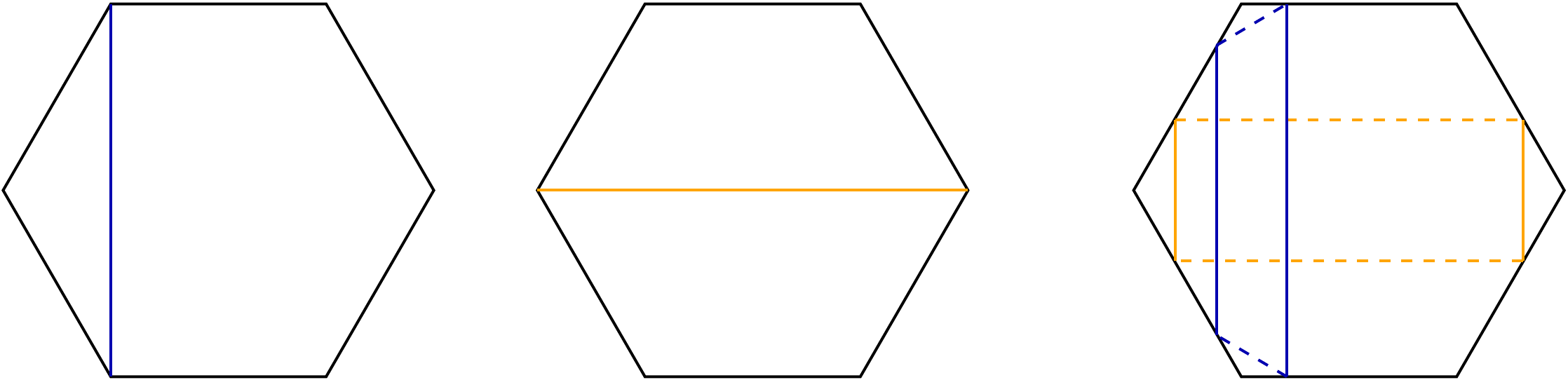}
     \caption{The arcs $b_{1,5}$ and $c_{1,6}$ that induce the curves $\beta_{1,5}$ and $\gamma_{1,6}$ respectively.}
     \label{fig:bBetacGamma}
 \end{figure}
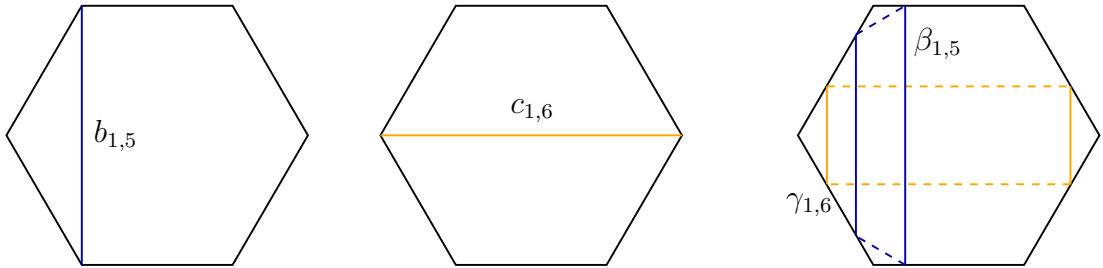
 
 Then, let $\Dfone = \{\beta_{i,j}\}$ be the set of outer curves $\beta_{i,j}$ corresponding to the arc $b_{i,j}$, and let $\Dftwo = \{\gamma_{i,j}\}$ be the set of outer curves $\gamma_{i,j}$ corresponding to the arc $c_{i,j}$.
 
 Finally, we define: $$\Yf{S} \ColonEqq \Of \cup \Dfone \cup \Dftwo.$$
\subsection{Proof of Theorem \ref{TeoA}}\label{subsec2-3} 
 For the proof of Theorem \ref{TeoA}, we first recall some definitions and properties around outer curves on $S$ and half-twists.
 
 Let $\alpha$ be a curve on $S$. We denote by $\eta_{\alpha}$ the (left) half-twist along $\alpha$ (see Figure \ref{fig:ExaHalfTwist}); recall that this is defined if and only if $\alpha$ is an outer curve. Also recall that there is exactly one half-twist along $\alpha$ if $S \ncong S_{0,4}$.
 
 \begin{figure}[ht]
     \labellist
     \pinlabel $\beta$ [bl] at 135 60
     \pinlabel $\alpha$ [bl] at 165 210
     \pinlabel $\mapsto$ at 360 130
     \pinlabel $\eta_{\alpha}(\beta)$ [bl] at 550 160
     \endlabellist
     \centering
     \includegraphics[height=4cm]{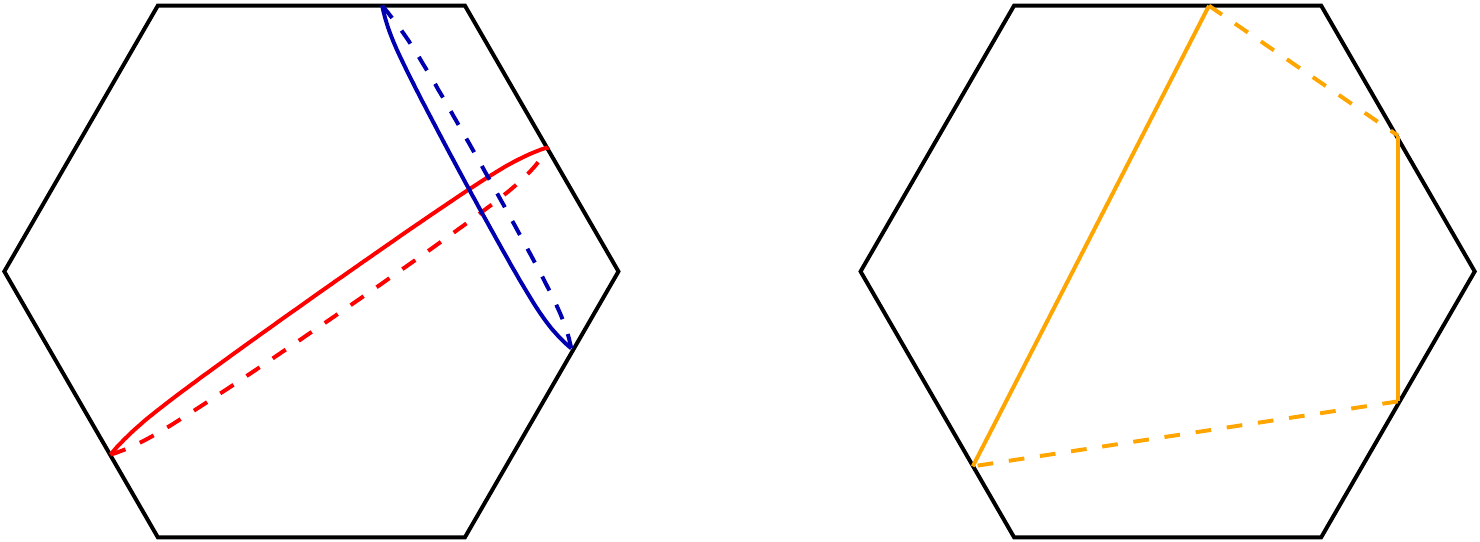}
     \caption{On the left, the outer curve $\alpha$ and the curve $\beta$; on the right, the curve $\eta_{\alpha}(\beta)$.}
     \label{fig:ExaHalfTwist}
 \end{figure}
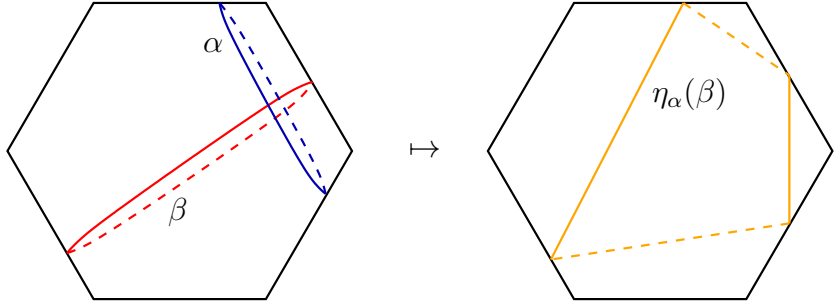
 
\begin{Rem}\label{RemSym}
 Note that if $\alpha$ and $\beta$ are two outer curves that intersect twice, then we have that $\eta_{\alpha}(\beta) = \eta_{\beta}^{-1}(\alpha)$. This fact can also be found in Proposition 3.9 of \cite{Ara2}.
\end{Rem}
 
 Let $O$ be a set of outer curves and $A$ be a set of curves on $S$. We denote by $\eta_{O}(A)$ the following set $$\eta_{O}(A) \ColonEqq \bigcup_{\veps \in O} \eta_{\veps}(A).$$
 
 Let $\Gf \ColonEqq \{\eta_{\alpha_{i}}^{\pm 1}: \alpha_{i} \in \Of\}$. It is a well-known fact (see Theorem 4.9 in \cite{FarbMar}) that $\Gf$ is a symmetric generating set of $\Mod{S}$. With this, if $A$ is a set of curves, we denote by $\Gf \cdot A$ the following set $\Gf \cdot A \ColonEqq \eta_{\Of}^{\pm 1}(A)$.
 
 Now, similarly to the proofs in \cite{JHH1}, we state a key lemma for the proof of Theorem \ref{TeoA}. The proof of this lemma is discussed and given \textit{after} the proof of Theorem \ref{TeoA} (see Subsections \ref{subsec2-4} and \ref{subsec2-5}).
 
 \begin{Lema}\label{KeyLemag0}
  Let $\Gf$ be as above. Then $\Gf \cdot \Yf{S} \subset \Yf{S}^{3}$.
 \end{Lema}
 
 \begin{proof}[\textbf{Proof of Theorem \ref{TeoA}:}]
  We divide this proof into two parts: first we prove that all outer curves on $S$ are elements in $\Yf{S}^{\omega}$; then we argue that any other curve on $S$ is uniquely determined by a finite number of outer curves and therefore we obtain the desired result.
  
  \textit{First part:} Let $\beta$ be an outer curve on $S$. Since all outer curves on $S$ have the same topological type, then there exists a element $\alpha_{i} \in \Of$ and a homeomorphism $h \in \Mod{S}$ such that $h(\alpha_{i}) = \beta$. Given that $\Gf$ is a symmetric generating set of $\Mod{S}$, then there exist elements $f_{1}, \ldots, f_{k} \in \Gf$ and powers $n_{1}, \ldots, n_{k} \geq 0$ such that $$f_{1}^{n_{1}} \circ \cdots \circ f_{k}^{n_{k}} (\alpha_{i}) = \beta.$$ By an iterative use of Lemma \ref{KeyLemag0}, this implies that $\beta \in \Yf{S}^{3(n_{1} + \cdots + n_{k})}$. Thus, every outer curve on $S$ is an element of $\Yf{S}^{\omega}$.
  
  \textit{Second part:} Let $\gamma$ be a curve on $S$ that is not an outer curve. Since there are only finitely many topological types of such curves, which depend on how they separate the punctures of $S$, up to homeomorphism $\gamma$ is as in the example on Figure \ref{fig:DetSepCurves-g0}. Then, there exists two (finite) outer chains, namely $A$ and $B$, that uniquely determine $\gamma$ (see Figure \ref{fig:DetSepCurves-g0} for an example). Given that by the first part of this proof $A \cup B \subset \Yf{S}^{k}$ for some $k \geq 0$, we have that $\gamma \in \Yf{S}^{k+1}$. Therefore, every curve on $S$ that is not an outer curve, is also an element of $\Yf{S}^{\omega}$.
  
  \begin{figure}[ht]
      \centering
      \includegraphics[height=4cm]{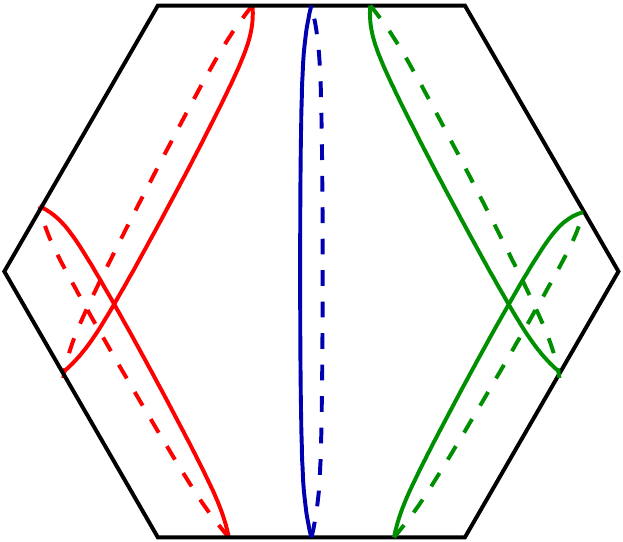}
      \caption{The set $A$ (red curves) and the set $B$ (green curves) uniquely determining the curve $\gamma$ (in blue).}
      \label{fig:DetSepCurves-g0}
  \end{figure}
 \end{proof}
 
 As for the proof of Lemma \ref{KeyLemag0}, we first prove the lemma for the complexity $2$ case in Subsection \ref{subsec2-lc}; then we prove them for complexity at least $3$ in more generality. The reason for this division is that in the complexity $2$ we ``do not have enough space'' for the general method to work.
 
 For the complexity at least $3$ case, we divide the proof into the following claims:\\
 \textbf{Claim 1:} $\eta_{\Of}^{\pm 1}(\Of) \subset \Yf{S}$\newline
 \textbf{Claim 2:} $\eta_{\Of}^{\pm 1}(\Dfone) \subset \Yf{S}^{2}$\newline
 \textbf{Claim 3:} $\eta_{\Of}^{\pm 1}(\Dftwo) \subset \Yf{S}^{3}$
\subsection{Complexity $2$ case}\label{subsec2-lc}
 Let $S$ be a $5$-punctured sphere, $\alpha \in \Of$ and $\delta \in \Yf{S}$. Note that if $\alpha$ and $\delta$ are disjoint, we have that $\eta_{\alpha}^{\pm 1}(\delta) = \delta$. Thus, we focus solely on the cases when they intersect.
 
 We can easily check that $$\eta_{\alpha_{i+1}}(\alpha_{i}) = \gamma_{i,i+2} \in \Yf{S},$$ $$\eta_{\alpha_{i-1}}(\alpha_{i}) = \beta_{i,i+2} \in \Yf{S}.$$
 
 This and Remark \ref{RemSym} imply that $$\eta_{\alpha_{i+1}}^{-1}(\alpha_{i}) = \eta_{\alpha_{i}}(\alpha_{i+1}) \in \Yf{S},$$ $$\eta_{\alpha_{i-1}}^{-1}(\alpha_{i}) = \eta_{\alpha_{i}}(\alpha_{i-1}) \in \Yf{S}.$$
 
 Hence, $\eta_{\Of}^{\pm 1}(\Of) \subset \Yf{S}$.
 
 Since $S$ is a $5$-punctured sphere, then every $\beta \in \Dfone$ is of the form either $\beta_{i,i+2}$ or $\beta_{i,i+3}$ for some $0 \leq i < n$ (with the indices modulo $5$).

 If $\beta = \beta_{i,i+3}$ we have that $$\eta_{\alpha_{i-1}}(\beta_{i,i+3}) = \alpha_{i-2} \in \Yf{S},$$ $$\eta_{\alpha_{i}}(\beta_{i,i+3}) = \langle \{\beta_{i,i+2}, \gamma_{i-1,i+2}\} \rangle \in \Yf{S}^{1}.$$ See Figure \ref{fig:Sec2-4Fig1} for an example.

 \begin{figure}[ht]
     \centering
     \includegraphics[height=4cm]{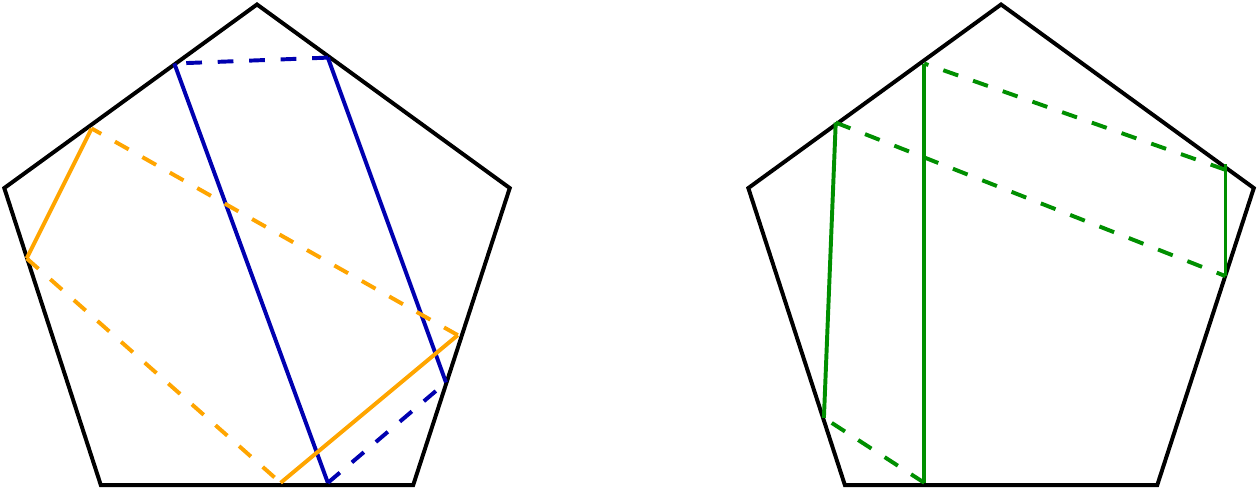}
     \caption{On the left, the curves $\gamma_{i-1,i+2}$ (orange) and $\beta_{i,i+2}$; on the right the curve they uniquely determine, $\eta_{\alpha_{i}}(\beta_{i,i+3})$, in green.}
     \label{fig:Sec2-4Fig1}
 \end{figure}
 
 If $\beta = \beta_{i,i+2}$, on one hand we have that $$\eta_{\alpha_{i-1}}(\beta_{i,i+2}) = \beta_{i-1,i+2} \in \Yf{S}.$$ On the other hand, for $\eta_{\alpha_{i}}(\beta_{i,i+2})$ we need an auxiliary curve, which we define as follows (using indices modulo $5$): $$\delta = \langle \{\beta_{i-1,i+2}, \gamma_{i-1,i+1}\} \rangle.$$ See Figure \ref{fig:Sec2-4Fig2}.
 \begin{figure}[ht]
     \centering
     \includegraphics[height=4cm]{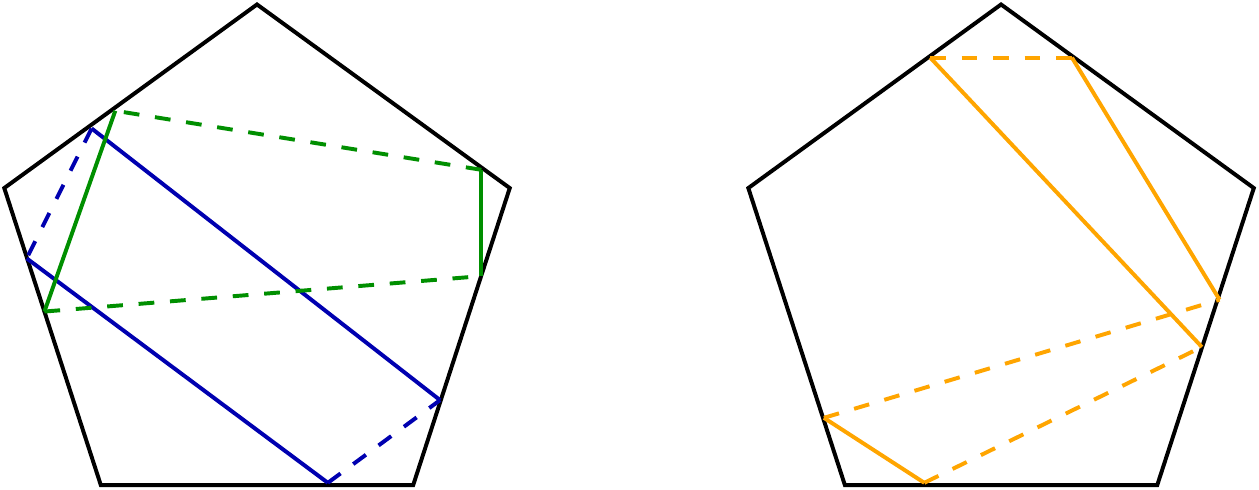}
     \caption{On the left, the curves $\beta_{i-1,i+2}$ (blue) and $\gamma_{i-1,i+1}$ (green); on the right the auxiliary curve $\delta$ (orange) which they uniquely determine.}
     \label{fig:Sec2-4Fig2}
 \end{figure}
 Then, we have that $$\eta_{\alpha_{i}}(\beta_{i,i+2}) = \langle\{\alpha_{i-2}, \delta\}\rangle \in \Yf{S}^{2}.$$ See Figure \ref{fig:Sec2-4Fig3}.
 
 \begin{figure}[ht]
     \centering
     \includegraphics[height=4cm]{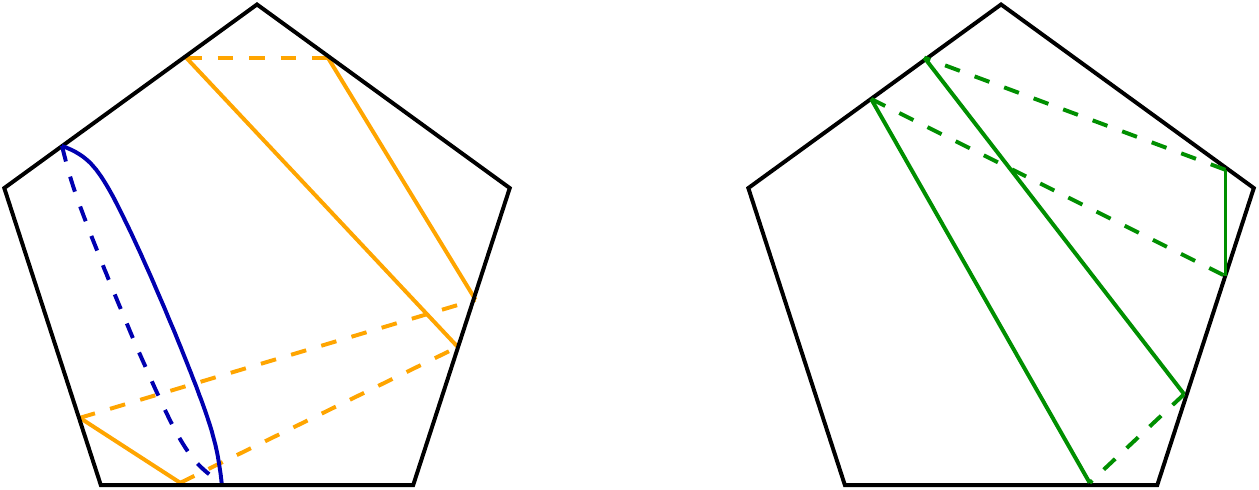}
     \caption{On the left, the curves $\alpha_{i-2}$ (blue) and the auxiliary curve $\delta$ (orange); on the right, $\eta_{\alpha_{i}}(\beta_{i,i+2})$, which they uniquely determine.}
     \label{fig:Sec2-4Fig3}
 \end{figure}
 
 Given that the rest of the cases (with $\alpha_{i+1}$, $\alpha_{i+2}$ and $\alpha_{i+3}$) are analogous to the cases above, we get that $\eta_{\Of}(\Dfone) \subset \Yf{S}^{2}$.
 
 Again, since $S$ is a $5$-punctured sphere, then every $\gamma \in \Dftwo$ is of the form either $\gamma_{i,i+2}$ or $\gamma_{i,i+3}$ for some $0 \leq i < n$ (with the indices modulo $5$).
 
 If $\gamma = \gamma_{i,i+2}$, we have that $$\eta_{\alpha_{i}}(\gamma_{i,i+2}) = \alpha_{i+1} \in \Yf{S},$$ and that the case for $\eta_{\alpha_{i-1}}(\gamma_{i,i+2})$ is analogous to the case for $\eta_{\alpha_{i}}(\beta_{i,i+3})$; hence $$\eta_{\alpha_{i-1}}(\gamma_{i,i+2}) \in \Yf{S}^{1}.$$
 
 If $\gamma = \gamma_{i,i+3}$, we have that $$\eta_{\alpha_{i}}(\gamma_{i,i+3}) = \gamma_{i+1,i+3} \in \Yf{S},$$ and that the case for $\eta_{\alpha_{i-1}}(\gamma_{i,i+3})$ is analogous to the case for $\eta_{\alpha_{i}}(\beta_{i,i+2})$; hence $$\eta_{\alpha_{i-1}}(\gamma_{i,i+3}) \in \Yf{S}^{2}.$$
 
 Analogously to the cases for $\Dfone$, the cases above imply that $\eta_{\Of}(\Dftwo) \subset \Yf{S}^{2}$.
 
 Finally, the cases for $\eta_{\Of}^{-1}(\Dfone)$ and $\eta_{\Of}^{-1}(\Dftwo)$ mirror the cases for $\eta_{\Of}(\Dftwo)$ and $\eta_{\Of}(\Dfone)$, respectively. Therefore, $\eta_{\Of}^{\pm 1}(\Yf{S}) \subset \Yf{S}^{2}$.
 
\subsection{Proof of Claim 1: $\eta_{\Of}^{\pm 1}(\Of) \subset \Yf{S}$}\label{subsec2-4} 
 Let $\alpha_{i}, \alpha_{j} \in \Of$. Note that if $\alpha_{i}$ is disjoint from $\alpha_{j}$, then $\eta_{\alpha_{i}}^{\pm 1}(\alpha_{j}) = \alpha_{j} \in \Yf{S}$. So, we only focus on the cases where $\alpha_{i}$ intersects $\alpha_{j}$, i.e. $|i-j| = 1$.
 
 For the proof of Claim 1, let $0 \leq i < n$, then we have that (using indices modulo $n$): $$\eta_{\alpha_{i+1}}(\alpha_{i}) = \gamma_{i,i+2}.$$
 Also, we have $$\eta_{\alpha_{i+1}}^{-1}(\alpha_{i}) = \beta_{i,i+2}.$$
 
 With this and Remark \ref{RemSym} we have that $\eta_{\Of}^{\pm 1}(\Of) \subset \Yf{S}$.
\subsection{Proof of Claim 2: $\eta_{\Of}^{\pm 1}(\Dfone) \subset \Yf{S}^{2}$}\label{subsec2-5}
 Let $\beta_{i,j} \in \Dfone$ and $\alpha_{k} \in \Of$. Note that if $\alpha_{k}$ is disjoint from $\beta_{i,j}$ then $\eta_{\alpha_{k}}^{\pm 1}(\beta_{i,j}) = \beta_{i,j} \in \Yf{S}$. So, we only focus on the cases where $\alpha_{k}$ intersects $\beta_{i,j}$.
 
 Now, we divide the proof of this claim into three cases according to the triple $(i,j,k)$:
 
 \textit{First case:} We have the following subcases:
 \begin{enumerate}
     \item[\textit{(i)}] Suppose $j \neq i +2$ and $i \neq j+2$ (both modulo $n$).
     \item[\textit{(ii)}] Suppose $j = i+2$ and $k \neq i, i+1$.
     \item[\textit{(iii)}] Suppose $i = j+2$ and $k \neq j,j+1$.
 \end{enumerate}
 
 Then, there exists non-empty sets $O \subset \Of$ and $D \subset \Dftwo$ such that $\beta_{i,j} = \langle O \cup D \rangle$ and $\alpha_{k}$ is disjoint from every element in $D$. Thus, $\eta_{\alpha_{k}}^{\pm 1}(\beta_{i,j}) = \langle \eta_{\alpha_{k}}^{\pm 1}(O) \cup \eta_{\alpha_{k}}^{\pm 1}(D) \rangle = \langle \eta_{\alpha_{k}}^{\pm 1}(O) \cup D\rangle$. From this, by Claim 1, we have that $\eta_{\alpha_{k}}^{\pm 1}(\beta_{i,j}) \in \Yf{S}^{1}$.
 
 The existence of these sets for each subcase is inferred from the following examples (see Figure \ref{fig:Sec2-6Fig1y2}):
 
 \begin{enumerate}
  \item[\textit{(i)}] If $j \neq i+2$ and $i \neq j+2$ (both modulo $n$), and if $k=i-1$ or $k = i$, we set $O = \{\alpha_{i+1}, \ldots, \alpha_{j-2}\} \cup \{\alpha_{j+1}, \ldots, \alpha_{i-2}\}$ and $D = \{\gamma_{j-1,j+1}\}$. The case for $k = j-1$ or $k = j$ is analogous.
  \item[\textit{(ii)}] If $j = i+2$ and $k \neq i, i+1$, we set $O = \{\alpha_{j+1}, \ldots, \alpha_{i-2}\}$ and $D = \{\gamma_{i-1,i+1}\}$.
  \item[\textit{(iii)}] The case for $i = j+2$ and $k \neq j,j+1$ is analogous to the subcase \textit{(ii)}.
 \end{enumerate}
 
 \begin{figure}[ht]
     \labellist
     \pinlabel $i$ [bl] at 230 250
     \pinlabel $k$ [br] at 70 250
     \pinlabel $j$ [r] at 70 0
     \endlabellist
     \centering
     \includegraphics[height=4cm]{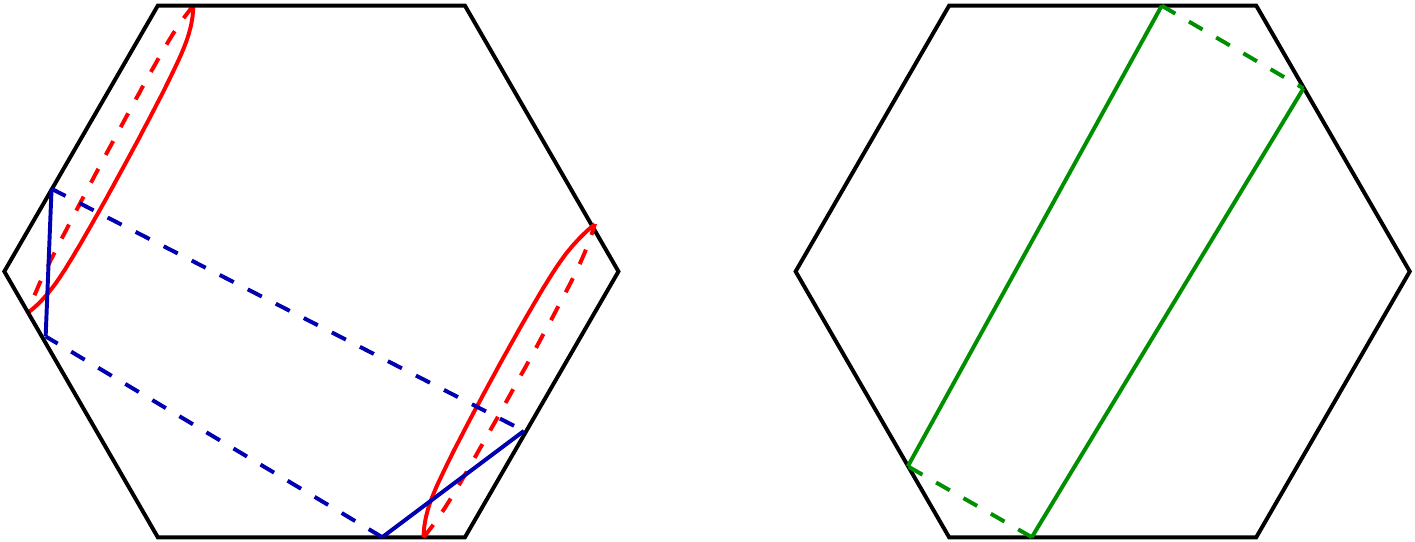}\\[5mm]
     \labellist
     \pinlabel $i$ [bl] at 230 250
     \pinlabel $j$ [l] at 225 0
     \endlabellist
     \includegraphics[height=4cm]{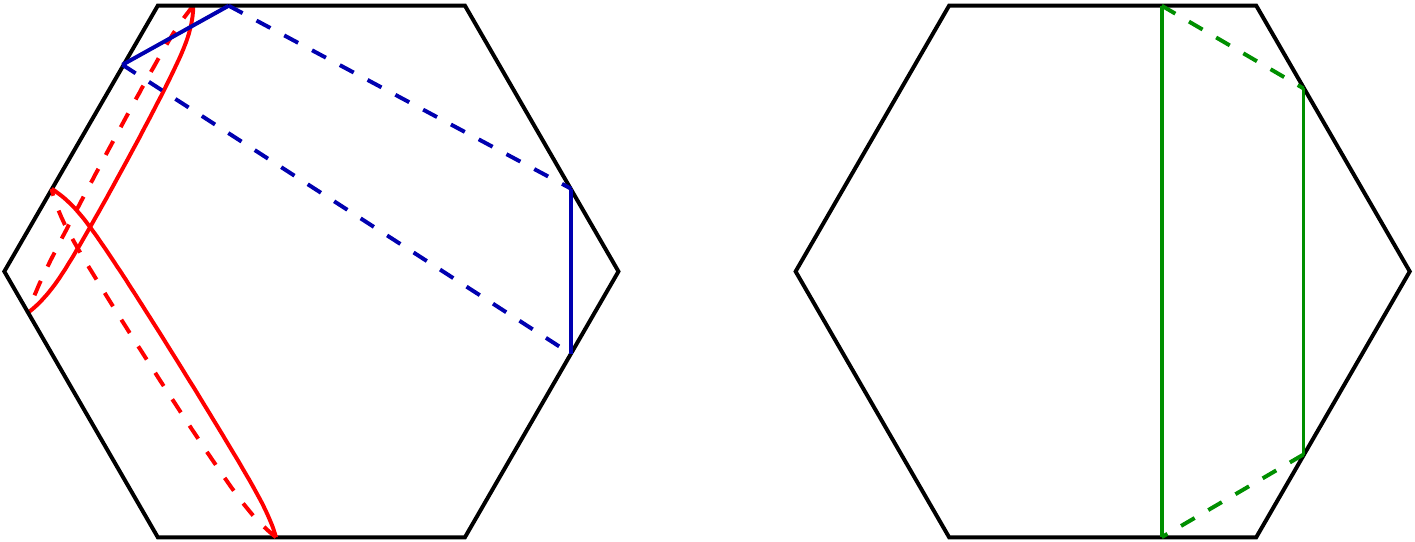}
     \caption{On top, the curves corresponding to the subcase $(i)$ and the curve $\beta_{i,j}$; on the bottom, the curves corresponding to the subcase $(ii)$. On the right, the curves $\beta_{i,j}$ they uniquely determine.}
     \label{fig:Sec2-6Fig1y2}
 \end{figure}
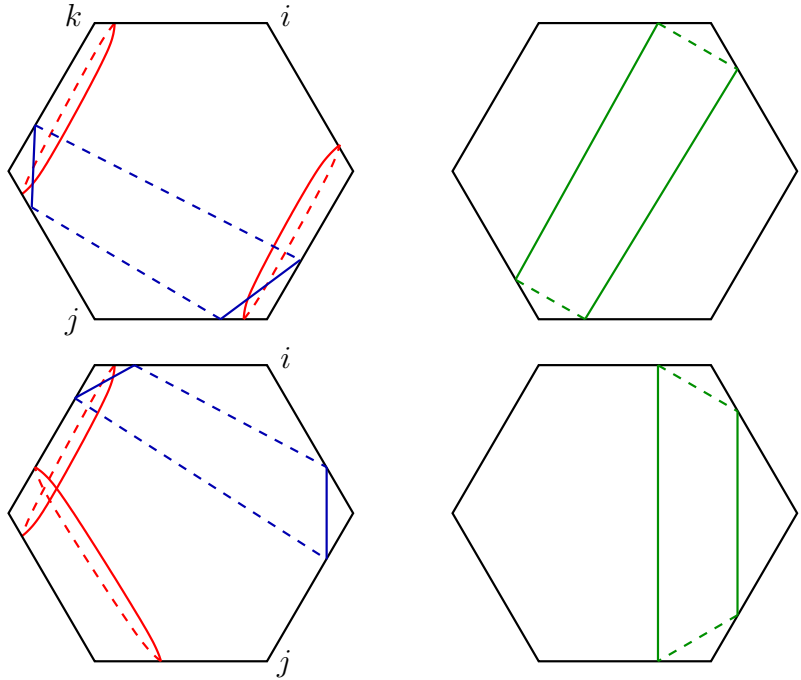
 
 \textit{Second case:} Suppose $j = i+2$ and $k = i$ (the case $i = j+2$ and $k =j$ is completely analogous). Then $\eta_{\alpha_{i}}^{-1}(\beta_{i,i+2}) = \alpha_{i+1}$.
 
 For $\eta_{\alpha_{i}}(\beta_{i,i+2})$ we need an auxiliary curve, which we define as follows (recall the indices are modulo $n$): $$\delta \ColonEqq \langle \{\alpha_{i+4}, \ldots \alpha_{i-2}\} \cup \{\beta_{i+2,i+4}\} \cup \{\gamma_{i-1,i+1}\}\rangle.$$ See Figure \ref{fig:Sec2-6Fig3y4} for an example.
 
 Then, we have that $$\eta_{\alpha_{i}}(\beta_{i,i+2}) = \langle \{\alpha_{i+3}, \ldots, \alpha_{i-2}\} \cup \{\delta\}\rangle \in \Yf{S}^{2}.$$ See Figure \ref{fig:Sec2-6Fig3y4} for an example.
 
 \begin{figure}[ht]
     \labellist
     \pinlabel $i=k$ [bl] at 230 250
     \pinlabel $j$ [l] at 225 5
     \pinlabel $\delta$ [bl] at 625 220
     \endlabellist
     \centering
     \includegraphics[height=4cm]{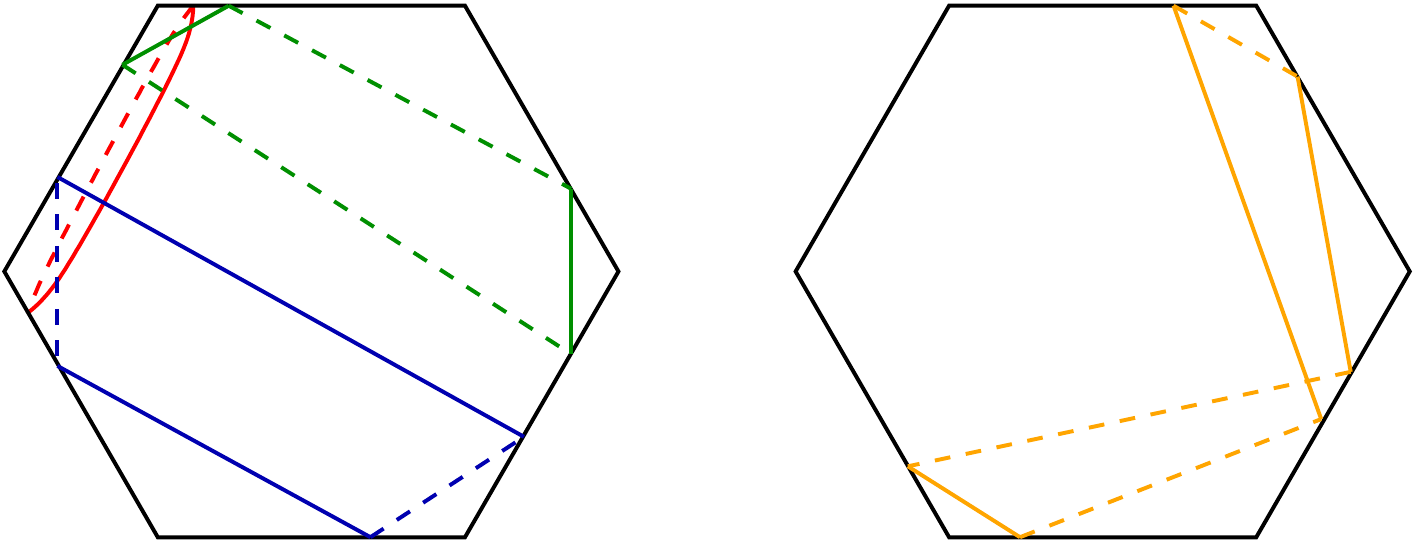}\\[5mm]
     \labellist
     \pinlabel $i$ [bl] at 230 245
     \pinlabel {$\eta_{\alpha_{i}}(\beta_{i,i+2})$} [l] at 660 170
     \endlabellist
     \includegraphics[height=4cm]{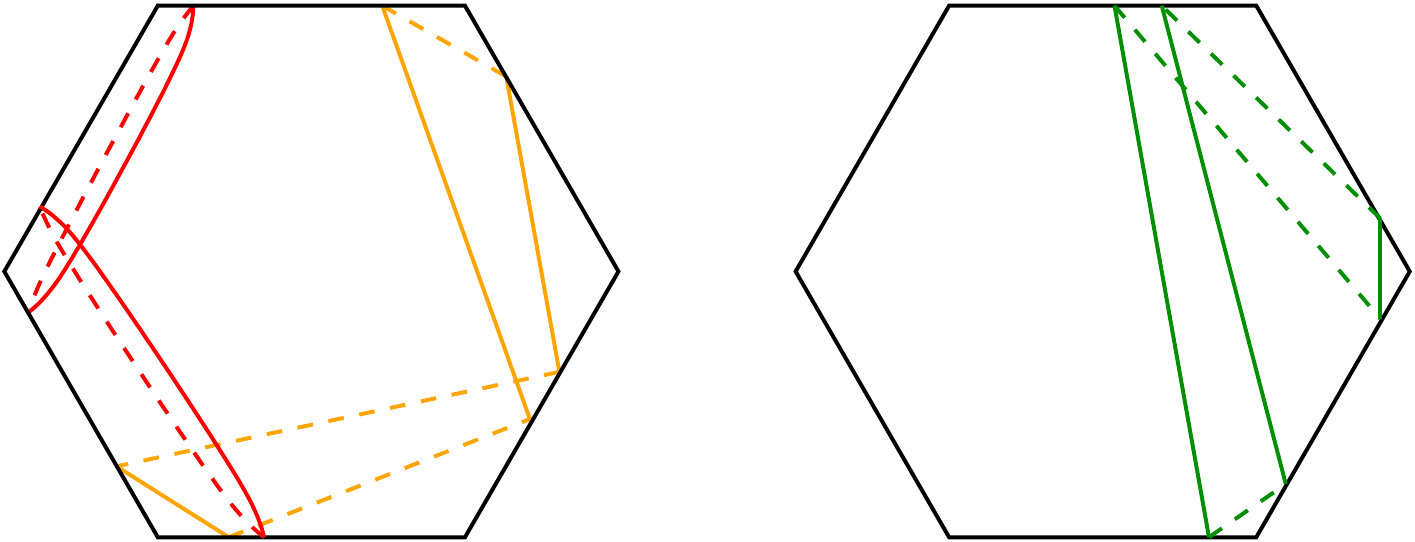}
     \caption{On top, the curves needed to uniquely determine the auxiliary curve $\delta$. On the bottom, the curves needed to uniquely determine $\eta_{\alpha_{i}}(\beta_{i,i+2})$ in the second case.}
     \label{fig:Sec2-6Fig3y4}
 \end{figure}
 
 \textit{Third case:} Suppose $j = i+2$ and $k = i+1$ (the case $i = j+2$ and $k =j+1$ is completely analogous). Then $\eta_{\alpha_{i+1}}(\beta_{i,i+2}) = \alpha_{i}$.
 
 Now, for $\eta_{\alpha_{i+1}}^{-1}(\beta_{i,i+2})$ we need again an auxiliary curve, which we define as follows (recall the indices are modulo $n$): $$\delta \ColonEqq \langle \{\alpha_{i+3}, \ldots, \alpha_{i-3}\} \cup \{\beta_{i-2,i}\} \cup \{\gamma_{i+1,i+3}\}\rangle.$$ See Figure \ref{fig:Sec2-6Fig5y6} for an example.
 
 Then, we have that $$\eta_{\alpha_{i+1}}^{-1}(\beta_{i,i+2}) = \langle \{\alpha_{i+3}, \ldots, \alpha_{i-2}\} \cup \{\delta\}\rangle \in \Yf{S}^{2}.$$ See Figure \ref{fig:Sec2-6Fig5y6} for an example.
 
 \begin{figure}[ht]
     \labellist
     \pinlabel $i$ [l] at 300 130
     \pinlabel $\delta$ [bl] at 625 220
     \endlabellist
     \centering
     \includegraphics[height=4cm]{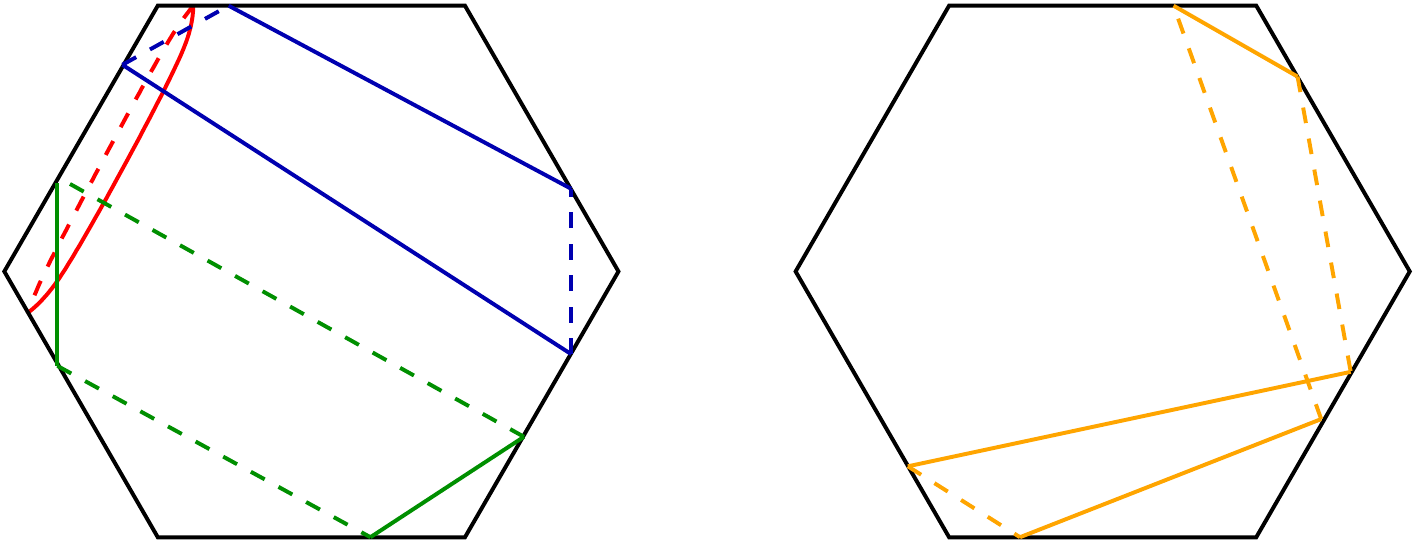}\\[5mm]
     \labellist
     \pinlabel $i$ [l] at 300 130
     \pinlabel {$\eta_{\alpha_{i+1}}^{-1}(\beta_{i,i+2})$} [l] at 660 170
     \endlabellist
     \includegraphics[height=4cm]{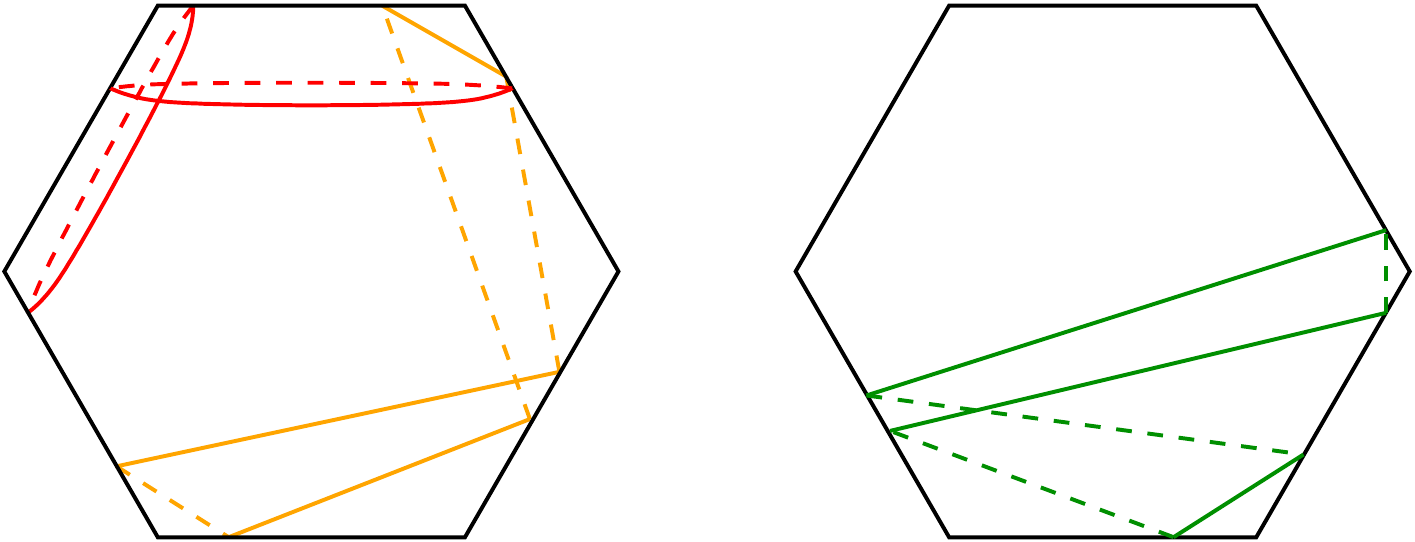}
     \caption{On top, the curves needed to uniquely determine the auxiliary curve $\delta$. On the bottom, the curves needed to uniquely determine $\eta_{\alpha_{i+1}}^{-1}(\beta_{i,i+2})$ in the second case.}
     \label{fig:Sec2-6Fig5y6}
 \end{figure}
 
 Thus, $\eta_{\Of}(\Dfone) \subset \Yf{S}^{2}$.
\subsection{Proof of Claim 3: $\eta_{\Of}^{\pm 1}(\Dftwo) \subset \Yf{S}^{3}$}\label{subsec2-6} 
 The proof of this claim follows from Claim 1 and Claim 2: If $\gamma \in \Dftwo$, then there exist sets $O \subset \Of$ and $D \subset \Dfone$ such that $\gamma = \langle O \cup D \rangle$. Thus, for all $\alpha_{i} \in \Of$, $\eta_{\alpha_{i}}^{\pm 1}(\gamma) = \langle \eta_{\alpha_{i}}^{\pm 1}(O) \cup \eta_{\alpha_{i}}^{\pm 1}(D)\rangle$. By Claims 1 and 2, this implies that $\eta_{\alpha_{i}}(\gamma) \in \Yf{S}^{3}$.
 
 The existence of these sets can be easily inferred from the following examples (see Figure \ref{fig:Sec2-7Fig1y2}):
 
 \begin{itemize}
  \item If $j \neq i+2$ and $i \neq j+2$ (both modulo $n$), we set $O = \{\alpha_{i+1}, \ldots, \alpha_{j-2}\} \cup \{\alpha_{j+1}, \ldots, \alpha_{i-2}\}$ and $D = \{\beta_{i+1,j+1}\}$.
  \item If $j = i+2$ (modulo $n$), we set $O = \{\alpha_{j+1}, \ldots, \alpha_{i-2}\}$ and $D = \{\beta_{i+1,j+1}\}$.
 \end{itemize}
 
 \begin{figure}[ht]
     \centering
     \labellist
     \pinlabel $i$ [bl] at 225 250
     \pinlabel $j$ [r] at 75 0
     \endlabellist
     \includegraphics[height=4cm]{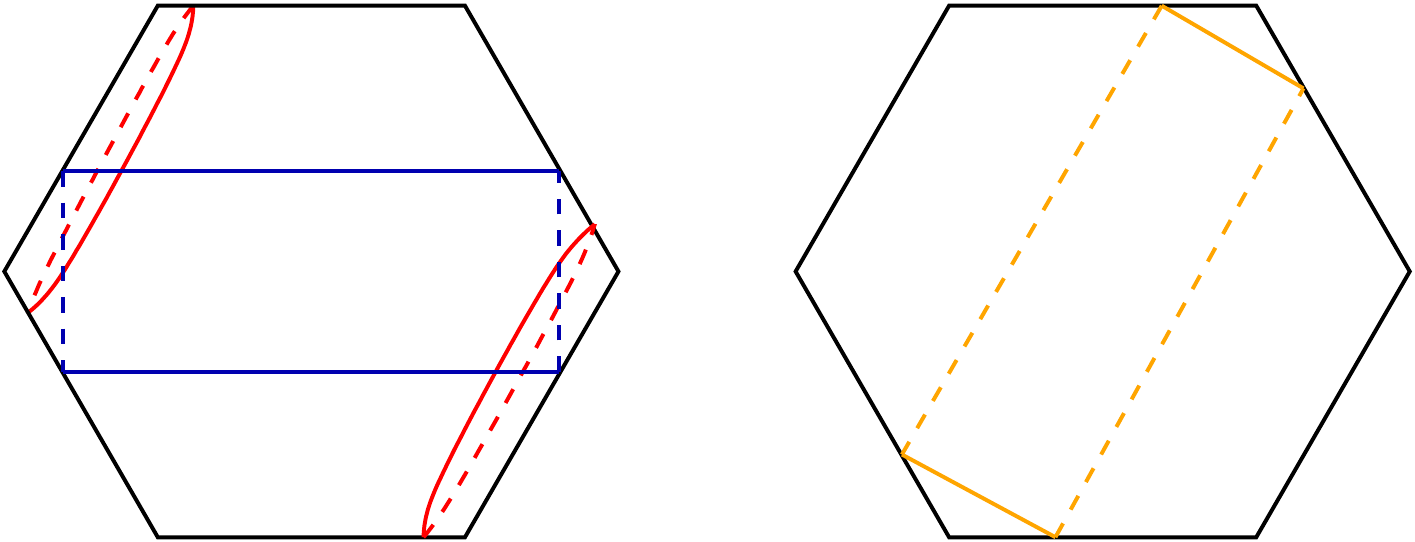}\\[5mm]
     \labellist
     \pinlabel $i$ [bl] at 225 250
     \pinlabel $j$ [l] at 220 0
     \endlabellist
     \includegraphics[height=4cm]{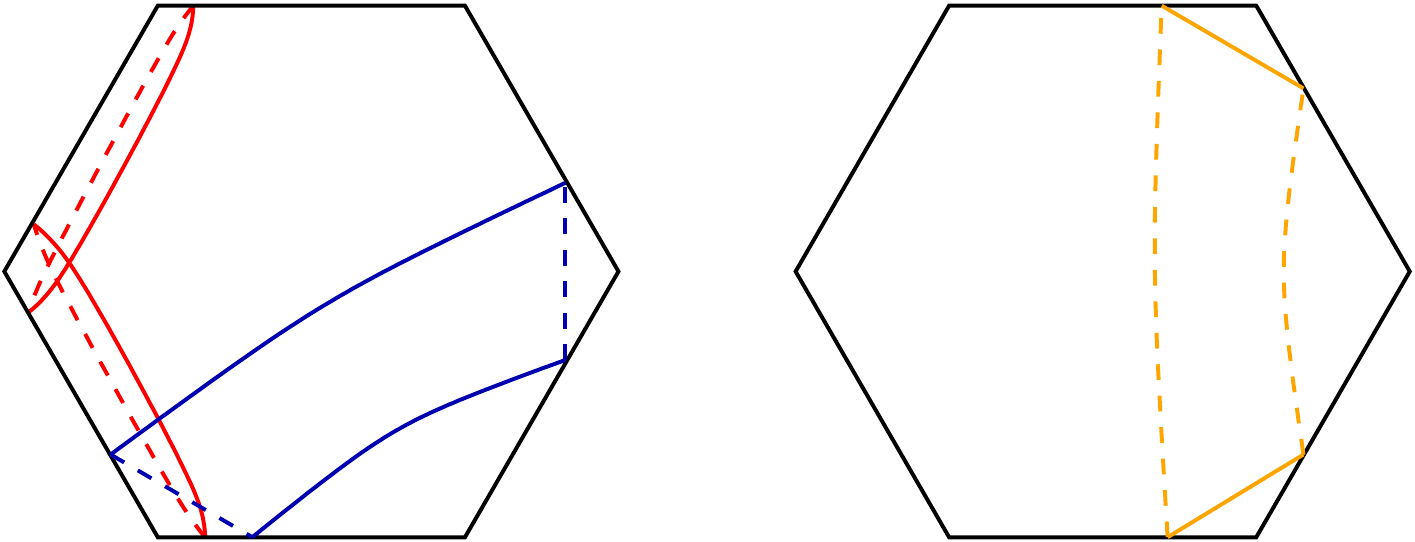}
     \caption{On the left, the sets $O$ (red) and $D$ (blue); on the right, the curve $\gamma_{i,j}$ they uniquely determine.}
     \label{fig:Sec2-7Fig1y2}
 \end{figure}
 
 The rest of the cases are analogous to one of the above.
\subsection{Proof of Theorem \ref{TeoB}}\label{subsec2-7}
 For the purposes of this work, we define the set $\Xf{S}$ as the set $\mathfrak{X}_{A}$ from Section 4 in \cite{Ara2}, which was proved to be rigid (see Theorem 4.2 in \cite{Ara2}).
 
 According to the notation above, we can define $\Xf{S}$ as follows: $$\Xf{S} \ColonEqq \Of^{1} \cup \Gf\cdot\Of.$$
 
 \begin{proof}[\textbf{Proof of Theorem \ref{TeoB}}]
  To prove Theorem \ref{TeoB} we only need to prove that $\Yf{S} \subset \Xf{S}^{1}$, since this would imply that (by Theorem \ref{TeoA}) $\ccomp{S} = \Yf{S}^{\omega} \subset \Xf{S}^{\omega}$.
  
  For $\gamma \in \Of$, it is obvious that $\gamma \in \Xf{S}$. Thus, we suppose that $\gamma \in \Dfone \cup \Dftwo$.
  
  If $\gamma \in \Dfone$, there exists $0 \leq i < j <n$ such that $\gamma = \beta_{i,j}$. Thus we have two possible cases:
  \begin{enumerate}
   \item Suppose that $j \neq i+2$ (modulo $n$). Then (see Figure \ref{fig:Sec2-8Fig1}) $$\beta_{i,j} = \langle \{\alpha_{i+1}, \ldots, \alpha_{j-2}\} \cup \{\alpha_{j+1}, \ldots \alpha_{i-2}\} \cup \{\eta_{\alpha_{i}}(\alpha_{i-1})\}\rangle.$$
   \item Suppose that $j = i+2$ (modulo $n$). Then $\beta_{i,j} = \eta_{\alpha_{i+1}}^{-1}(\alpha_{i})$. 
  \end{enumerate}
 
 \begin{figure}[ht]
     \centering
     \labellist
     \pinlabel $i$ [bl] at 225 250
     \pinlabel $j$ [r] at 75 0
     \endlabellist
     \includegraphics[height=4cm]{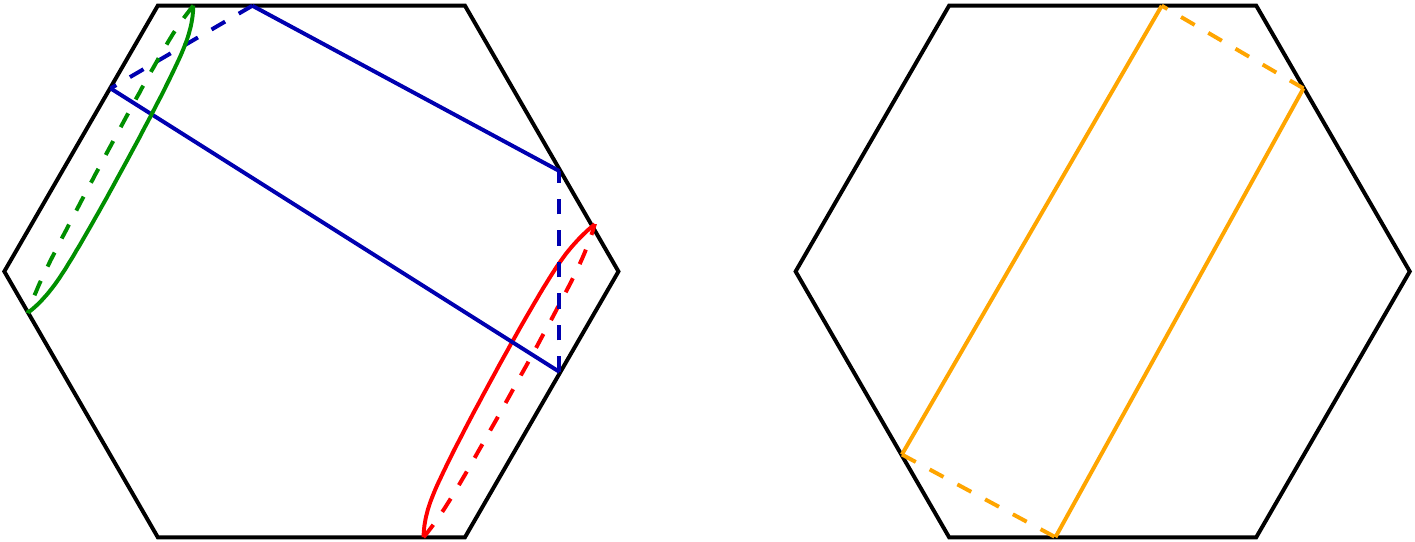}
     \caption{On the left, the curves needed to uniquely determine the curve $\beta_{i,j}$ on the right.}
     \label{fig:Sec2-8Fig1}
 \end{figure}
 
 If $\gamma \in \Dftwo$, there exists $0 \leq i < j <n$ such that $\gamma = \gamma_{i,j}$. Thus we have two possible cases, which are analogous to (1) and (2) above, simply substituting $\eta_{\alpha_{i}}(\alpha_{i-1})$ for $\eta_{\alpha_{i}}^{-1}(\alpha_{i-1})$ in case (1) and $\eta_{\alpha_{i+1}}^{-1}(\alpha_{i})$ for $\eta_{\alpha_{i+1}}(\alpha_{i})$ in case (2).
 
 Then, $\Yf{S} \subset \Xf{S}^{1}$, and therefore $\Xf{S}^{\omega} = \ccomp{S}$.
 \end{proof}
\section{Genus one case}\label{sec3}
 Given that $\ccomp{S_{1,2}} \cong \ccomp{S_{0,5}}$, Subsections \ref{subsec2-lc} and \ref{subsec2-7} already prove Theorems \ref{TeoB} and \ref{TeoA} for the twice-punctured torus. So, in this section we assume that $S = S_{1,n}$ with $n \geq 3$ (so that $\kappa(S) \geq 3$), unless otherwise stated.
 
 The structure of this section is as follows: In Section \ref{subsec3-1} we define $\Cf$ and $\Df$ which are the basis upon which we construct $\Yf{S}$; in Section \ref{subsec3-2} we define the set of auxiliary curves $\Af$, and we also define $\Yf{S}$; in Section \ref{subsec3-3} we give the proof of Theorem \ref{TeoA} pending the proof of a key lemma (Lemma \ref{KeyLemag1}); in Subsections \ref{subsec3-4}, \ref{subsec3-5}, \ref{subsec3-6} and \ref{subsec3-7}, we give the proof of the key lemma; finally, in Section \ref{subsec3-8} we recall from \cite{Ara2} the definition of $\Xf{S}$ and prove Theorem \ref{TeoB}.
\subsection{The basis of $\Yf{S}$}\label{subsec3-1}  
 Let $\Cf \ColonEqq \{\alpha_{0}^{0}, \ldots, \alpha_{0}^{n-1}, \alpha_{1}\}$ be a set of curves such that $\alpha_{0}^{i}$ is disjoint from $\alpha_{0}^{j}$ if $i \neq j$, and $i(\alpha_{1},\alpha_{0}^{i}) = 1$ for all $0 \leq i <n$. See Figure \ref{fig:Sec3-1Fig1}.
 
 Using $\Cf$, we number the punctures in $S$ in the following way: the $i$-th puncture is in the annulus bounded by $\alpha_{0}^{i-1}$ and $\alpha_{0}^{i}$. In Figure \ref{fig:Sec3-1Fig1}, this means that the $i$-th puncture is on the left of $\alpha_{0}^{i}$.
 
 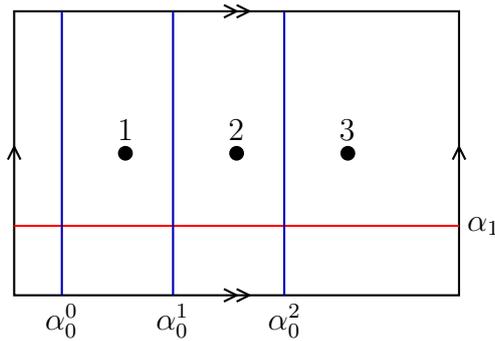
\begin{figure}[ht]
     \centering
     \labellist
     \pinlabel $1$ [b] at 107 140
     \pinlabel $2$ [b] at 206 140
     \pinlabel $3$ [b] at 305 140
     \pinlabel $\alpha_{1}$ [l] at 405 70
     \pinlabel $\alpha_{0}^{0}$ [t] at 50 7
     \pinlabel $\alpha_{0}^{1}$ [t] at 149 7
     \pinlabel $\alpha_{0}^{2}$ [t] at 249 7
     \endlabellist
     \includegraphics[height=4cm]{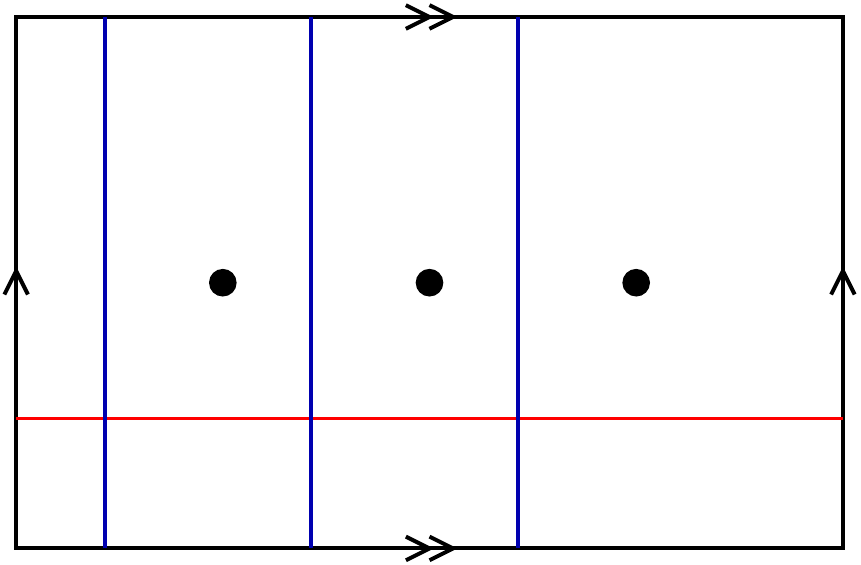}
     \caption{The curve $\alpha_{1}$ in red, and the curves $\alpha_{0}^{0}, \alpha_{0}^{1}, \alpha_{0}^{2}$ in blue.}
     \label{fig:Sec3-1Fig1}
 \end{figure}
 
 With this, we can define for each $0 \leq i <n$ the following curve (see Figure \ref{fig:Sec3-1Fig2}): $$\beta_{i} \ColonEqq \langle \Cf \backslash \{\alpha_{0}^{i}\}\rangle.$$
 
 \begin{figure}[ht]
     \centering
     \labellist
     \pinlabel $\beta_{1}$ [b] at 705 165
     \endlabellist
     \includegraphics[height=4cm]{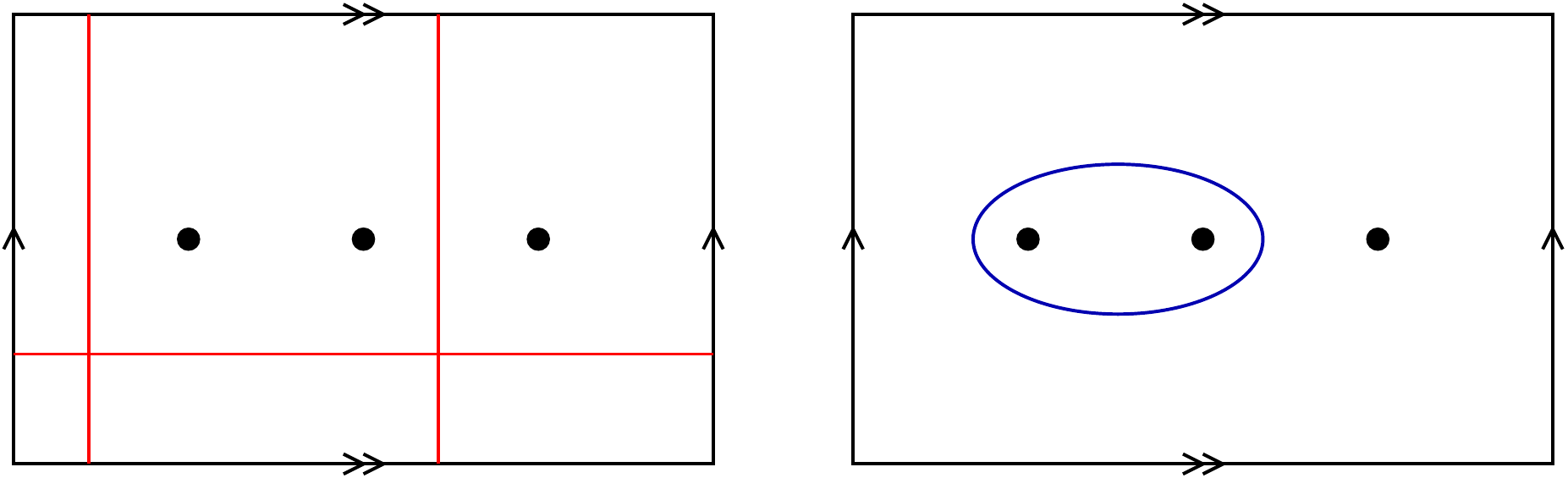}
     \caption{On the left, the curves needed to uniquely determine the curve $\beta_{1}$ on the right.}
     \label{fig:Sec3-1Fig2}
 \end{figure}
 
 Now, we define $\Df \ColonEqq \{\beta_{i}\}_{i = 0}^{n-1}$. Note that $\Df$ is a closed outer chain.
 
 Ideally we would like $\Yf{S}$ to be equal to the union of $\Cf$ and $\Df$; however, $(\Cf \cup \Df)^{2} = (\Cf \cup \Df)^{1}$. So, we need some auxiliary curves.
\subsection{Auxiliary curves}\label{subsec3-2}
 To define the auxiliary curves, first we must recall (from Subsection \ref{subsec2-3}) that if $\alpha$ is an outer curve on $S$, then $\eta_{\alpha}$ denotes the half-twist along $\alpha$.
 
 For each $0 \leq i <n$, we now define the following curves (with indices modulo $n$): $$\veps_{i}^{+} \ColonEqq \eta_{\beta_{i}}(\beta_{i-1}) \hspace{2cm} \veps_{i}^{-} \ColonEqq \eta_{\beta_{i}}^{-1}(\beta_{i-1}).$$
 
 See Figure \ref{fig:Sec3-2Fig1}.
 
 \begin{figure}[ht]
     \centering
     \labellist
     \pinlabel $i$ [tr] at 100 130
     \pinlabel $i$ [br] at 580 143
     \pinlabel $\veps_{i}^{+}$ [t] at 260 130
     \pinlabel $\veps_{i}^{-}$ [b] at 740 140
     \endlabellist
     \includegraphics[height=4cm]{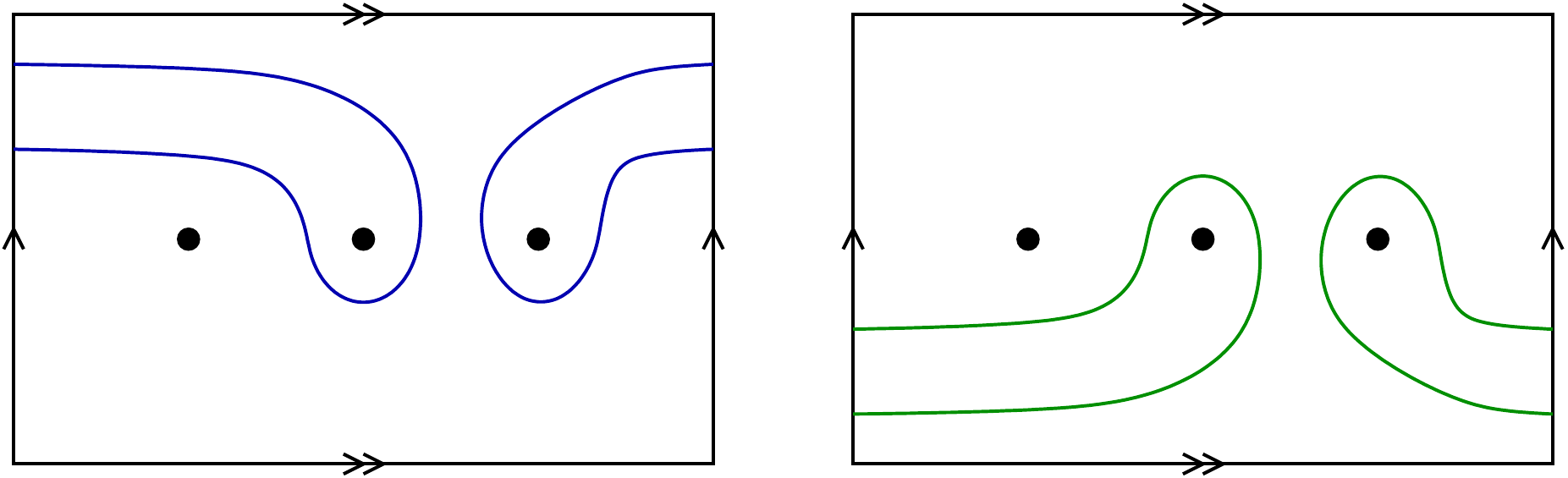}
     \caption{On the left the curve $\veps_{i}^{+}$; on the right the curve $\veps_{i}^{-}$.}
     \label{fig:Sec3-2Fig1}
 \end{figure}
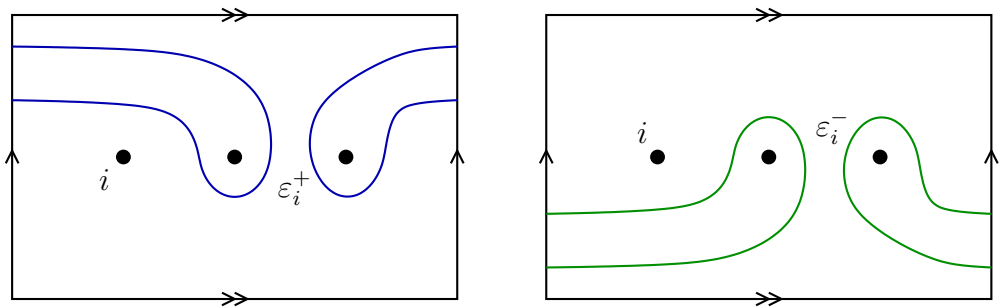
 
 Thus, we can define the sets $\Af^{+} \ColonEqq \{\veps_{i}^{+}\}_{i = 0}^{n-1}$, $\Af^{-} \ColonEqq \{\veps_{i}^{-}\}_{i = 0}^{n-1}$, and $\Af \ColonEqq \Af^{+} \cup \Af^{-}$.
 
 Finally, we define $$\Yf{S} \ColonEqq \Cf \cup \Df \cup \Af.$$
\subsection{Proof of Theorem \ref{TeoA}}\label{subsec3-3}
 For the proof of Theorem \ref{TeoA}, we first recall some definitions and properties of Dehn twists along curves on $S$.
 
 Let $\alpha$ be a curve on $S$; we denote by $\tau_{\alpha}$ the (left) Dehn twist along $\alpha$. Similarly to Remark \ref{RemSym}, we have the following property.
 
 \begin{Rem}\label{RemSymtau}
  Note that if $\alpha$ and $\beta$ are curves that intersect once, then we have that $\tau_{\alpha}(\beta) = \tau_{\beta}^{-1}(\alpha)$. Once again, this fact can also be found in Proposition 3.9 in \cite{Ara2}.
 \end{Rem}
 
 Similarly to the notation around half-twists, if $C$ and $A$ are sets of curves on $S$, we denote by $\tau_{C}(A)$ the following set $$\tau_{C}(A) = \bigcup_{\gamma \in C} \tau_{\gamma}(A).$$
 
 Now, let $\Gf \ColonEqq \{\tau_{\alpha}^{\pm 1}: \alpha \in \Cf\} \cup \{\eta_{\beta}^{\pm 1}: \beta \in \Df\}$. It is a well-known fact (Theorem 4.14 in \cite{FarbMar}) that $\Gf$ is a symmetric generating set of $\Mod{S}$ (often called the Humphries-Lickorish generating set). As in the previous section if $A$ is a set of curves on $S$, we denote by $\Gf \cdot A$ the following set $\Gf \cdot A \ColonEqq \tau_{\Cf}^{\pm 1}(A) \cup \eta_{\Df}^{\pm 1}(A)$.
 
 As in the previous section, we state a key lemma for the proof of Theorem \ref{TeoA}, leaving its proof to the following subsections (see Subsections \ref{subsec3-4}, \ref{subsec3-5}, \ref{subsec3-6} and \ref{subsec3-7}).
 
 \begin{Lema}\label{KeyLemag1}
  Let $\Gf$ be as above. Then $\Gf \cdot \Yf{S} \subset \Yf{S}^{6}$.
 \end{Lema}
 
 \begin{proof}[\textbf{Proof of Theorem \ref{TeoA}}]
  We divide the proof into two parts: first we prove that all outer curves and non-separating curves on $S$ are elements in $\Yf{S}^{\omega}$; then we argue that any curve on $S$ is uniquely determined by a finite set of non-separating curves and outer curves, thus obtaining the desired result.
  
  \textit{First part:} Let $\gamma$ be either an outer curve or a non-separating curve on $S$. Recall that all outer curves on $S$ have the same topological type, and it is the same case for all non-separating curves. Then, there exists a curve $\delta \in \Yf{S}$ and a homeomorphism $h \in \Mod{S}$ such that $h(\delta) = \gamma$. Since $\Gf$ is a symmetric generating set of $\Mod{S}$, there exist elements $f_{1}, \ldots, f_{k} \in \Gf$ and powers $n_{1}, \ldots, n_{k} \geq 0$ such that $$f_{1}^{n_{1}} \circ \cdots \circ f_{k}^{n_{k}}(\delta) = \gamma.$$ By an iterated use of Lemma \ref{KeyLemag1}, this implies that $\gamma \in \Yf{S}^{6(n_{1} + \cdots + n_{k})}$. Thus, every outer curve and every non-separating curve on $S$ is an element of $\Yf{S}^{\omega}$.
  
  \textit{Second part:} Let $\gamma$ be a separating curve on $S$ that is not an outer curve. The topological type of $\gamma$ is then determined the punctures it bounds. Thus, up to homeomorphism, $\gamma$ is as in the example on Figure \ref{fig:Sec3-3Fig1} (right). Then, there exist finite sets $N$ and $O$ of non-separating and outer curves respectively, that uniquely determine $\gamma$. See Figure \ref{fig:Sec3-3Fig1} for an example. The first part of this proof implies that there exists $k \geq 0$ such that $N \cup O \subset \Yf{S}^{k}$. Thus, $\gamma \in \Yf{S}^{k+1}$ and therefore every curve on $S$ that is neither an outer curve nor a non-separating curve, is also an element of $\Yf{S}^{\omega}$.
  
  \begin{figure}[ht]
      \centering
      \includegraphics[height=4cm]{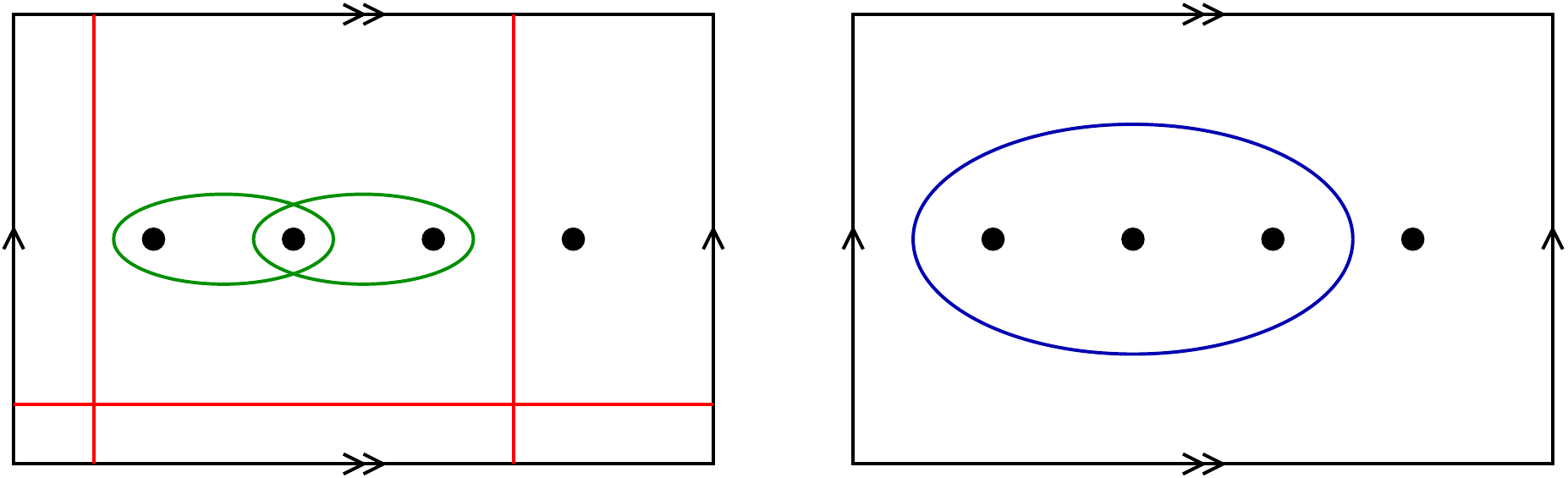}
      \caption{On the left, the set $N$ (red) and the set $O$ (green); on the right the non-outer separating curve they uniquely determine.}
      \label{fig:Sec3-3Fig1}
  \end{figure}
 \end{proof}
 
 Similarly to the previous section, 
 we prove Lemma \ref{KeyLemag1} by proving the following claims:\\[0.3cm]
 \textbf{Claim 1:} $\tau_{\Cf}^{\pm 1}(\Cf \cup \Df) \subset \Yf{S}^{3}$.\newline
 \textbf{Claim 2:} $\tau_{\Cf}^{\pm 1}(\Af) \subset \Yf{S}^{6}$.\newline
 \textbf{Claim 3:} $\eta_{\Df}^{\pm 1}(\Cf \cup \Df) \subset \Yf{S}^{1}$.\newline
 \textbf{Claim 4:} $\eta_{\Df}^{\pm 1}(\Af) \subset \Yf{S}^{3}$.
\subsection{Proof of Claim 1: $\tau_{\Cf}^{\pm 1}(\Cf \cup \Df) \subset \Yf{S}^{3}$}\label{subsec3-4}
 We divide this proof into two parts:
 
 \textbf{First part} $\tau_{\Cf}^{\pm 1}(\Df) \subset \Yf{S}^{2}$\textbf{:} Note that $\tau_{\alpha}(\beta) = \beta$ if $\alpha$ and $\beta$ are disjoint from each other. So, we only need to prove that $\tau_{\alpha_{0}^{i}}^{\pm 1}(\beta_{i}) \in \Yf{S}^{2}$; to prove this we need some auxiliary curves. 
 
 For each $0 \leq i <n$ we define the following curves (recall the subindices are modulo $n$) $$\gamma_{i}^{+} \ColonEqq \langle\{\beta_{i+1}, \ldots, \beta_{i-2}\} \cup \{\veps_{i}^{+}\} \cup \{\alpha_{1}\}\rangle \in \Yf{S}^{1},$$ $$\gamma_{i}^{-} \ColonEqq \langle \{\beta_{i+1}, \ldots, \beta_{i-2}\} \cup \{\veps_{i}^{-}\} \cup \{\alpha_{1}\}\rangle \in \Yf{S}^{1}.$$
 
 See Figure \ref{fig:Sec3-4Fig1y2} for examples.
 
 \begin{figure}[ht]
     \centering
     \labellist
     \pinlabel $i$ [t] at 206 127
     \pinlabel $i$ [t] at 683 127
     \pinlabel $\gamma_{i}^{+}$ [l] at 884 65
     \endlabellist
     \includegraphics[height=4cm]{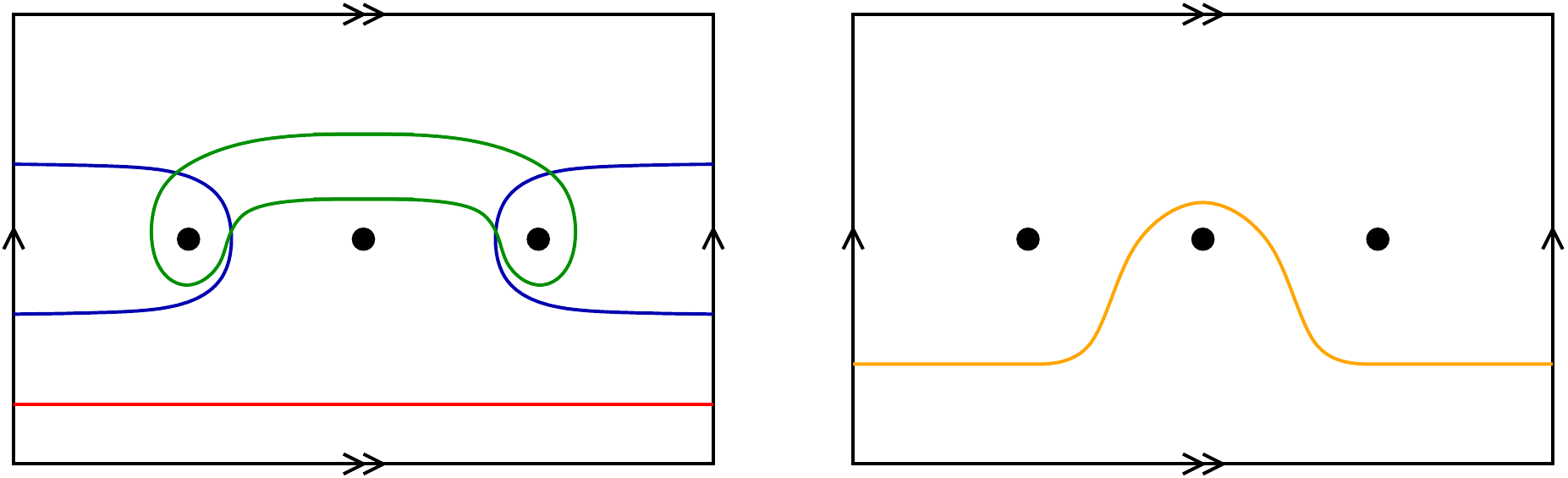}
     \labellist
     \pinlabel $i$ [b] at 206 147
     \pinlabel $i$ [b] at 683 147
     \pinlabel $\gamma_{i}^{-}$ [l] at 884 195
     \endlabellist
     \includegraphics[height=4cm]{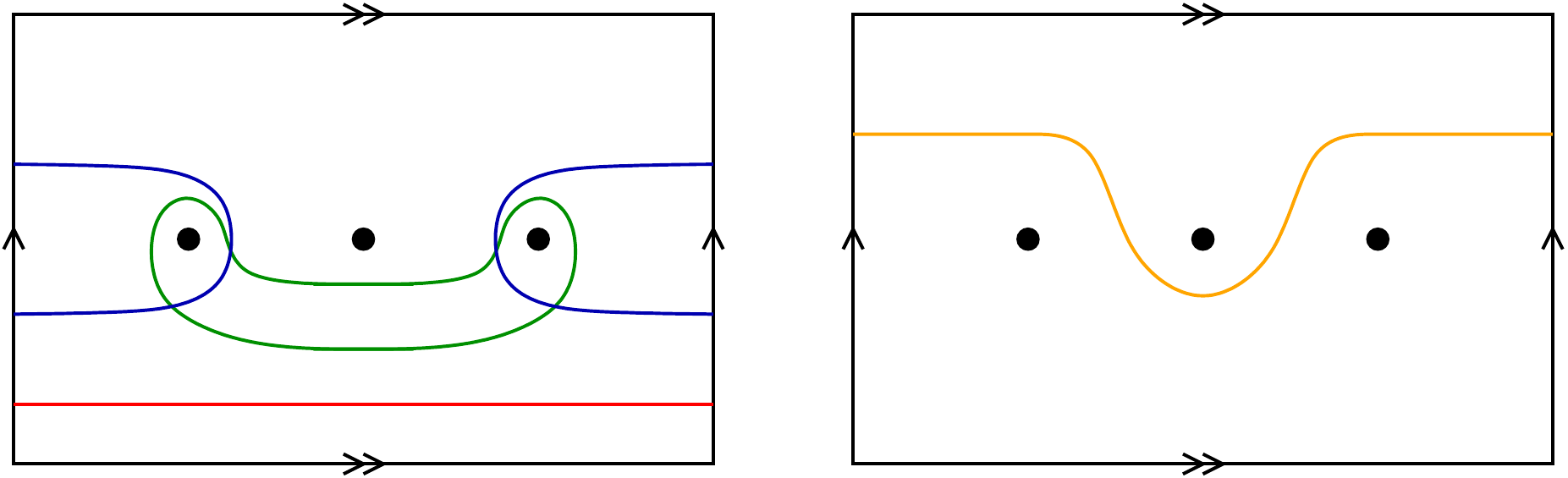}
     \caption{On the left, the curves needed to uniquely determine the curves $\gamma_{i}^{+}$ and $\gamma_{i}^{-}$.}
     \label{fig:Sec3-4Fig1y2}
 \end{figure}
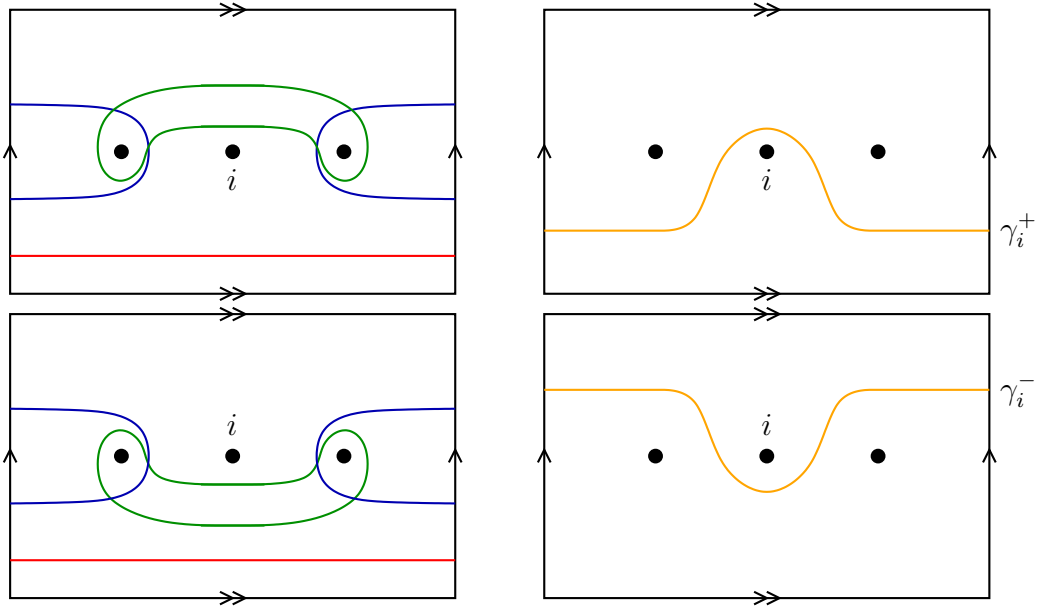
 
 Now, we have the following (see Figure \ref{fig:Sec3-4Fig3y4}):
 $$\tau_{\alpha_{0}^{i}}(\beta_{i}) = \langle \{\gamma_{i}^{-},\gamma_{i+1}^{+}\} \cup \{\alpha_{0}^{i+1}, \ldots, \alpha_{0}^{i-1}\}\rangle \in \Yf{S}^{2},$$ $$\tau_{\alpha_{0}^{i}}^{-1}(\beta_{i}) = \langle \{\gamma_{i}^{+}, \gamma_{i+1}^{-}\} \cup \{\alpha_{0}^{i+1}, \ldots, \alpha_{0}^{i-1}\}\rangle \in \Yf{S}^{2}.$$
 
 \begin{figure}[ht]
     \centering
     \labellist
     \pinlabel $i$ [b] at 206 144
     \pinlabel $\tau_{\alpha_{0}^{i}}(\beta_{i})$ [b] at 782 145
     \endlabellist
     \includegraphics[height=4cm]{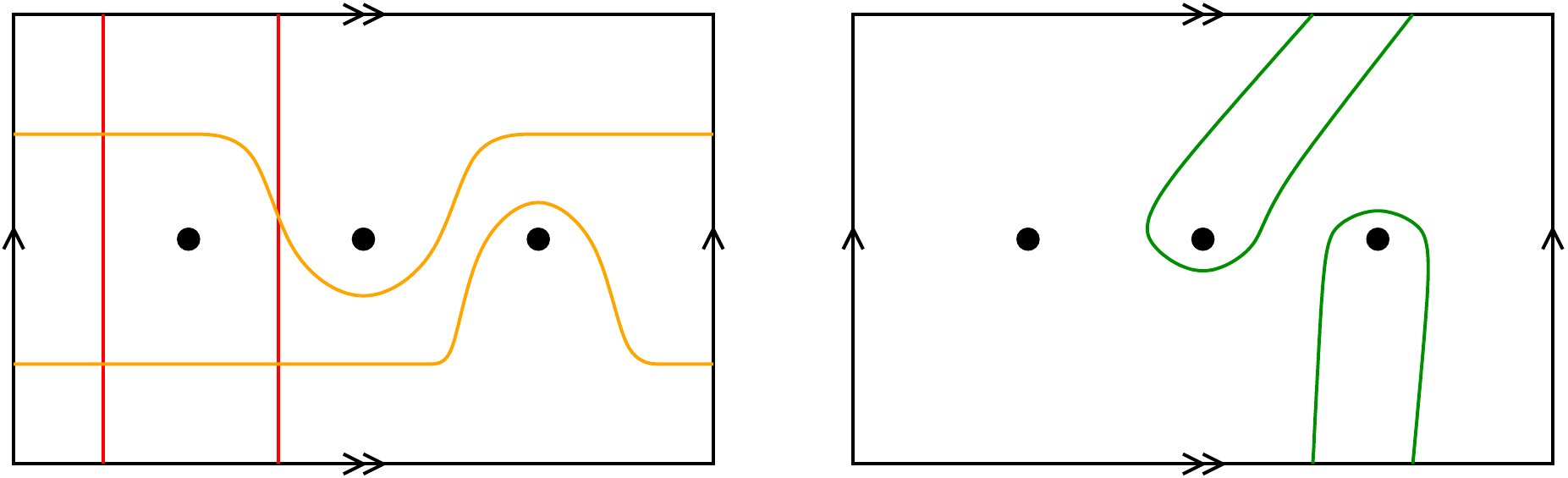}\\[5mm]
     \labellist
     \pinlabel $i$ [t] at 206 124
     \pinlabel $\tau_{\alpha_{0}^{i}}^{-1}(\beta_{i})$ [t] at 782 125
     \endlabellist
     \includegraphics[height=4cm]{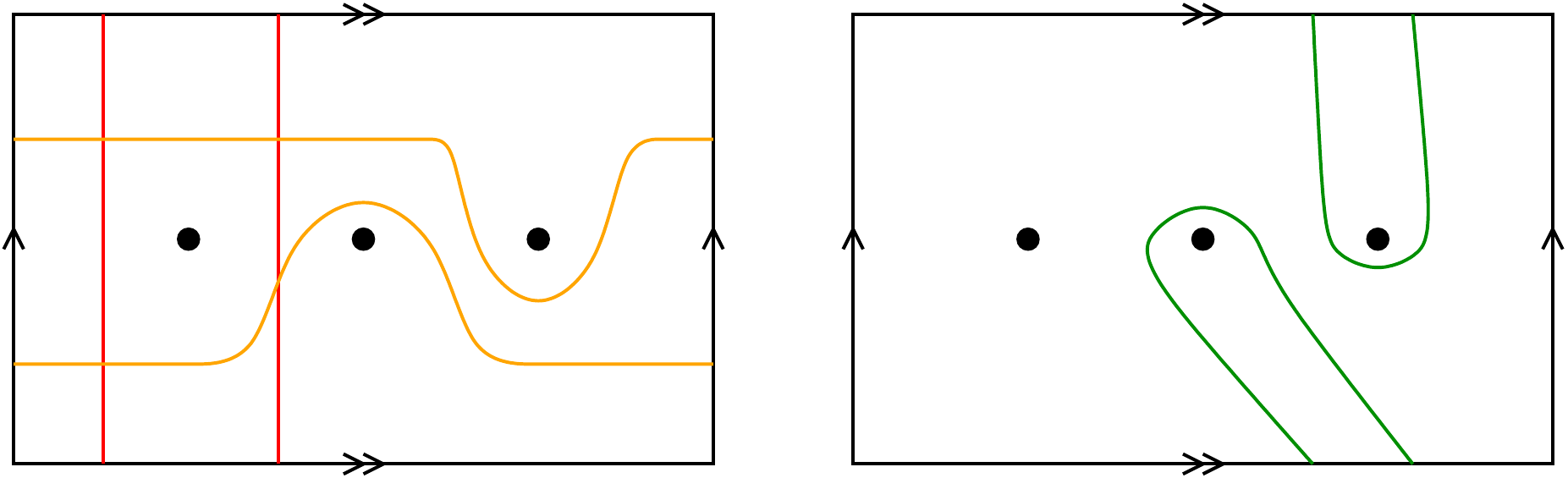}
     \caption{On the left, the curves needed to uniquely determine the curves on the right $\tau_{\alpha_{0}^{i}}(\beta_{i})$ (top) and $\tau_{\alpha_{0}^{i}}^{-1}(\beta_{i})$ (bottom).}
     \label{fig:Sec3-4Fig3y4}
 \end{figure}
 
 Thus, $\tau_{\Cf}^{\pm 1}(\Df) \subset \Yf{S}^{2}$.
 
 \textbf{Second part} $\tau_{\Cf}^{\pm 1}(\Cf) \subset \Yf{S}^{3}$\textbf{:} Recalling Remark \ref{RemSymtau} and that $\tau_{\alpha}(\beta) = \beta$ if $\alpha$ and $\beta$ are disjoint from each other, we only need to prove that $\tau_{\alpha_{0}^{i}}^{\pm1}(\alpha_{1}) \in \Yf{S}^{3}$ for each $0 \leq i <n$. 
 
 Note that $\alpha_{1} = \langle \Df \rangle$ (see Figure \ref{fig:Sec3-4Fig5}), thus $\tau_{\alpha_{0}^{i}}^{\pm1}(\alpha_{1}) = \langle \tau_{\alpha_{0}^{i}}^{\pm1}(\Df)\rangle$. This implies that $\tau_{\alpha_{0}^{i}}^{\pm1}(\alpha_{1}) \in \Yf{S}^{3}$ for each $0 \leq i <n$.
 
 \begin{figure}[ht]
     \centering
     \includegraphics[height=4cm]{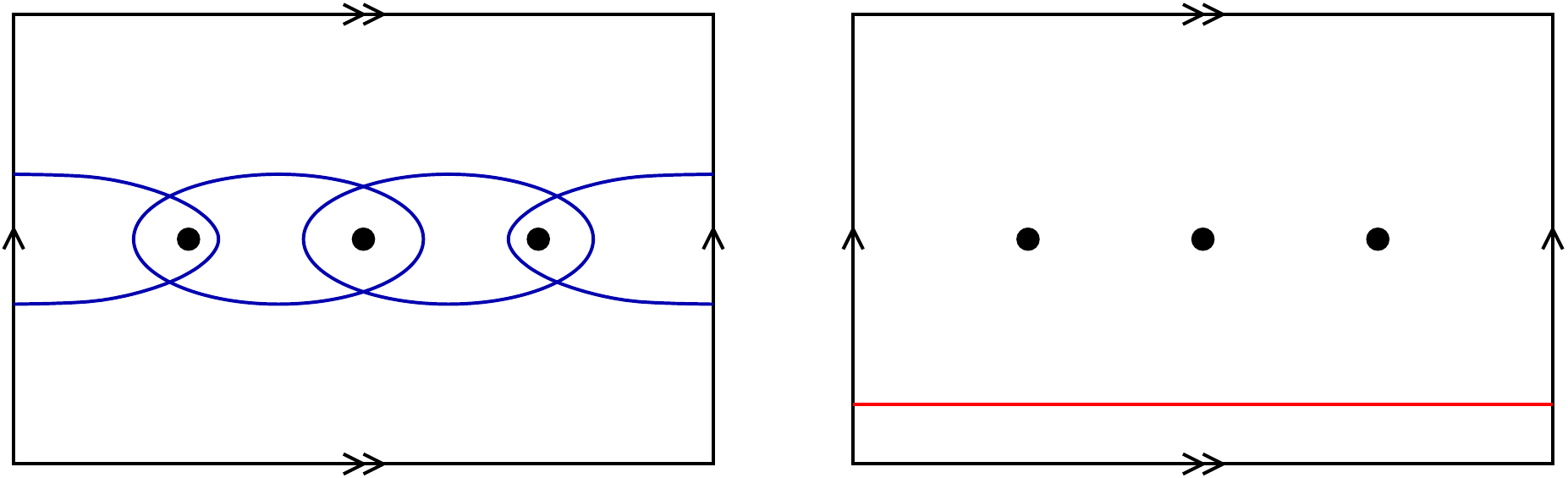}
     \caption{On the left, the set $\Df$; on the right the curve $\alpha_{1}$.}
     \label{fig:Sec3-4Fig5}
 \end{figure}
 
 Therefore, $\tau_{\Cf}^{\pm 1}(\Cf \cup \Df) \subset \Yf{S}^{3}$.
\subsection{Proof of Claim 2: $\tau_{\Cf}^{\pm 1}(\Af) \subset \Yf{S}^{6}$}\label{subsec3-5}
 As was done in the previous subsection, we only need to prove that for each $0 \leq i <n$ we have that $\tau_{\alpha_{0}^{i-1}}^{\pm 1}(\veps_{i}^{\pm}), \tau_{\alpha_{0}^{i}}^{\pm 1}(\veps_{i}^{\pm}) \in \Yf{S}^{6}$, since any other curve in $\Cf$ is disjoint from both $\veps_{i}^{+}$ and $\veps_{i}^{-}$.
 
 To do this, we divide the proof into two parts, according to the number of punctures in $S$.
 
 \textbf{First part, $n =3$:} Since we do not have enough punctures for this part, we uniquely determine $\veps_{i}^{\pm}$ by using the curves $\gamma_{i}^{\pm}$ from Subsection \ref{subsec3-4} twice.
 
 In the case for $\alpha_{0}^{i-1}$ recall and check that:
 $$\gamma_{i+1}^{\pm} := \langle \{\beta_{i+2}\} \cup \{\veps_{i+1}^{\pm}\} \cup \{\alpha_{1}\} \rangle, \text{ (see Figure \ref{fig:Sec3-4Fig1y2})}$$
 $$\gamma_{i}^{\pm} = \langle \{\beta_{i+1}\} \cup \{\alpha_{1}\} \cup \{\gamma_{i+1}^{\mp}\} \rangle, \text{ (see Figure \ref{fig:Sec3-5Fig1y2} top)}$$
 $$\veps_{i}^{\pm} = \langle \{\alpha_{0}^{i-2}\} \cup \{\alpha_{1}\} \cup \{\gamma_{i}^{\pm}\} \rangle. \text{ (see Figure \ref{fig:Sec3-5Fig1y2} bottom)}$$
 
 \begin{figure}[ht]
     \centering
     \labellist
     \pinlabel $i$ [tr] at 200 130
     \endlabellist
     \includegraphics[height=4cm]{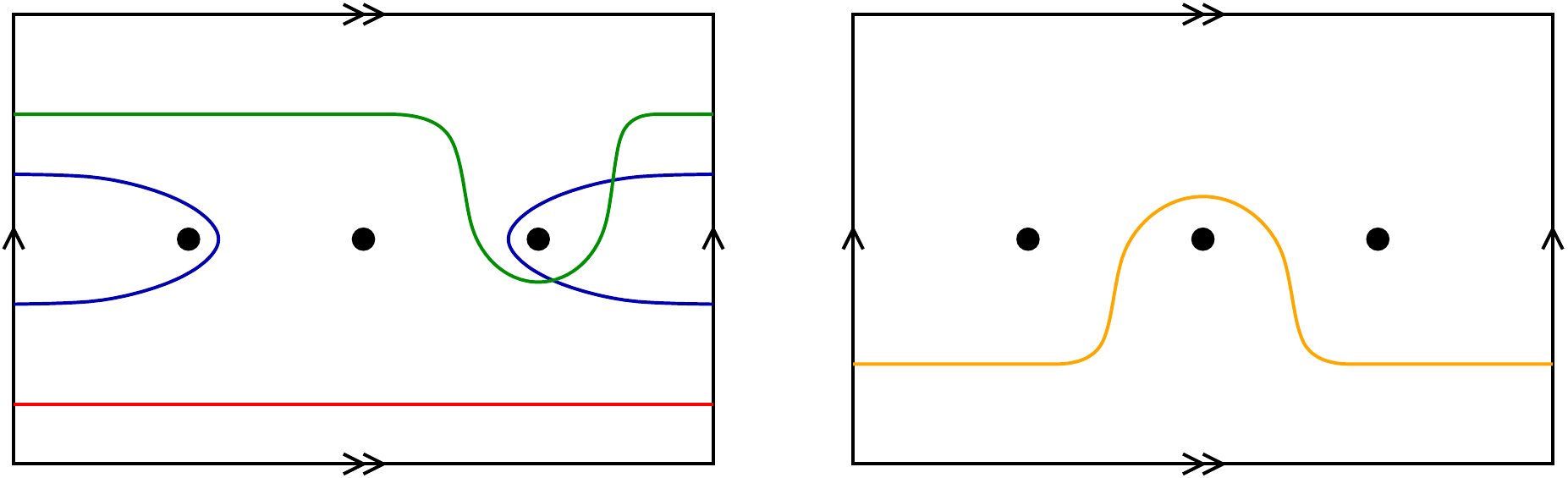}\\[5mm]
     \labellist
     \pinlabel $i$ [tr] at 200 130
     \endlabellist
     \includegraphics[height=4cm]{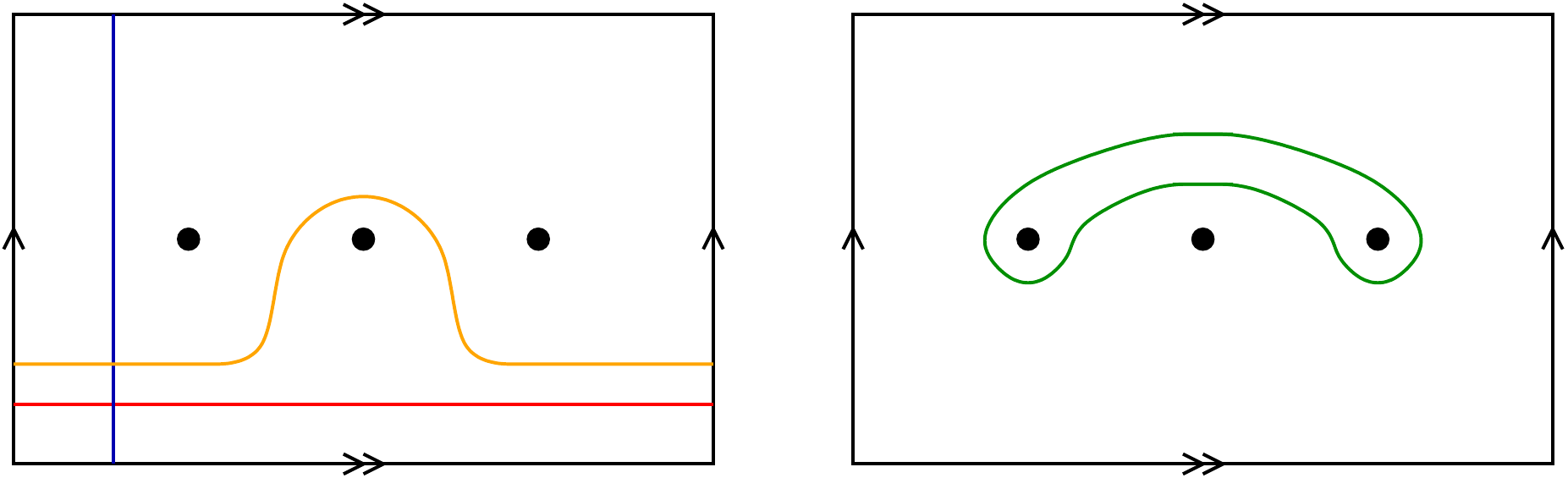}
     \caption{On the left, the curves needed to uniquely determine the curves $\gamma_{i}^{+}$ (top) and $\veps_{i}^{+}$ (bottom) on the right.}
     \label{fig:Sec3-5Fig1y2}
 \end{figure}
 
 Thus, by Claim 1 we have that $$\begin{array}{rl}
    \tau_{\alpha_{0}^{i-1}}^{\pm 1}(\gamma_{i+1}^{\pm})  & = \langle \tau_{\alpha_{0}^{i-1}}^{\pm 1}(\{\beta_{i+2}\}) \cup \tau_{\alpha_{0}^{i-1}}^{\pm 1}(\{\veps_{i+1}^{\pm}\}) \cup \tau_{\alpha_{0}^{i-1}}^{\pm 1}(\{\alpha_{1}\}) \rangle\\
      & = \langle \tau_{\alpha_{0}^{i-1}}^{\pm 1}(\{\beta_{i+2}\}) \cup \{\veps_{i+1}^{\pm}\} \cup \tau_{\alpha_{0}^{i-1}}^{\pm 1}(\{\alpha_{1}\}) \rangle \in \Yf{S}^{4},
 \end{array}$$ hence $$\tau_{\alpha_{0}^{i-1}}^{\pm 1}(\gamma_{i}^{\pm}) = \langle \tau_{\alpha_{0}^{i-1}}^{\pm 1}(\{\beta_{i+1}\}) \cup \tau_{\alpha_{0}^{i-1}}^{\pm 1}(\{\alpha_{1}\}) \cup \tau_{\alpha_{0}^{i-1}}^{\pm 1}(\{\gamma_{i+1}^{\mp})\} \rangle \in \Yf{S}^{5},$$ and finally $$\tau_{\alpha_{0}^{i-1}}^{\pm 1}(\veps_{i}^{\pm}) = \langle \tau_{\alpha_{0}^{i-1}}^{\pm 1}(\{\alpha_{0}^{i-2}\}) \cup \tau_{\alpha_{0}^{i-1}}^{\pm 1}(\{\alpha_{1}\}) \cup \tau_{\alpha_{0}^{i-1}}^{\pm 1}(\{\gamma_{i}^{\pm}\}) \rangle \in \Yf{S}^{6}.$$
 
 The case for $\alpha_{0}^{i}$ is completely analogous, substituting $\gamma_{i+1}^{\pm}$ by $\gamma_{i-1}^{\pm}$.
 \vspace{3mm}
 
 \textbf{Second part, $n \geq 4$:} We proceed similarly to the first case in Section \ref{subsec2-5} to prove that $\tau_{\alpha}(\veps_{i}^{\pm}) \in \Yf{S}^{4}$ with $\alpha = \alpha_{0}^{i-1}, \alpha_{0}^{i}$. We can uniquely determine $\veps_{i}^{\pm}$ as follows: We define the set $$C = \{\alpha_{0}^{i+1}, \ldots, \alpha_{0}^{i-2}\} \cup \{\alpha_{1}\},$$ and the curve $$\veps^{\pm} = \left\{
 \begin{array}{cl}
    \veps_{i+1}^{\pm}  & \text{ for the case $\alpha_{0}^{i-1}$,} \\
    \veps_{i-1}^{\pm}  & \text{ for the case $\alpha_{0}^{i}$.}
 \end{array}\right.$$
 
 Then we check that (see Figure \ref{fig:Sec3-5Fig3y4}) $$\veps_{i}^{\pm} = \langle C \cup \{\veps^{\mp}\}\rangle.$$
 
 \begin{figure}[ht]
     \centering
     \includegraphics[height=4cm]{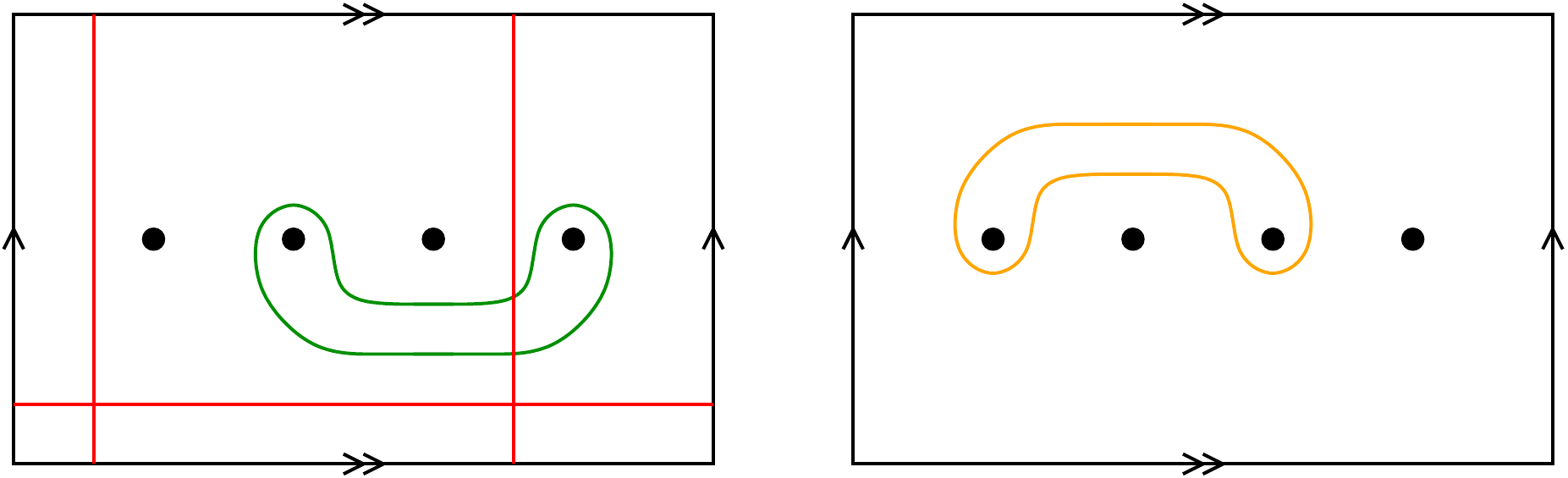}\\[5mm]
     \includegraphics[height=4cm]{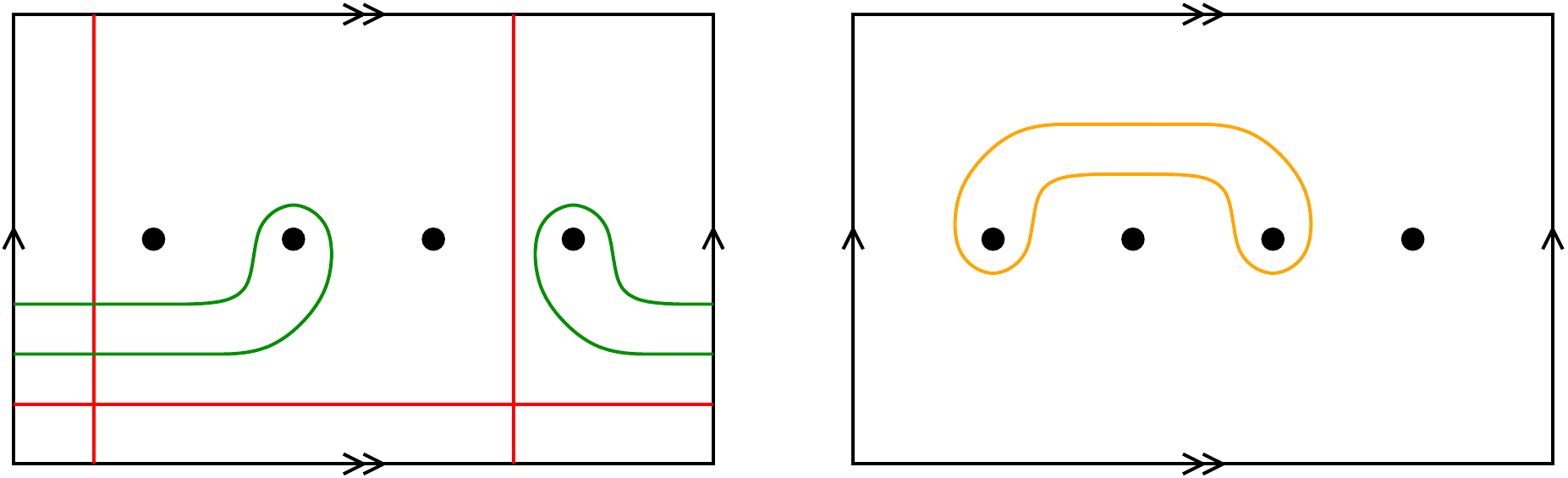}
     \caption{On the left, the curves needed to uniquely determine $\veps_{i}^{+}$ on the right, for the cases $\alpha_{0}^{i-1}$ (top) and $\alpha_{0}^{i}$ (bottom).}
     \label{fig:Sec3-5Fig3y4}
 \end{figure}
 
 
 Thus, by Claim 1 we have that $$\begin{array}{rl}
    \tau_{\alpha_{0}^{i-1}}(\veps_{i}^{\pm})  & = \langle \tau_{\alpha_{0}^{i-1}}(C) \cup \tau_{\alpha_{0}^{i-1}}(\{\veps_{i+1}^{\mp}\})\rangle \\
      &  = \langle \tau_{\alpha_{0}^{i-1}}(C) \cup \{\veps_{i+1}^{\mp}\}\rangle \in \Yf{S}^{4},
 \end{array}
  $$ $$\begin{array}{rl}
    \tau_{\alpha_{0}^{i}}(\veps_{i}^{\pm})  & = \langle \tau_{\alpha_{0}^{i}}(C) \cup \tau_{\alpha_{0}^{i}}(\{\veps_{j}\})\rangle \\
       &  = \langle \tau_{\alpha_{0}^{i}}(C) \cup \{\veps_{j}\}\rangle \in \Yf{S}^{4}.
  \end{array} $$ The cases for $\tau_{\alpha_{0}^{i-1}}^{-1}(\veps_{i}^{\pm}), \tau_{\alpha_{0}^{i}}^{-1}(\veps_{i}^{\pm}) \in \Yf{S}^{4}$ are analogous.
 \vspace{5mm}
 
 Therefore $\tau_{\Cf}^{\pm 1}(\Af) \subset \Yf{S}^{6}$.
\subsection{Proof of Claim 3: $\eta_{\Df}^{\pm 1}(\Cf \cup \Df) \subset \Yf{S}^{1}$}\label{subsec3-6}
 Recall that $\veps_{i}^{\pm} = \eta_{\beta_{i}}^{\pm1}(\beta_{i-1})$. Using this and Remark \ref{RemSym} we have that $\eta_{\Df}^{\pm 1}(\Df) \subset \Yf{S}$.
 
 Now, similarly to the previous subsections, we only need to prove that $\eta_{\beta_{i}}^{\pm 1}(\alpha_{0}^{i}) \in \Yf{S}^{1}$. This can be done as follows: We define a set $B \subset \Df$ as 
 $$ B = \left\{\begin{array}{cl}
    \varnothing  & \text{ if $n = 3$,} \\
    \{\beta_{i+2}, \ldots, \beta_{i-2}\}  & \text{if $n \geq 4$}.
 \end{array}\right.$$ Then we have (see Figure \ref{fig:Sec3-6Fig1ayb}) $$\eta_{\beta_{i}}(\alpha_{0}^{i}) = \langle \{\veps_{i}^{+},\veps_{i+1}^{-}\} \cup B \cup \{\alpha_{0}^{i+1}, \ldots, \alpha_{0}^{i-1}\}\rangle,$$ $$\eta_{\beta_{i}}^{-1}(\alpha_{0}^{i}) = \langle \{\veps_{i}^{-},\veps_{i+1}^{+}\} \cup B \cup \{\alpha_{0}^{i+1}, \ldots, \alpha_{0}^{i-1}\}\rangle.$$
 
 \begin{figure}[ht]
     \centering
     \labellist
     \pinlabel $i$ at 167 115
     \pinlabel $i$ at 645 155
     \endlabellist
     \includegraphics[height=4cm]{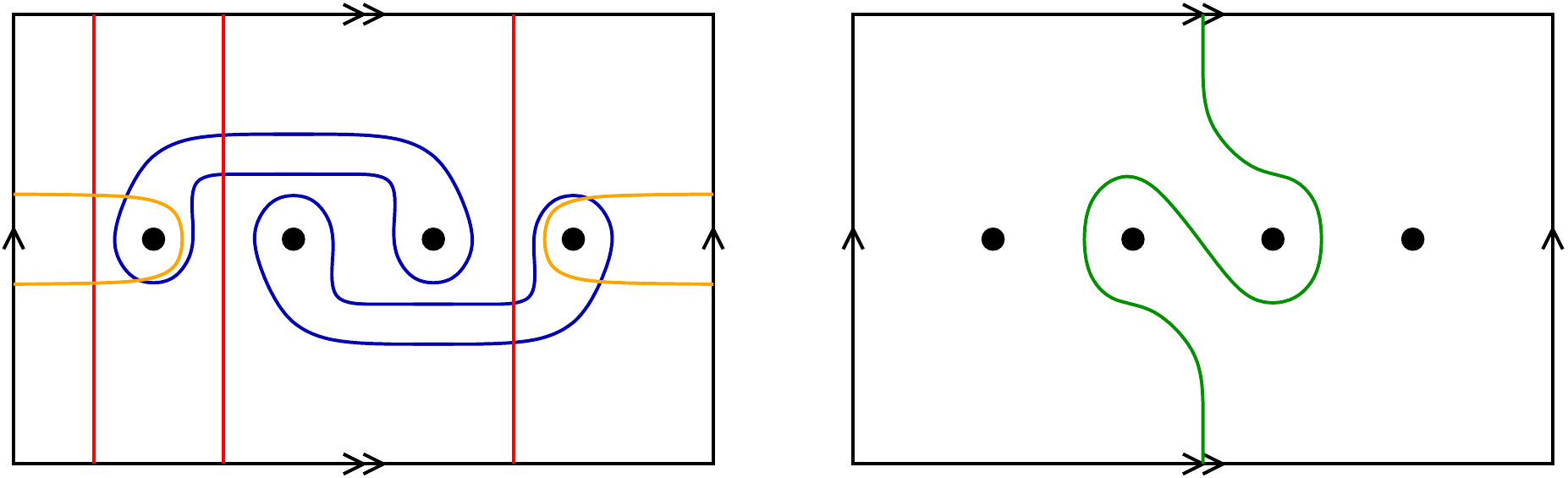}\\[5mm]
     \labellist
     \pinlabel $i$ at 167 155
     \pinlabel $i$ at 640 115
     \endlabellist
     \includegraphics[height=4cm]{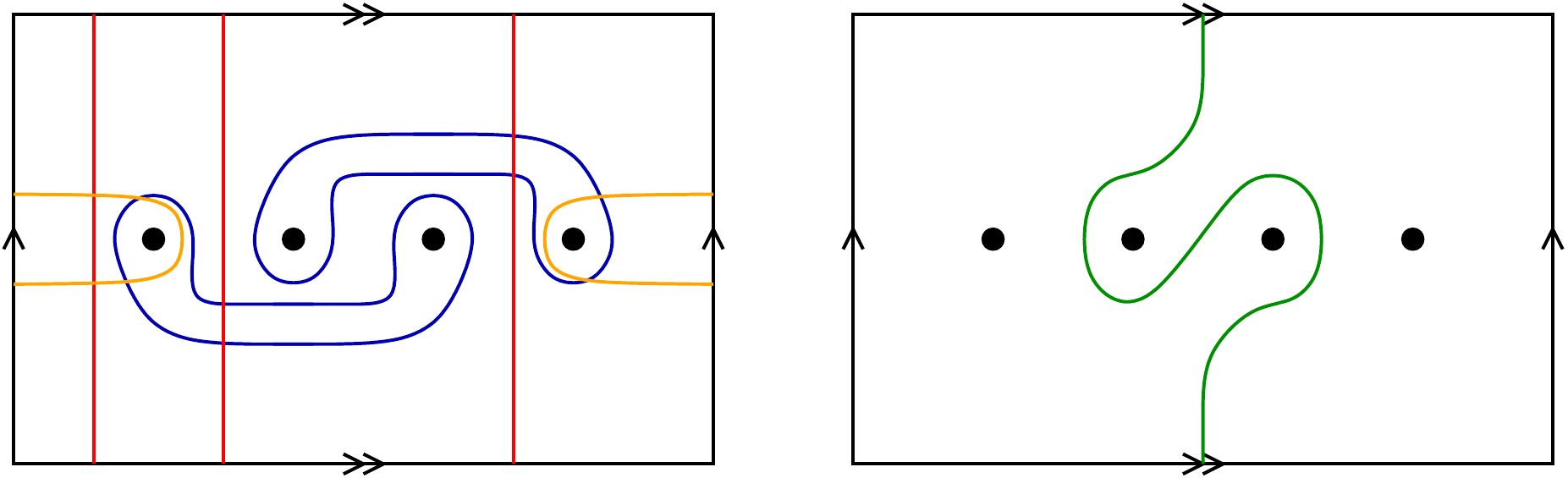}
     \caption{On the left, the curves needed to uniquely determine the curves $\eta_{\beta_{i}}(\alpha_{0}^{i})$ (top) and $\eta_{\beta_{i}}^{-1}(\alpha_{0}^{i})$ (bottom) on the right.}
     \label{fig:Sec3-6Fig1ayb}
 \end{figure}
 
 Therefore, $\eta_{\Df}^{\pm 1}(\Cf \cup \Df) \subset \Yf{S}^{1}$.
\subsection{Proof of Claim 4: $\eta_{\Df}^{\pm 1}(\Af) \subset \Yf{S}^{3}$}\label{subsec3-7}
 For this claim, similarly to the previous subsections, we only need to prove for $\veps = \veps_{i}^{+}, \veps_{i}^{-}$, that $\eta_{\beta_{i-2}}^{\pm 1}(\veps), \eta_{\beta_{i-1}}^{\pm 1}(\veps), \eta_{\beta_{i}}^{\pm 1}(\veps), \eta_{\beta_{i+1}}^{\pm 1}(\veps) \in \Yf{S}^{3}$.
 
 First note that by definition $\eta_{\beta_{i-1}}(\veps_{i}^{+}) = \eta_{\beta_{i-1}}^{-1}(\veps_{i}^{-}) = \beta_{i}  \in \Yf{S}$ and $\eta_{\beta_{i}}^{-1}(\veps_{i}^{+}) = \eta_{\beta_{i}}(\veps_{i}^{-}) = \beta_{i-1} \in \Yf{S}$.
 
 Also note that we can uniquely determine $\veps_{i}^{\pm}$ as in Subsection \ref{subsec3-5} as follows (see Figure \ref{fig:Sec3-7Fig0}): $$\veps_{i}^{\pm} = \langle \{\alpha_{1}\} \cup \{\alpha_{0}^{i+1}, \ldots, \alpha_{0}^{i-2}\} \cup \{\gamma_{i}^{\pm}\} \rangle,$$ hence, by Claim 3 and using that $\eta_{\beta_{i-2}}^{\pm 1}(\gamma_{i}^{+}) = \eta_{\beta_{i+1}}^{\pm 1}(\gamma_{i}^{+}) = \gamma_{i}^{+}$ and $\eta_{\beta_{i-2}}^{\pm 1}(\gamma_{i}^{-}) = \eta_{\beta_{i+1}}^{\pm 1}(\gamma_{i}^{-}) = \gamma_{i}^{-}$, we get that $\eta_{\beta_{i-2}}^{\pm 1}(\veps_{i}^{+}), \eta_{\beta_{i-2}}^{\pm 1}(\veps_{i}^{-}), \eta_{\beta_{i+1}}^{\pm 1}(\veps_{i}^{+}), \eta_{\beta_{i+1}}^{\pm 1}(\veps_{i}^{-}) \in \Yf{S}^{2}.$
 
 \begin{figure}[ht]
     \centering
     \labellist
     \pinlabel $i$ [t] at 87 125
     \pinlabel $i$ [t] at 563 125
     \endlabellist
     \includegraphics[height=4cm]{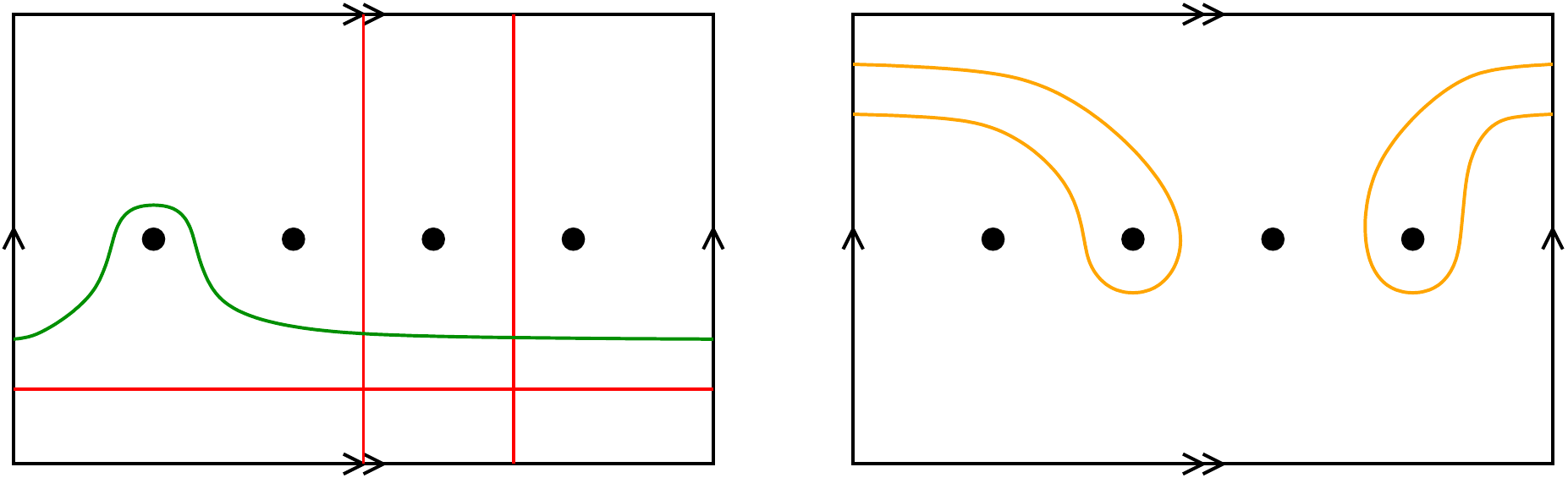}
     \caption{On the left, the curve needed to uniquely determine the curve $\veps_{i}^{+}$ on the right. The case for $\veps_{i}^{-}$ is analogous.}
     \label{fig:Sec3-7Fig0}
 \end{figure}
 
 
 Thus we only need to prove that $\eta_{\beta_{i-1}}^{-1}(\veps_{i}^{+})$, $\eta_{\beta_{i}}(\veps_{i}^{+})$, $\eta_{\beta_{i-1}}(\veps_{i}^{-})$, $\eta_{\beta_{i}}^{-1}(\veps_{i}^{-}) \in \Yf{S}^{3}$.
 
 All of these cases can be uniquely determined by an auxiliary curve $\delta \in \Yf{S}^{2}$ (a different one for each case) and the set $\Cf \backslash \{\alpha_{0}^{i-1}, \alpha_{0}^{i}\}$. See Figure \ref{fig:Sec3-7Fig1} for an example. 
 
 \begin{figure}[ht]
     \centering
     \labellist
     \pinlabel $i$ [br] at 194 120
     \pinlabel $i$ [br] at 675 100
     \endlabellist
     \includegraphics[height=4cm]{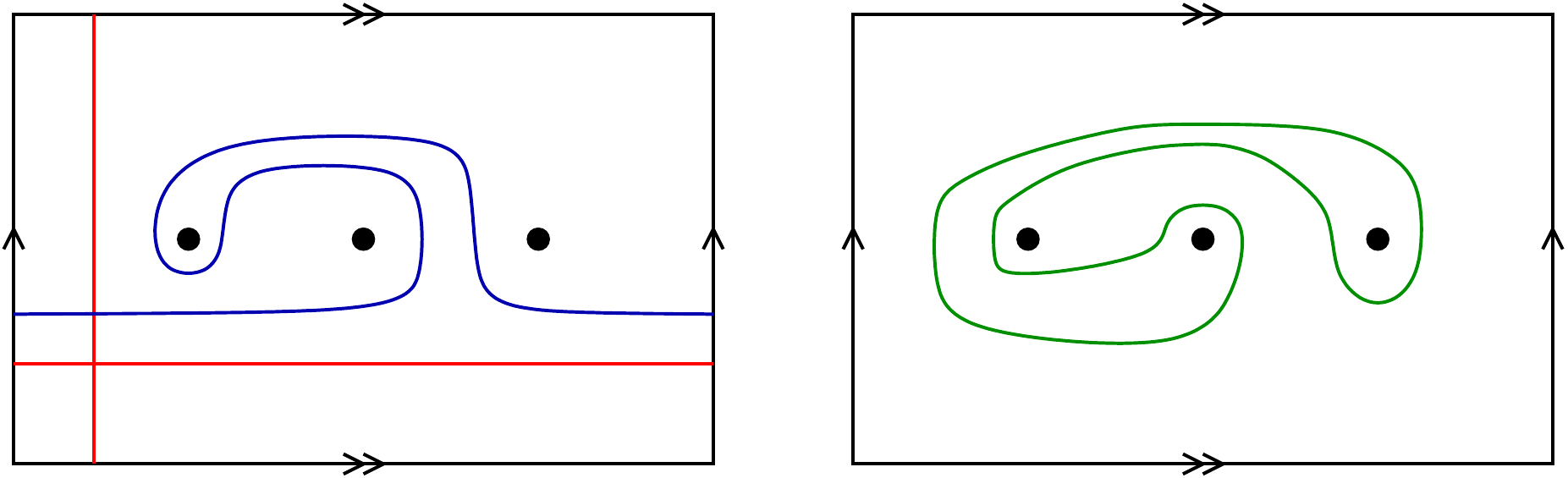}
     \caption{An example of the curve $\eta_{\beta_{i-1}}^{-1}(\veps_{i}^{+})$ (right) being uniquely determined by the set $\Cf \backslash \{\alpha_{0}^{i-1}, \alpha_{0}^{i}\}$ and the curve $\delta$ (left).}
     \label{fig:Sec3-7Fig1}
 \end{figure}
 
 The subsets of $\Yf{S}^{1}$ used to uniquely determine $\delta$ are different for each case and again we make use of the curves of the form $\gamma_{j}^{\pm}$ from the first part of Subsection \ref{subsec3-4}, but they are completely analogous to each other and can be inferred from the following examples and Figure \ref{fig:Sec3-7Fig2}:
 
 \begin{itemize}
  \item If $n = 3$, then we use the following curve to obtain $\eta_{\beta_{i-1}}^{-1}(\veps_{i}^{+})$: $$\delta \ColonEqq \langle \{\alpha_{1}\} \cup \{\veps_{i-1}^{-}\} \cup \{\gamma_{i+1}^{-}\}\rangle.$$
  \item If $n \geq 4$, then we use the following curve to obtain $\eta_{\beta_{i-1}}^{-1}(\veps_{i}^{+})$: $$\delta \ColonEqq \langle \{\alpha_{1}\} \cup \{\veps_{i-1}^{-}\} \cup \{\beta_{i+1}, \ldots, \beta_{i-3}\} \cup \{\gamma_{i+1}^{-}\}\rangle.$$
 \end{itemize}
 
 \begin{figure}[ht]
     \centering
     \labellist
     \pinlabel $i$ [r] at 200 135
     \pinlabel $i$ [r] at 675 135
     \endlabellist
     \includegraphics[height=35mm]{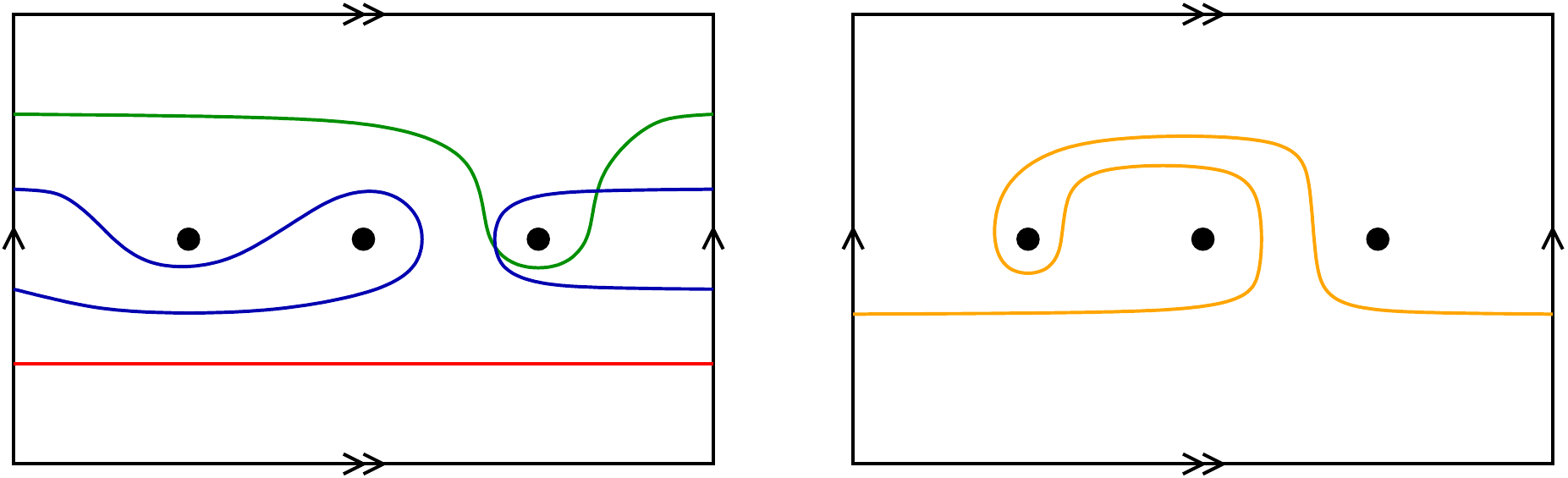}\\[5mm]
     \labellist
     \pinlabel $i$ [r] at 158 130
     \pinlabel $i$ [r] at 635 135
     \endlabellist
     \includegraphics[height=35mm]{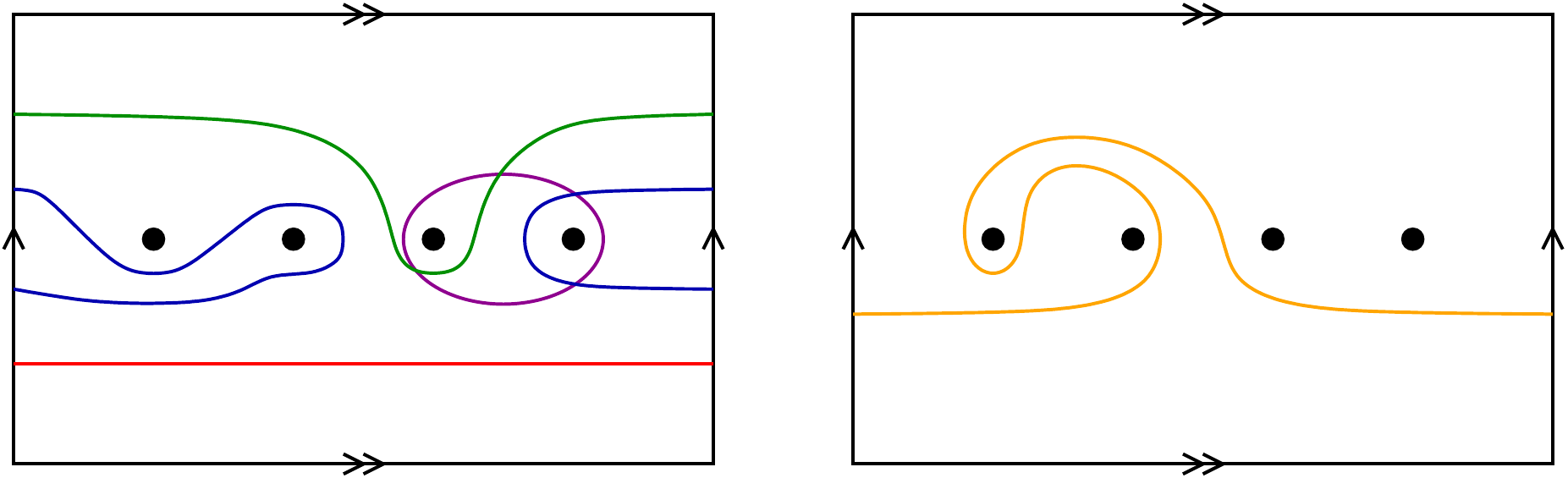}
     \caption{On the left, the curves needed to uniquely determine the curve $\delta$ on the right.}
     \label{fig:Sec3-7Fig2}
 \end{figure}
 
 Then, $\eta_{\Df}^{\pm 1}(\Af) \subset \Yf{S}^{3}$
\subsection{Proof of Theorem \ref{TeoB}}\label{subsec3-8}
 For the purposes of this work, we define the set $\Xf{S}$ as the set $\mathfrak{X}_{1}$ from Section 6 in \cite{Ara2}, which was proved to be rigid (see Proposition 6.2 in \cite{Ara2}). For this definition we need the curves $\gamma_{i}^{\pm}$ from Subsection \ref{subsec3-4} and the following auxiliary curves.
 
 If $n = 3$, for each $0 \leq i <n$ we define the following curves $\gamma_{i,i+1}^{+} \ColonEqq \gamma_{i+2}^{-}$ and $\gamma_{i,i+1}^{-} \ColonEqq \gamma_{i+2}^{+}$.
 
 If $n \geq 4$, for each $0 \leq i <n$ we define the following curves (see Figure \ref{fig:Sec3-8Fig1}): $$\gamma_{i,i+1}^{+} \ColonEqq \langle \{\alpha_{1}\} \cup \{\gamma_{i}^{+}, \gamma_{i+1}^{+}, \gamma_{i+2}^{-}\} \cup \{\beta_{i+2}, \ldots, \beta_{i-2}\}\rangle,$$ $$\gamma_{i,i+1}^{-} \ColonEqq \langle \{\alpha_{1}\} \cup \{\gamma_{i}^{-}, \gamma_{i+1}^{-}, \gamma_{i+2}^{-}\} \cup \{\beta_{i+2}, \ldots, \beta_{i-2}\}\rangle.$$
 
 \begin{figure}[ht]
     \centering
     \labellist
     \pinlabel $i$ [t] at 165 130
     \pinlabel $i$ [t] at 642 130
     \endlabellist
     \includegraphics[height=4cm]{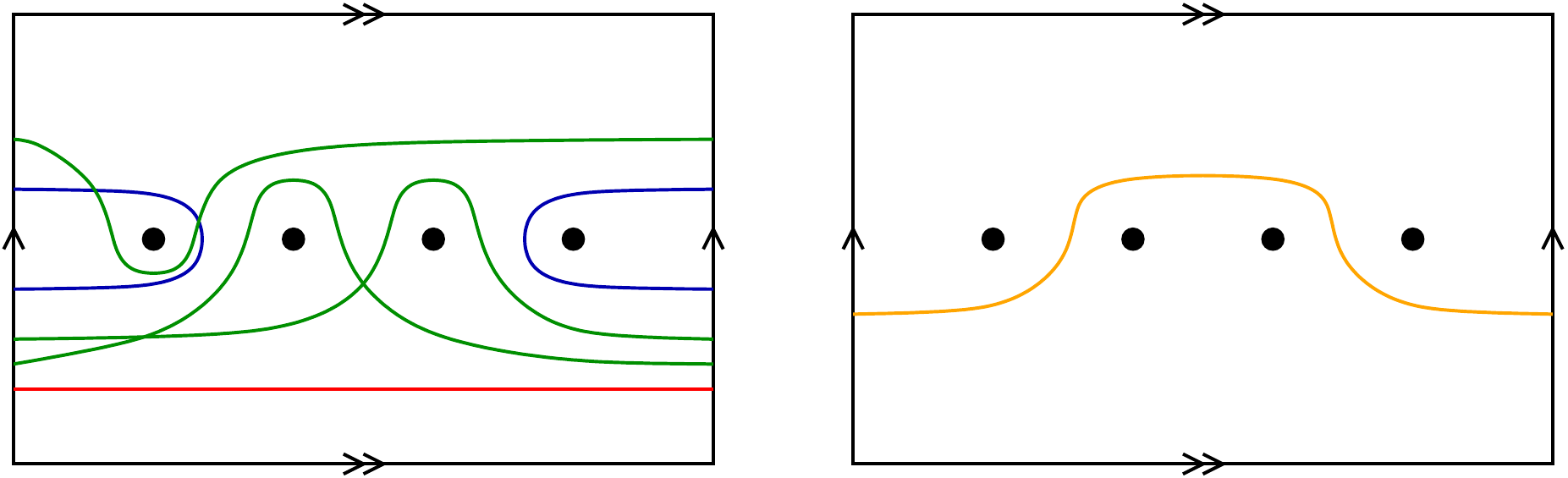}\\[5mm]
     \labellist
     \pinlabel $i$ [b] at 167 145
     \pinlabel $i$ [b] at 642 145
     \endlabellist
     \includegraphics[height=4cm]{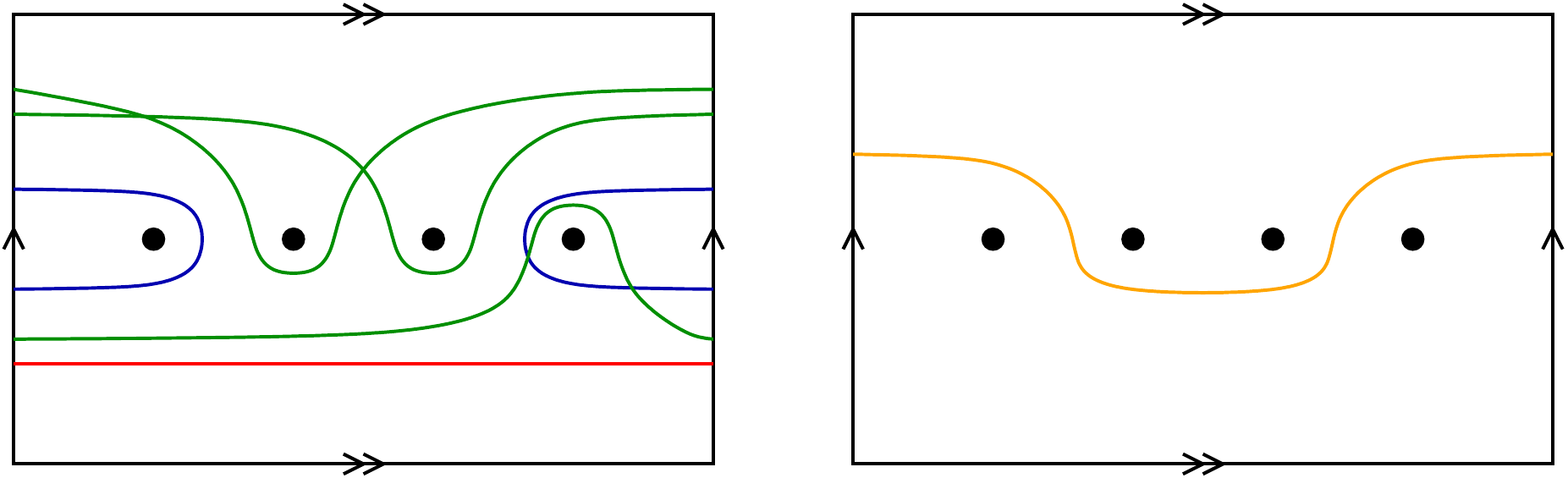}
     \caption{On the left, the curves needed to uniquely determine the curves $\gamma_{i,i+1}^{+}$ (top) and $\gamma_{i,i+1}^{-}$ (bottom) on the right.}
     \label{fig:Sec3-8Fig1}
 \end{figure}
 
 According to the notation above, we can define $\Xf{S}$ as follows: $$\Xf{S} \ColonEqq (\Cf \cup \Df)^{1} \cup \{\gamma_{i}^{\pm}\}_{i = 0}^{n-1} \cup \{\gamma_{i,i+1}^{\pm}\}_{i = 0}^{n-1}.$$
 \begin{proof}[\textbf{Proof of Theorem \ref{TeoB}}]
  For this proof, we proceed as in Subsection \ref{subsec2-7}, so we only need to prove that $\Yf{S} \subset \Xf{S}^{1}$.
  
  Given that $\Cf \cup \Df \subset \Xf{S}$, we only need to prove that $\Af \subset \Xf{S}^{1}$. For this, as in Subsection \ref{subsec3-7} note that for each $0 \leq i <n$ we have that (see Figure \ref{fig:Sec3-7Fig0}): $$\veps_{i}^{+} = \langle \{\gamma_{i}^{+}\} \cup (\Cf \backslash \{\alpha_{i-1}, \alpha_{i}\})\rangle,$$ $$\veps_{i}^{-} = \langle \{\gamma_{i}^{-}\} \cup (\Cf \backslash \{\alpha_{i-1}, \alpha_{i}\})\rangle.$$
  Thus, $\Yf{S} \subset \Xf{S}^{1}$, and therefore $\Xf{S}^{\omega} = \ccomp{S}$.
 \end{proof}
\section{Genus two case}\label{sec4}
 Given that $\ccomp{S_{2,0}}$ is isomorphic to $\ccomp{S_{0,6}}$, Subsections \ref{subsec2-3} and \ref{subsec2-7} prove Theorems \ref{TeoA} and \ref{TeoB} (respectively) for $S_{2,0}$. Thus, in this section we assume that $S = S_{2,n}$ with $n \geq 1$ (then $\kappa(S) \geq 4$).
 
 The structure of this section is as follows: In Subsection \ref{subsec4-1} we define $\Cf$, $\Bf$ and $\Yf{S}$; in Subsection \ref{subsec4-2} we define the sets of auxiliary curves $\Df$ used for the proof of Theorem \ref{TeoA}; in Subsection \ref{subsec4-3} we give the proof of Theorem \ref{TeoA} pending the proof of a key lemma (Lemma \ref{KeyLemag2}); in Subsection \ref{subsec4-4} we define the auxiliary curves $\Ef$ and $\Zf$ used to prove the key lemma; in Subsections \ref{subsec4-5}, \ref{subsec4-6}, \ref{subsec4-7} and \ref{subsec4-8}, we give the proof of the key lemma; finally, in Subsection \ref{subsec4-9} we recall from \cite{Ara2} the definition of $\Xf{S}$ and prove Theorem \ref{TeoB}.
\subsection{The definition of $\Yf{S}$}\label{subsec4-1}
 Let $C = \{\gamma_{1}, \ldots, \gamma_{k}\}$ be a set of curves for $k >1$. We say $C$ is a \textit{chain} if $i(\gamma_{i},\gamma_{j}) = 1$ for $|i-j| =1$, and $\gamma_{i}$ is disjoint from $\gamma_{j}$ otherwise. Note the similarities with an outer chain (defined in Subsection \ref{subsec2-1}). We define the \textit{length} of $C$ as its cardinality. See Figure \ref{fig:Sec4-1Fig1} for an example.
 
 \begin{figure}[ht]
     \centering
     \includegraphics[height=4cm]{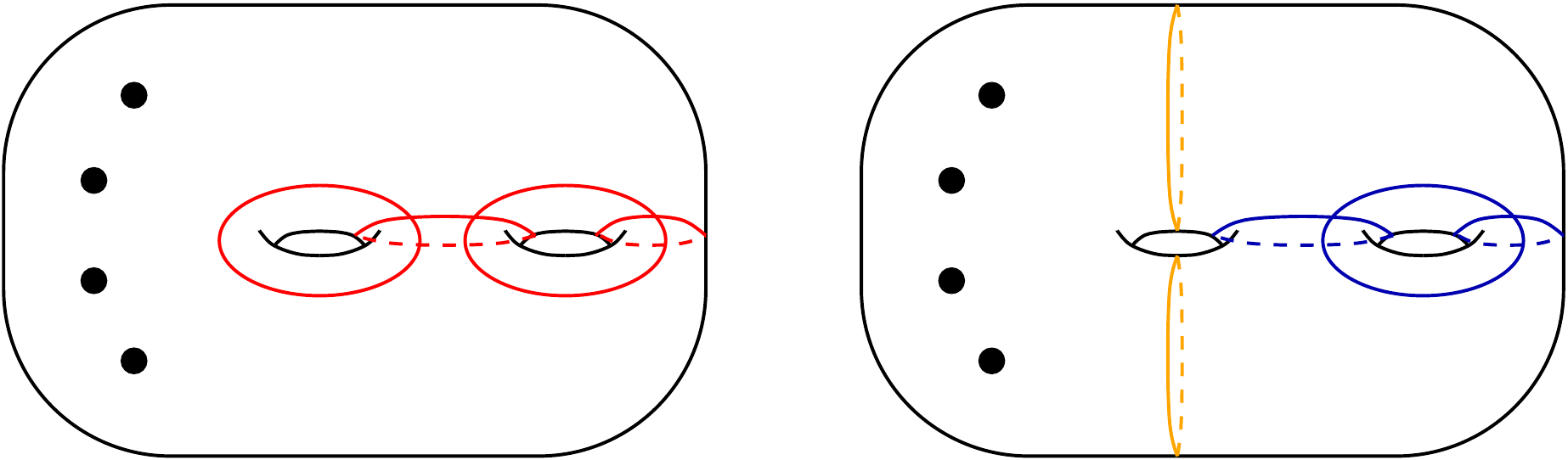}
     \caption{On the left, a chain of length 4; on the right, a chain of length 3 and its corresponding bounding pair.}
     \label{fig:Sec4-1Fig1}
 \end{figure}
 
 Note that if a chain $C$ has length $2j$, any (open) regular neighbourhood of $C$ is homeomorphic to $S_{j,1}$, while if $C$ has length $2j+1$, then any (open) regular neighbourhood of $C$ is homeomorphic to $S_{j,2}$.
 
 Given a chain $C$ of odd length, the \textit{bounding pair} associated to $C$ is defined as the set containing the boundary curves of a closed regular neighbourhood of $C$. See Figure \ref{fig:Sec4-1Fig1} for an example.
 
 Let $\Cf_{0} \ColonEqq \{\alpha_{1}, \ldots, \alpha_{5}\}$ be the chain depicted in Figure \ref{fig:Sec4-1Fig2}, and let $\Cf_{f} \ColonEqq \{\alpha_{0}^{0}, \ldots, \alpha_{0}^{n}\}$ be the multicurve also depicted in Figure \ref{fig:Sec4-1Fig2}. Note that for all $0 \leq i \leq n$ we have that $i(\alpha_{0}^{i},\alpha_{j}) = 1$ if $j = 1, 5$, and $\alpha_{0}^{i}$ is disjoint from $\alpha_{j}$ otherwise. Then, we define the set $\Cf \ColonEqq \Cf_{0} \cup \Cf_{f}$.
 
 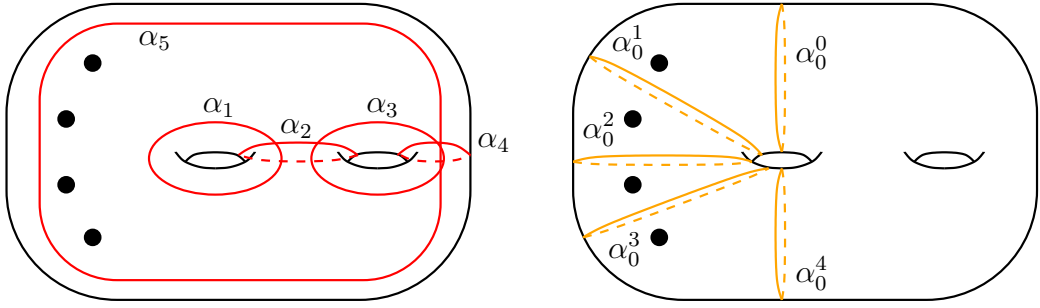
\begin{figure}[ht]
     \centering
     \labellist
     \pinlabel $\alpha_{1}$ [bl] at 165 155
     \pinlabel $\alpha_{2}$ [b] at 250 135
     \pinlabel $\alpha_{3}$ [bl] at 305 155
     \pinlabel $\alpha_{4}$ [bl] at 400 125
     \pinlabel $\alpha_{5}$ [tl] at 110 238
     \pinlabel $\alpha_{0}^{0}$ [l] at 670 215
     \pinlabel $\alpha_{0}^{1}$ [bl] at 515 205
     \pinlabel $\alpha_{0}^{2}$ [bl] at 490 130
     \pinlabel $\alpha_{0}^{3}$ [tl] at 512 65
     \pinlabel $\alpha_{0}^{4}$ [l] at 670 25
     \endlabellist
     \includegraphics[height=4cm]{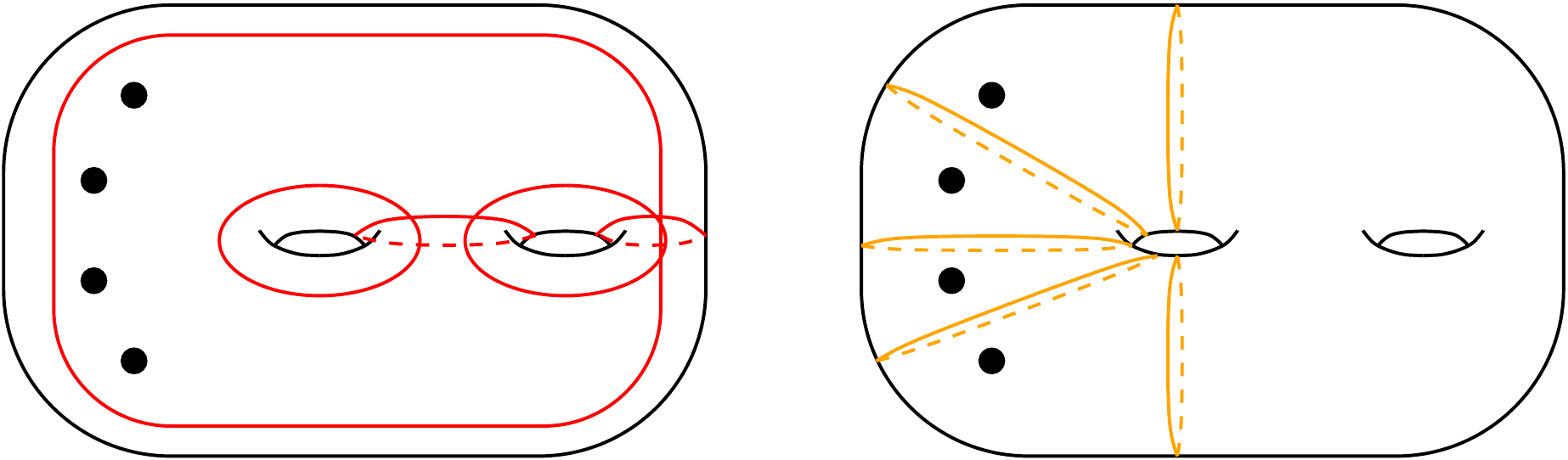}
     \caption{The curves in $\Cf_{0}$ on the right and the curves in $\Cf_{f}$ on the left.}
     \label{fig:Sec4-1Fig2}
 \end{figure}
 
 We can then fix several subsurfaces of $S$ using $\Cf$ as follows (see Figure \ref{fig:Sec4-1Fig3}): 
 \begin{itemize}
  \item Let $\Sigma_{o}^{+}$ be the closed subsurface of $S$ bounded by $\{\alpha_{1},\alpha_{3}, \alpha_{5}\}$ that is not compact, and $\Sigma_{o}^{-}$ be the compact subsurface of $S$ bounded by $\{\alpha_{1},\alpha_{3}, \alpha_{5}\}$.
  \item Define $S_{0}^{+}$ as the compact subsurface of $S$ bounded by $\{\alpha_{0}^{0}, \alpha_{2}, \alpha_{4}\}$, and $S_{0}^{-}$ as the noncompact subsurface of $S$ bounded by $\{\alpha_{0}^{0}, \alpha_{2}, \alpha_{4}\}$.
  \item For each $0 < i \leq n$, let $S_{i}^{+}$ be the closed subsurface of $S$ bounded by $\{\alpha_{0}^{i}, \alpha_{2}, \alpha_{4}\}$ that contains $S_{0}^{+}$. Let $S_{0}^{-}$ be the closed subsurface of $S$ bounded by $\{\alpha_{0}^{i}, \alpha_{2}, \alpha_{4}\}$ that does not contain $S_{0}^{+}$.
 \end{itemize}
 
 \begin{figure}[H]
     \centering
     \labellist
     \pinlabel $\Sigma_{o}^{+}$ at 250 165
     \pinlabel $\Sigma_{o}^{-}$ [bl] at 367 220
     \pinlabel $S_{0}^{+}$ at 780 200
     \pinlabel $S_{0}^{-}$ at 735 60
     \endlabellist
     \includegraphics[height=4cm]{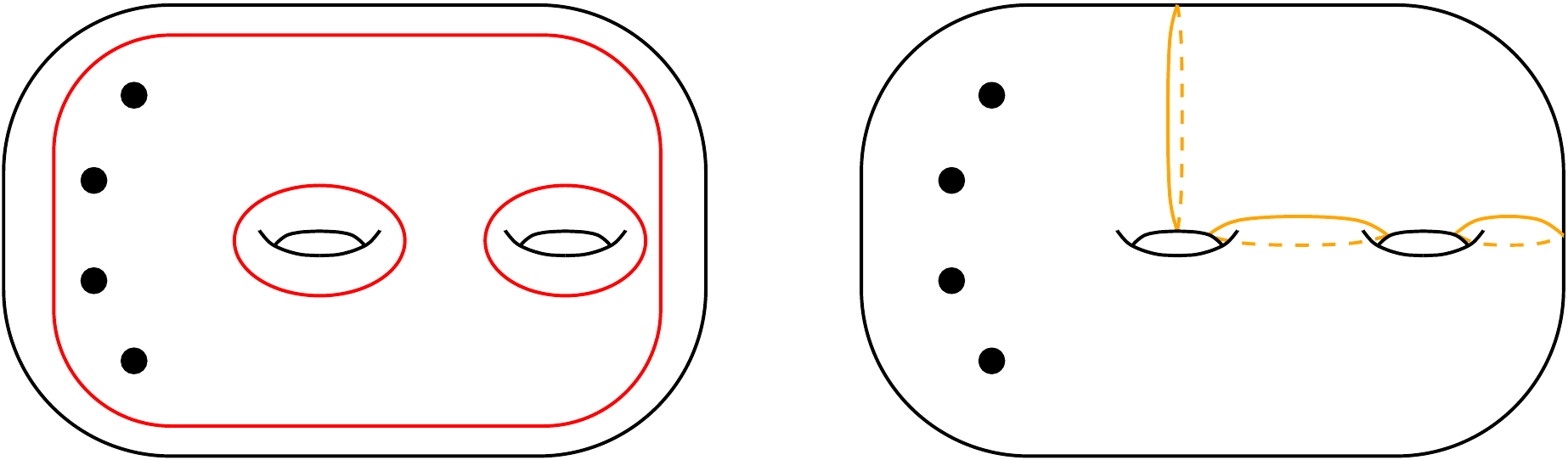}\\[5mm]
     \labellist
     \pinlabel $S_{2}^{+}$ at 250 190
     \pinlabel $S_{2}^{-}$ at 250 70
     \endlabellist
     \includegraphics[height=4cm]{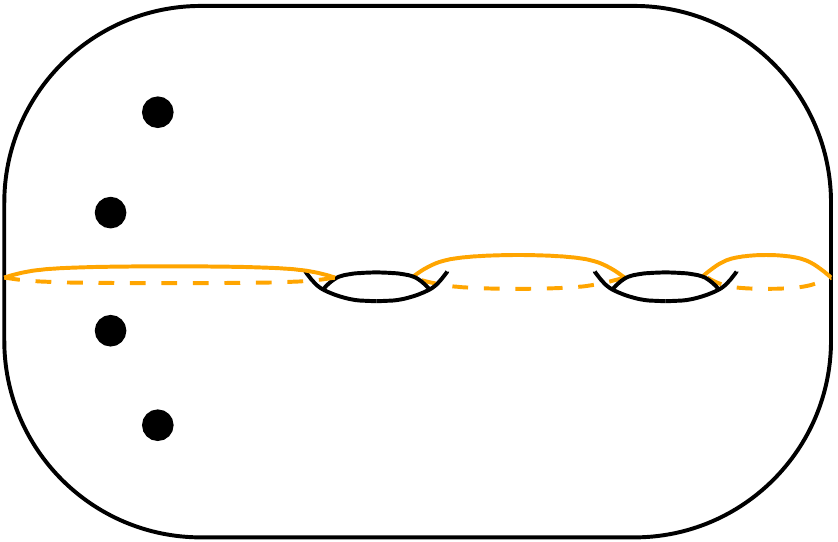}
     \caption{On the top left, the subsurfaces $\Sigma_{o}^{+}$ (punctured pair of pants at the front) and $\Sigma_{o}^{-}$ (pair of pants at the back); on the top right, the subsurfaces $S_{0}^{+}$ (pair of pants at the top right) and $S_{0}^{-}$ (punctured pair of pants at the bottom); on the bottom, the subsurfaces $S_{i}^{+}$ (punctured pair of pants at the top) and $S_{i}^{-}$ (possibly punctured pair of pants at the bottom).}
     \label{fig:Sec4-1Fig3}
 \end{figure}
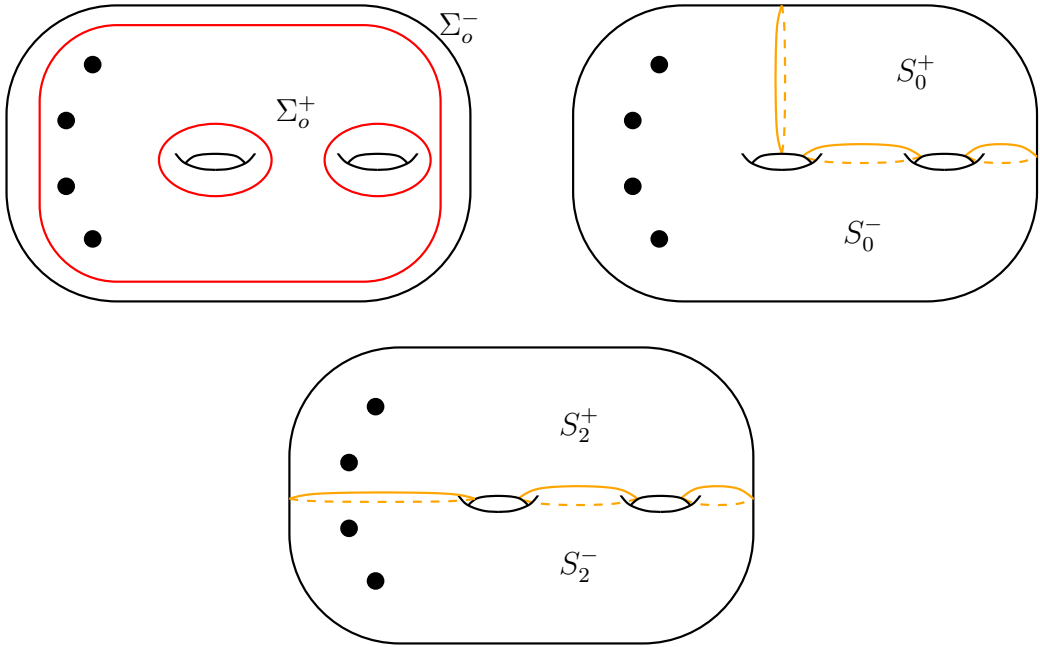
 
 Let $C\subset \Cf$ be a chain of odd length such that $|C \cap \Cf_{f}| \leq 1$, and let $B(C)$ be the bounding pair associated to $C$. Then, note that the elements of $B(C)$ are either contained in $\Int(S_{i}^{+}) \cup \Int(S_{i}^{-})$ for some $i$, or contained in $\Int(\Sigma_{o}^{+}) \cup \Int(\Sigma_{o}^{-})$. We denote by $\beta_{C}^{+}$ the element of $B(C)$ contained in either $\Int(S_{i}^{+})$ for some $i$ or $\Sigma_{o}^{+}$; we denote by $\beta_{C}^{-}$ the element of $B(C)$ contained in either $\Int(S_{i}^{-})$ for some $i$ or $\Sigma_{o}^{-}$. See Figure \ref{fig:Sec4-1Fig4}.
 
 \begin{figure}[ht]
     \centering
     \includegraphics[height=4cm]{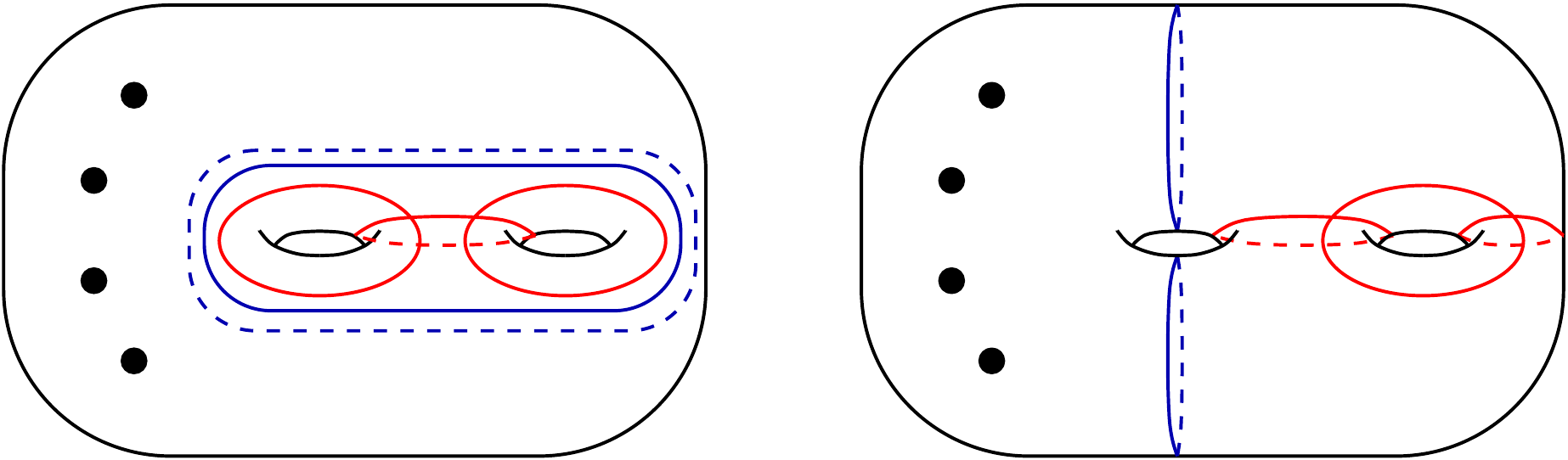}
     \caption{On the left, the set $C = \{\alpha_{1},\alpha_{2},\alpha_{3}\}$ and the curves $\beta_{C}^{+}$ (front) and $\beta_{C}^{-}$ (back); on the right, the set $C= \{\alpha_{2},\alpha_{3},\alpha_{4}\}$ and the curves $\beta_{C}^{+}$ (top) and $\beta_{C}^{-}$ (bottom).}
     \label{fig:Sec4-1Fig4}
 \end{figure}
 
 We define the set $\dsty \Bf \ColonEqq \{\beta_{C}^{+},\beta_{C}^{-} : C \subset \Cf$ is a chain such that $|C \cap \Cf_{f}| \leq 1\}$.
 
 Note that since we are considering only chains that contain at most one element of $\Cf_{f}$, then every curve in $\Bf$ is non-separating.
 
 Finally, we define $$\Yf{S} \ColonEqq \Cf \cup \Bf.$$
\subsection{Auxiliary curves for Theorem \ref{TeoA}}\label{subsec4-2}
 As a difference with the previous sections, here we use the auxiliary curves for the key lemma and the proof of Theorem \ref{TeoA}, instead of using them for the definition of $\Yf{S}$. 
 
 For each $0 < i < n$, we define the following outer curves (see Figure \ref{fig:Sec4-2Fig1}): $$\out{i} = \langle \Cf \backslash \{\alpha_{0}^{i}\}\rangle \in \Cf^{1}.$$
 
 Then, we define the following curve (see Figure \ref{fig:Sec4-2Fig1}): $$\Delta = \langle \Cf_{0} \cup \{\out{j}\}_{j=1}^{n-1}\rangle \in \Cf^{2}.$$
 
 \begin{figure}[ht]
     \centering
     \includegraphics[height=4cm]{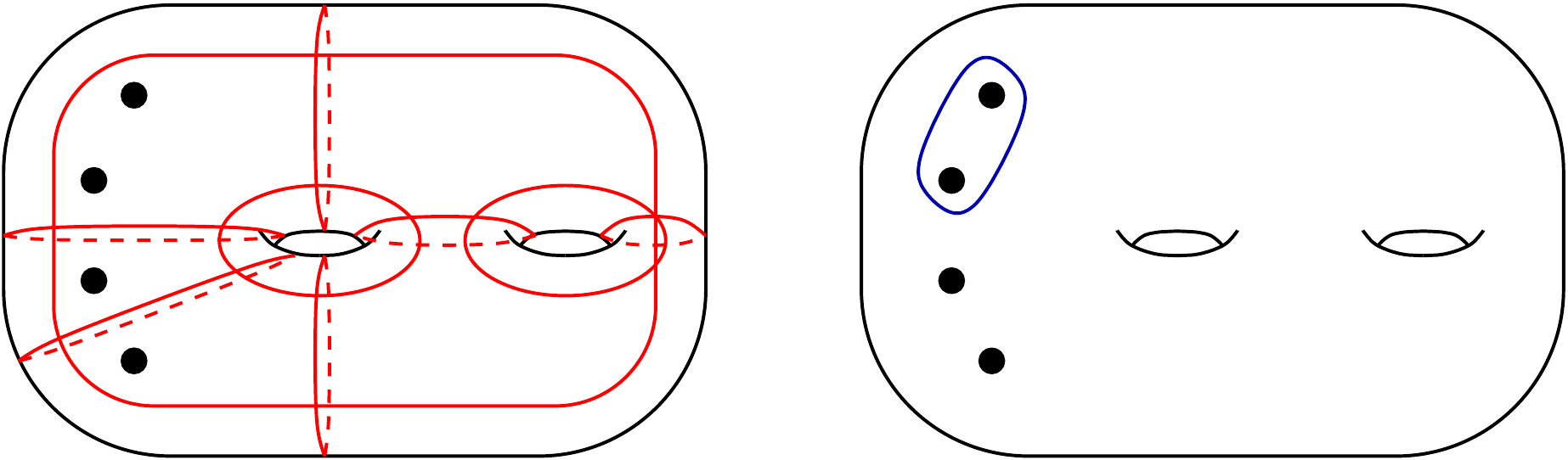}\\[5mm]
     \includegraphics[height=4cm]{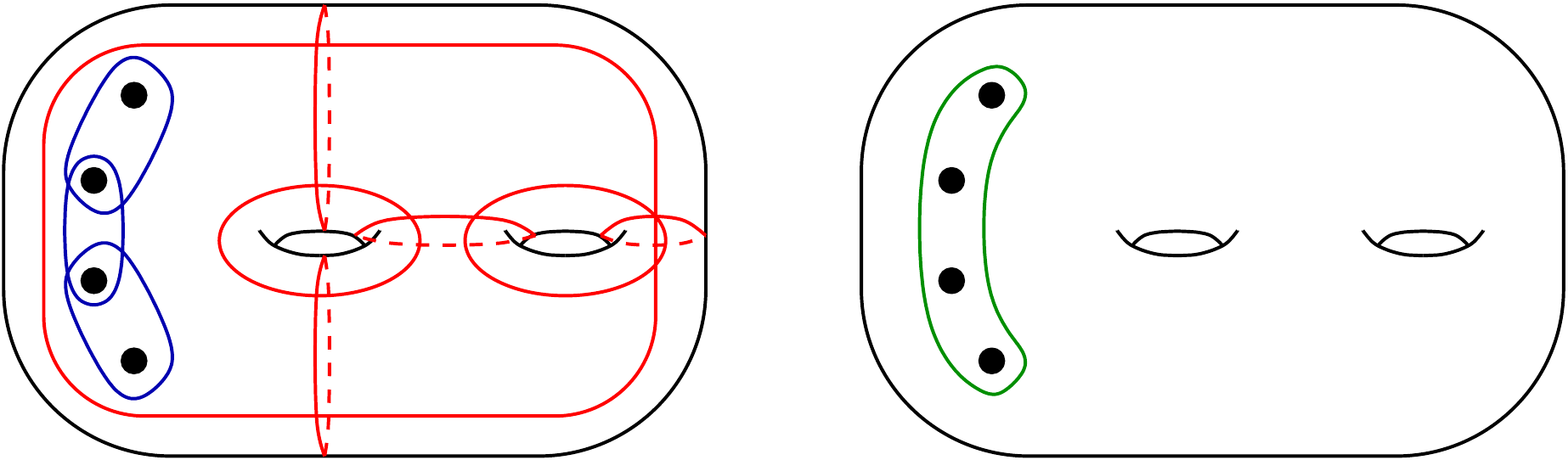}
     \caption{On the left, the curves needed to uniquely determine the curves $\delta_{1}$ (top) and $\Delta$ (bottom).}
     \label{fig:Sec4-2Fig1}
 \end{figure}
 
 With this, we can define the following set $$\Df \ColonEqq \{\out{i}: 0 < i < n\} \cup \{\Delta\} \subset \Cf^{2}$$
 
 Note that if $n = 1$, then $\Df = \varnothing$, and if $n = 2$ then $\Df = \{\out{1}\}$.
 
 
 
\subsection{Proof of Theorem \ref{TeoA}}\label{subsec4-3}
 Similarly to the previous sections, let $\Gf \ColonEqq \{\tau_{\alpha}^{\pm 1} : \alpha \in (\Cf \backslash \{\alpha_{5}\})\} \cup \{\eta_{\out{i}}^{\pm 1} : 0 < i < n\}$. It is a well-known fact (see Corollary 4.15 in \cite{FarbMar}) that $\Gf$ is a symmetric generating set of $\Mod{S}$ (often called the Humphries-Lickorish generating set). As before, if $A$ is a set of curves on $S$, we denote by $\Gf \cdot A$ the following set $\Gf \cdot A \ColonEqq \tau_{(\Cf \backslash \{\alpha_{5}\})}^{\pm 1}(A) \cup \eta_{\{\out{i} : 0 < i <n\}}^{\pm 1}(A)$.
 
 Now, we state the key lemma of this section, whose proof we leave to the following subections (see Subsections \ref{subsec4-5}, \ref{subsec4-6}, \ref{subsec4-7} and \ref{subsec4-8})
 
 \begin{Lema}\label{KeyLemag2}
  Let $\Gf$ be as above. Then $\Gf \cdot \Yf{S} \subset \Yf{S}^{11}$.
 \end{Lema}
 
 Recalling Proposition 3.7 in \cite{JHH1} (restated for our purposes), we obtain an immediate consequence.
 
 \begin{Prop}[3.7 in \cite{JHH1}]\label{Prop3-7}
  Let $h \in \Aut{\ccomp{S}}$ and $Y \subset \ccomp{S}$. If $h(Y) \subset \Yf{S}^{k}$ for some $k \geq 0$, then $h(Y^{m}) \subset \Yf{S}^{k+m}$ for all $m \geq 0$.
 \end{Prop}
 
 \begin{Cor}\label{CorKeyLemag2}
  For $\Gf$ and $\Df$ 
  as above, we have that $\Gf \cdot (\Yf{S} \cup \Df) \subset \Yf{S}^{13}$.
 \end{Cor}
 
 \begin{proof}[\textbf{Proof of Theorem \ref{TeoA}}]
  The proof of Theorem \ref{TeoA} for this case is the same as the one for the case $g =1$, simply substituting ``Lemma \ref{KeyLemag1}'' for ``Corollary \ref{CorKeyLemag2}'', ``Figure \ref{fig:Sec3-3Fig1}'' for ``Figure \ref{fig:Sec4-3Fig1}'', and fixing the rigid expansions to multiples of $13$ instead of $6$.
 \end{proof}
 
 \begin{figure}[ht]
     \centering
     \includegraphics[height=4cm]{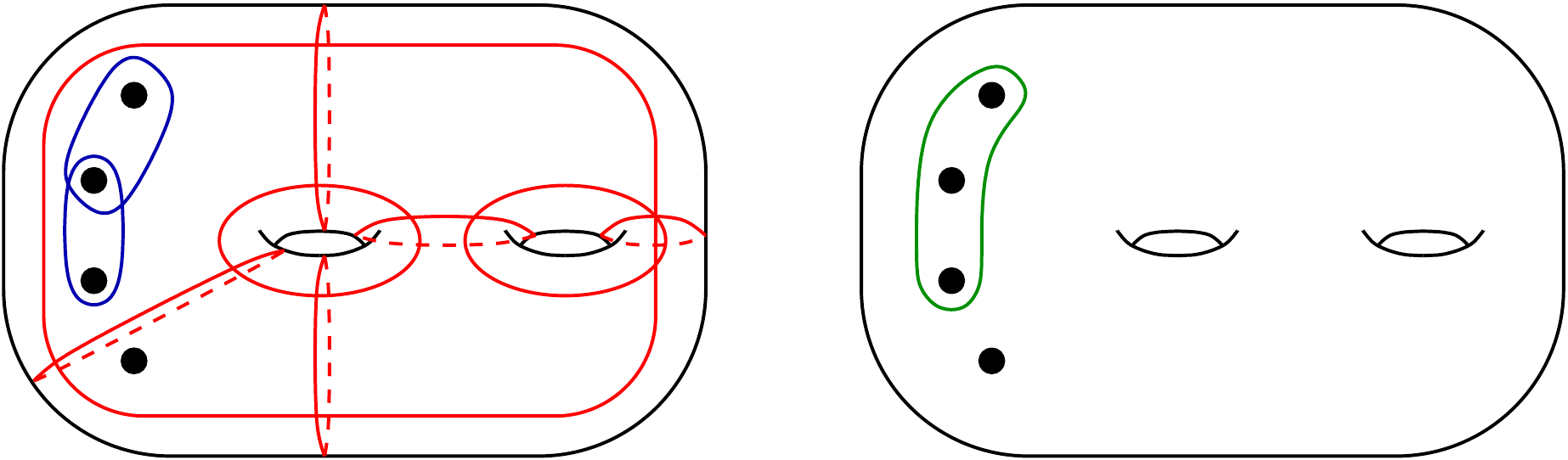}\\[3mm]
     \includegraphics[height=4cm]{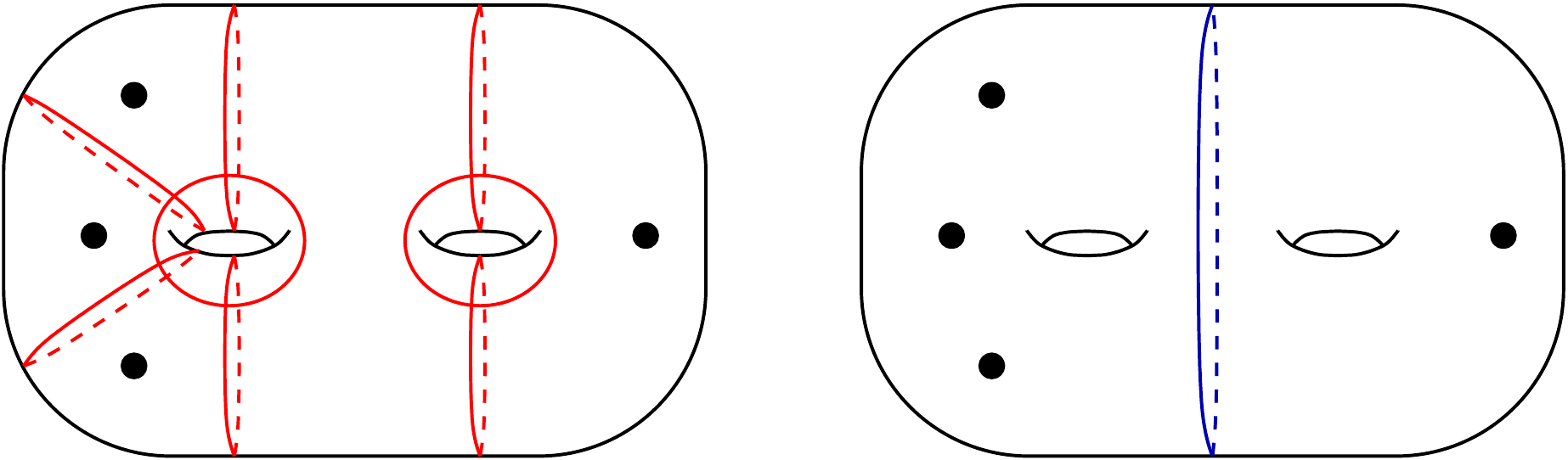}
     \caption{On the left, the curves needed to uniquely determine the non-outer separating curves on the right.}
     \label{fig:Sec4-3Fig1}
 \end{figure}
 
 As before, we divide the proof of Lemma \ref{KeyLemag2} into four claims:\\[0.3cm] 
 \textbf{Claim 1:} $\tau_{(\Cf \backslash \{\alpha_{5}\})}^{\pm 1}(\Cf) \subset \Yf{S}^{8}$.\newline
 \textbf{Claim 2:} $\tau_{(\Cf \backslash \{\alpha_{5}\})}^{\pm 1}(\Bf) \subset \Yf{S}^{11}$.\newline
 \textbf{Claim 3:} $\eta_{\{\out{i} : 0 < i <n\}}^{\pm 1}(\Cf) \subset \Yf{S}^{3}$.\newline
 \textbf{Claim 4:} $\eta_{\{\out{i} : 0 < i <n\}}^{\pm 1}(\Bf) \subset \Yf{S}^{6}$.\\[0.3cm]
 \indent In the following section we define the auxiliary curves used to prove these claims.
\subsection{Auxiliary curves for the claims}\label{subsec4-4}
 There are two types of auxiliary curves that are used throughout this section.
 
 The \textbf{first type} of curves (which are only defined for surfaces with $n \geq 2$ punctures) are defined for each $k = 2,4$ as follows (see Figure \ref{fig:Sec4-4Fig1}): $$\vout{k} \ColonEqq \langle \{\alpha_{0}^{i}\}_{i = 1}^{n-1} \cup (\Cf_{0} \backslash \{\alpha_{k}\})\rangle \in \Cf^{1}.$$
 \begin{figure}[H]
     \centering
     \includegraphics[height=4cm]{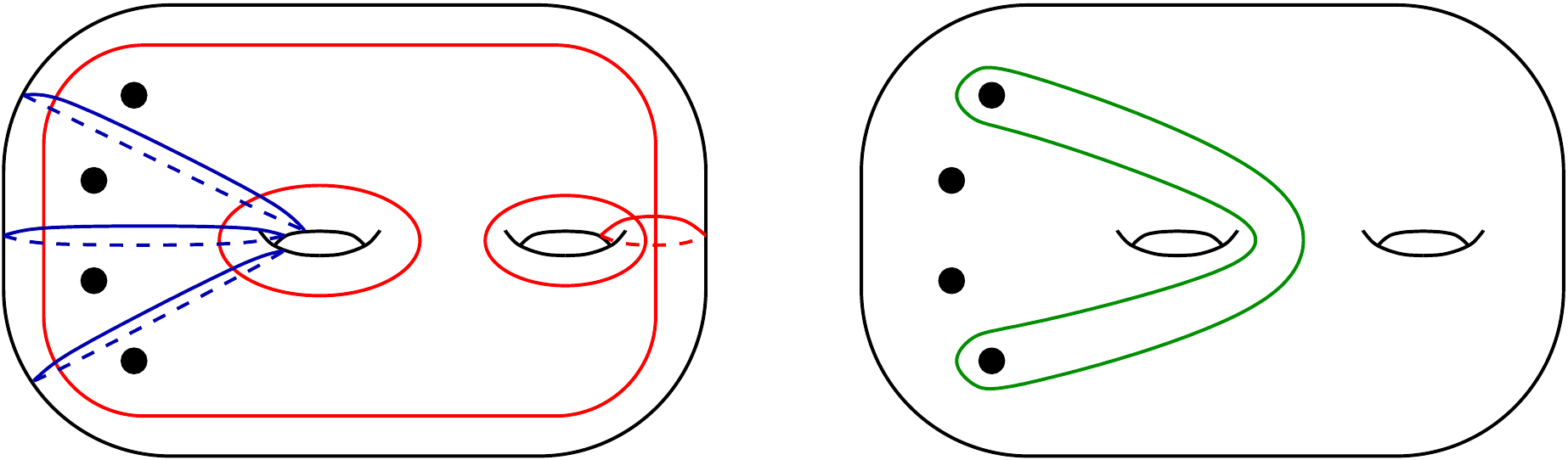}\\[3mm]
     \includegraphics[height=4cm]{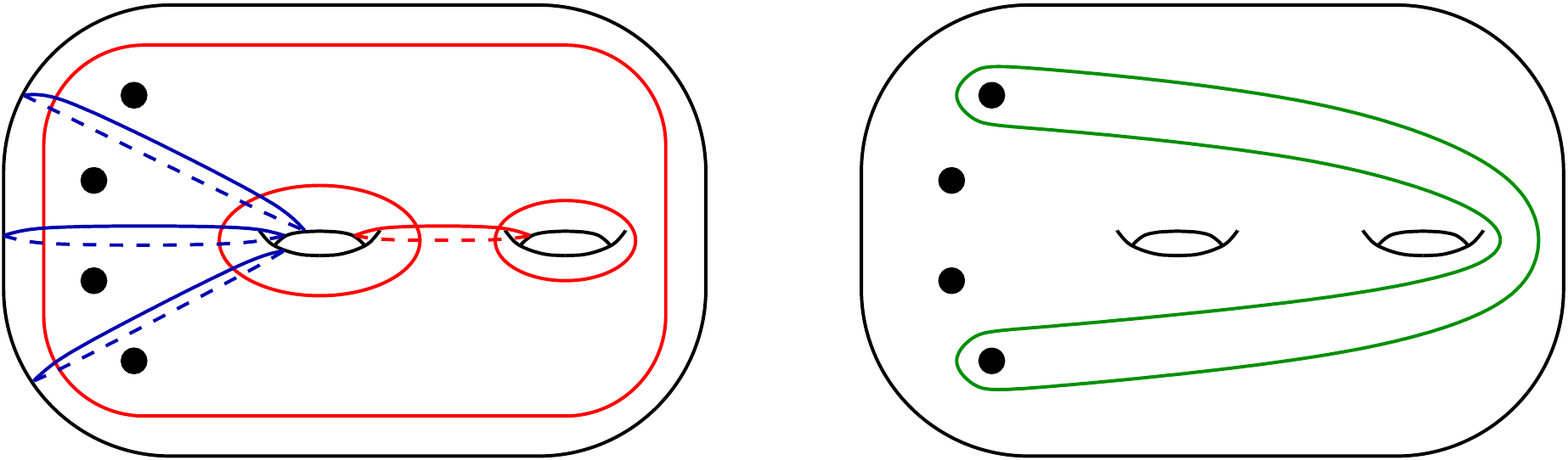}
     \caption{On the left, the curves needed to uniquely determine the curves $\vout{2}$ (top) and $\vout{4}$, on the right.}
     \label{fig:Sec4-4Fig1}
 \end{figure}
 %
 
 To define the second type of curve we must first define for each $0 \leq j \leq n$ the following chain: $$C_{j} \ColonEqq \{\alpha_{0}^{j},\alpha_{1},\alpha_{2}\}.$$
 
 Now, the \textbf{second type} of curve is defined for $i = 0$ as follows: $$\zeta^{0} \ColonEqq \langle \{\alpha_{1},\alpha_{3}, \alpha_{5}\} \cup \{\beta_{C_{j}}^{+}: 0 < j \leq n\}\rangle,$$ for each $0 < i < n+1$ as follows: $$\zeta^{i} = \langle \{\alpha_{1},\alpha_{3}, \alpha_{4}, \alpha_{5}\} \cup \{\beta_{C_{j}}^{+}: j < i\} \cup \{\beta_{C_{k}}^{-}: i \leq k\}\rangle,$$ and finally for $i = n+1$ as follows: $$\zeta^{n+1} \ColonEqq \langle \{\alpha_{1},\alpha_{3}, \alpha_{5}\} \cup \{\beta_{C_{j}}^{-}: 0 \leq j < n\}\rangle.$$ See Figure \ref{fig:Sec4-4Fig2} for examples.
 
 \begin{figure}[H]
     \centering
     \includegraphics[height=38mm]{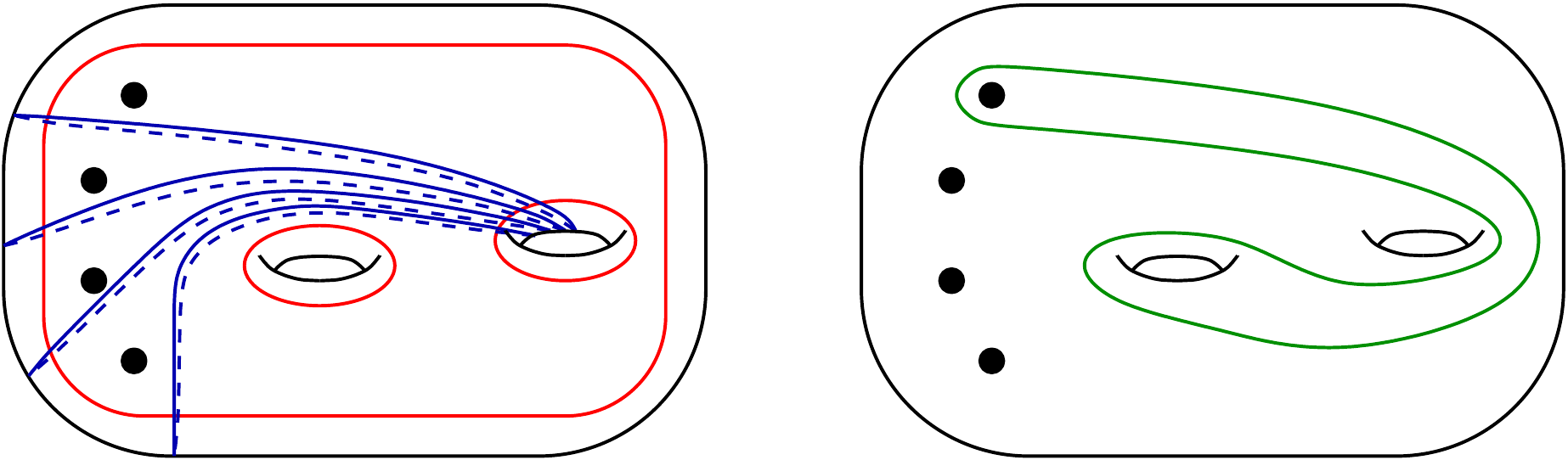}\\[3mm]
     \includegraphics[height=38mm]{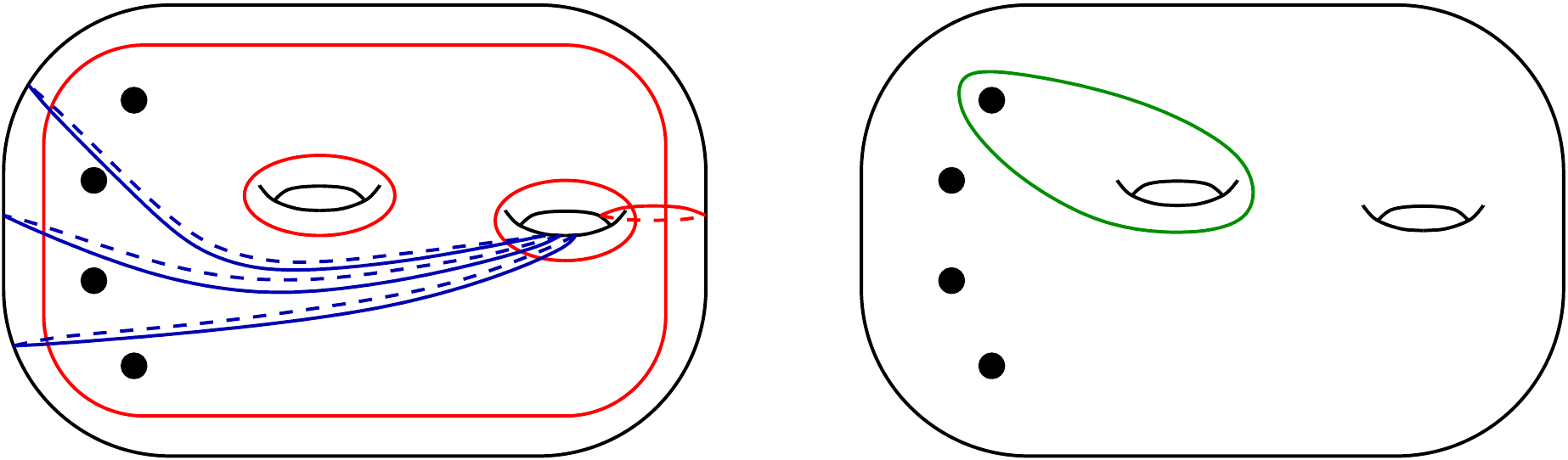}\\[3mm]
     \includegraphics[height=38mm]{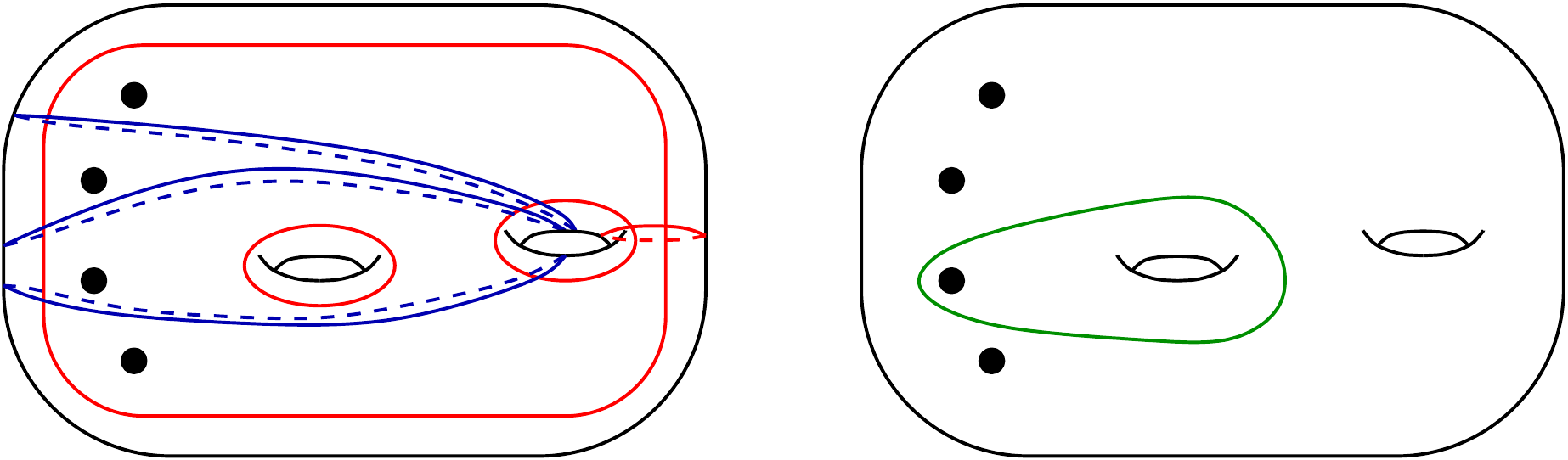}\\[3mm]
     \includegraphics[height=38mm]{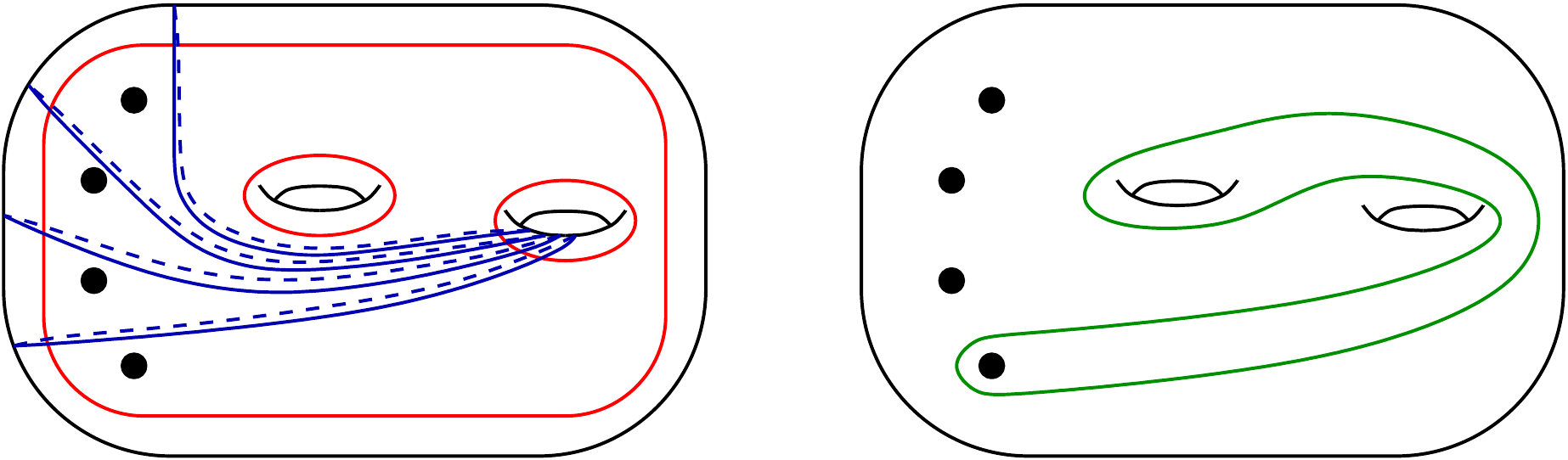}
     \caption{On the left, the curves needed to uniquely determine the curves $\zeta^{0}$ (top), $\zeta^{1}$ (upper middle), $\zeta^{3}$ (lower middle) and $\zeta^{n+1}$, on the right.}
     \label{fig:Sec4-4Fig2}
 \end{figure}
 
 Note that for all $0 \leq i \leq n+1$ we have that $\zeta^{i} \in \Yf{S}^{1}$.
 
 Finally, we define the sets $$\Ef := \left\{\begin{array}{cl}
    \Df  & \text{if $n = 1$}, \\
    \Df \cup \{\veps_{2}, \veps_{4}\}  & 
 \end{array} \right.$$ (thus $\Ef \subset \Cf^{2}$), and $$\Zf := \{\zeta^{j}\}_{j=0}^{n+1}$$ (thus $\Zf \subset \Yf{S}^{1}$).
 

\subsection{Proof of Claim 1: $\tau_{(\Cf \backslash \{\alpha_{5}\})}^{\pm 1}(\Cf) \subset \Yf{S}^{8}$}\label{subsec4-5}
 Given that $S$ has genus 2, the strategy used in \cite{JHH1} for genus at least 3 (find two ``standard'' cases of the claim and then use the action of a subgroup of $\Mod{S}$) is actually more complicated than listing exactly each case of $\tau_{\alpha}^{\pm 1}(\alpha^{\prime})$ for $\alpha, \alpha^{\prime} \in \Cf$ (there is too little space in a genus 2 surface). Also, using the fact that $\tau_{\alpha}^{\pm 1}(\alpha^{\prime}) = \alpha^{\prime}$ if and only if $i(\alpha,\alpha^{\prime}) = 0$, and recalling Remark \ref{RemSymtau}, we only need to prove that $\tau_{\alpha_{i}}^{\pm 1}(\alpha_{i-1}) \in \Yf{S}^{k}$ for $0 \leq i \leq 5$ (modulo 6), where $\alpha_{0} = \alpha_{0}^{i}$ for each case of $0 \leq i \leq n$ and $k \leq 8$.
 
 The \textbf{general strategy} to prove that $\tau_{\alpha_{i}}(\alpha_{i-1}) \in \Yf{S}^{8}$ (and $\tau_{\alpha_{i}}^{- 1}(\alpha_{i-1}) \in \Yf{S}^{8}$) is the following:
 \begin{itemize}
     \item For $n = 1$, we first prove that $\tau^{\pm 1}_{\alpha_{5}}(\alpha_{4}) \in \Yf{S}^{2}$. We do this is a similar fashion as in \cite{JHH1}. Then, for $i<5$ we use elements in $\Yf{S}$ to fill the complement of the one-holed torus induced by $\alpha_{i}$ and $\alpha_{i-1}$, and then use $\tau^{\mp 1}_{\alpha_{i+1}}(\alpha_{i})$ to uniquely determine $\tau^{\pm 1}_{\alpha_{i}}(\alpha_{i-1})$. This yields that $\tau^{\pm 1}_{\alpha_{i}}(\alpha_{i-1}) \in \Yf{S}^{7}$ for all $i$.
     \item For $n \geq 2$, the strategy is almost the same with the exception of needing particular auxiliary curves and $\Df$, hence getting that $\tau^{\pm 1}_{\alpha_{5}}(\alpha_{4}) \in \Yf{S}^{3}$. This causes that in the end $\tau^{\pm 1}_{\alpha_{i}}(\alpha_{i-1}) \in \Yf{S}^{8}$ for all $i$.
 \end{itemize}
 
 \begin{Rem}\label{Remg2-TCE}
  Note that, since $\Ef \subset \Cf^{2}$, due to Proposition \ref{Prop3-7} we have that $$\tau_{(\Cf \backslash \{\alpha_{5}\})}^{\pm 1}(\Ef) \subset \Yf{S}^{10}.$$
 \end{Rem}
 
 Now, for the purpose of facilitating the exposition to the reader, we illustrate the case $n=1$ and for the case $n\geq 2$ we introduce the auxiliary curves and instruct the necessary changes from the previous case.
 
\subsubsection{Case $n = 1$}
 First, let $\gamma^{\pm}$ be the following curves (see Figure \ref{fig:Sec4-5Fig1}): $$\gamma^{+} = \langle \{\alpha_{0}^{0}, \alpha_{0}^{1}, \alpha_{3}\} \cup \{\beta_{\{\alpha_{5}, \alpha_{0}^{1}, \alpha_{1}\}}^{+}\} \rangle \in \Yf{S}^{1},$$ $$\gamma^{-} = \langle \{\alpha_{0}^{0}, \alpha_{0}^{1}, \alpha_{3}\} \cup \{\beta_{\{\alpha_{5}, \alpha_{0}^{0}, \alpha_{1}\}}^{+}\} \rangle \in \Yf{S}^{1}.$$
 
 \begin{figure}[ht]
     \centering
     \includegraphics[height=35mm]{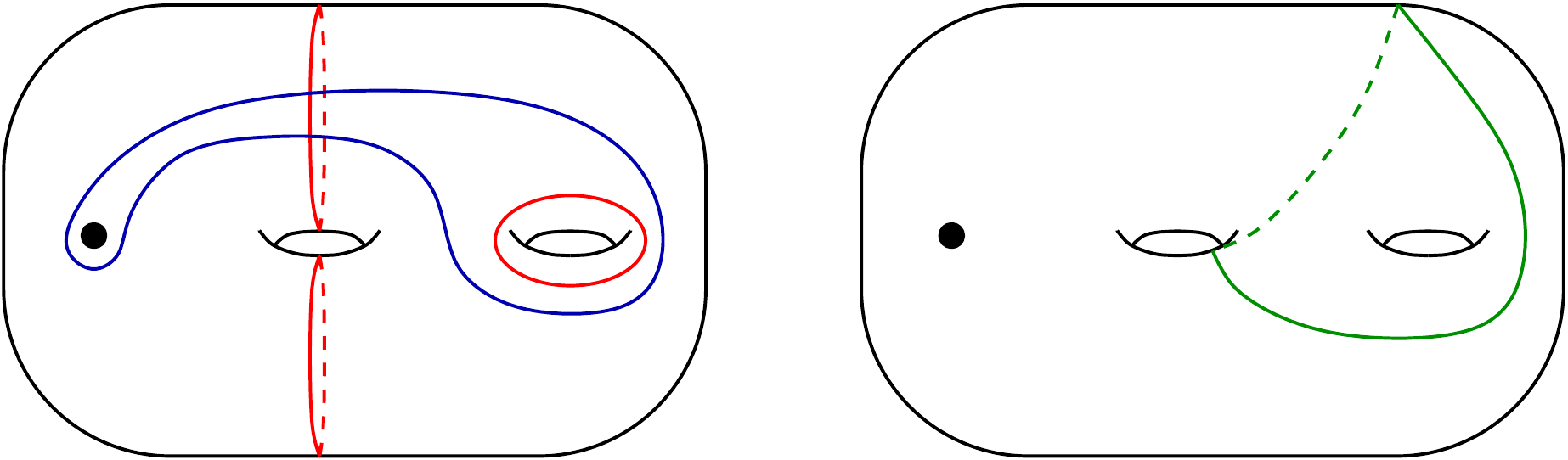}\\[3mm]
     \includegraphics[height=35mm]{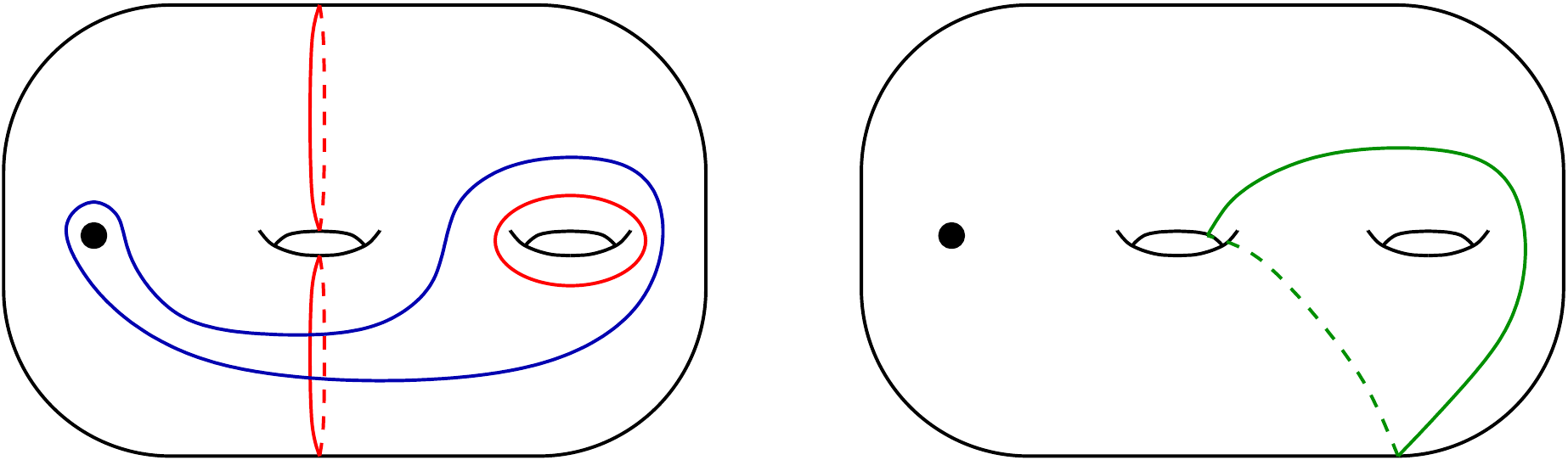}
     \caption{On the left, the curves needed to uniquely determine the curves $\gamma^{+}$ (top) and $\gamma^{-}$ (bottom), on the right.}
     \label{fig:Sec4-5Fig1}
 \end{figure}
 
 With these curves we get (see Figure \ref{fig:Sec4-5Fig2}) $$\tau^{\pm 1}_{\alpha_{5}}(\alpha_{4}) = \langle \{\alpha_{1}, \alpha_{2}\} \cup \{\beta_{\{\alpha_{3},\alpha_{4},\alpha_{5}\}}^{+}\} \cup \{\gamma^{\pm}\} \rangle \in \Yf{S}^{2}.$$
 
 \begin{figure}[ht]
     \centering
     \includegraphics[height=35mm]{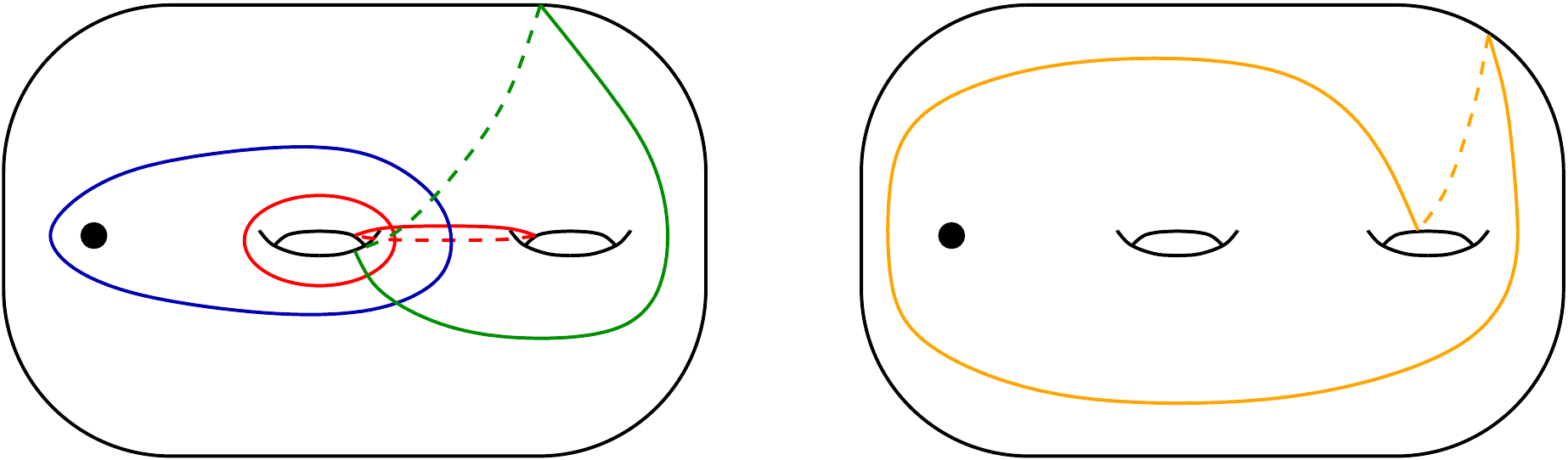}
     \caption{On the left, the curves needed to uniquely determine the curve $\tau_{\alpha_{5}}(\alpha_{4})$, on the right.}
     \label{fig:Sec4-5Fig2}
 \end{figure}
 
 For $\tau^{\pm 1}_{\alpha_{i}}(\alpha_{i-1})$ with $i < 5$, let $A$ be as in Table \ref{tab:ConjuntoA} and we get $$\tau^{\pm 1}_{\alpha_{i}}(\alpha_{i-1}) = \langle A \cup \{\tau^{\mp 1}_{\alpha_{i+1}}(\alpha_{i})\} \rangle \in \Yf{S}^{j},$$ where $i + j = 7$ and $\alpha_{0} = \alpha_{0}^{k}$ for each $k = 0, 1$. See Figure \ref{fig:Sec4-5Fig3} for examples.
 
 \begin{table}[H]
     \centering
     \begin{tabular}{|c|l|}\hline
        i  & $A$ \\\hline
        4  & $\{\alpha_{0}^{0}, \alpha_{0}^{1}, \alpha_{1}\}$ \\\hline
        3  & $\{\alpha_{0}^{0}, \alpha_{0}^{1}, \alpha_{5}\}$ \\\hline
        2  & $\{\alpha_{4}, \alpha_{5}, \beta_{\{\alpha_{1}, \alpha_{2}, \alpha_{3}\}}^{+}\}$ \\\hline
        1  & $\{\alpha_{3}, \alpha_{4}, \beta_{\{\alpha_{5}, \alpha_{0}^{0}, \alpha_{1}\}}^{+}\}$ for $k = 0$, \\\hline
          & $\{\alpha_{3}, \alpha_{4}, \beta_{\{\alpha_{1}, \alpha_{0}^{1}, \alpha_{5}\}}^{+}\}$ for $k = 1$, \\\hline
        0  & $\{\alpha_{2}, \alpha_{3}, \beta_{\{\alpha_{1},\alpha_{0}^{0}, \alpha_{5}\}}^{+}\}$ for $k = 0$,\\\hline
          & $\{\alpha_{2}, \alpha_{3}, \beta_{\{\alpha_{1},\alpha_{0}^{1}, \alpha_{5}\}}^{+}\}$ for $k = 1$.\\\hline
     \end{tabular}
     \caption{The set $A$ for uniquely determining $\tau^{\pm 1}_{\alpha_{i}}(\alpha_{i-1})$ for $n = 1$.}
     \label{tab:ConjuntoA}
 \end{table}
 
 \begin{figure}[H]
     \centering
     \includegraphics[height=4cm]{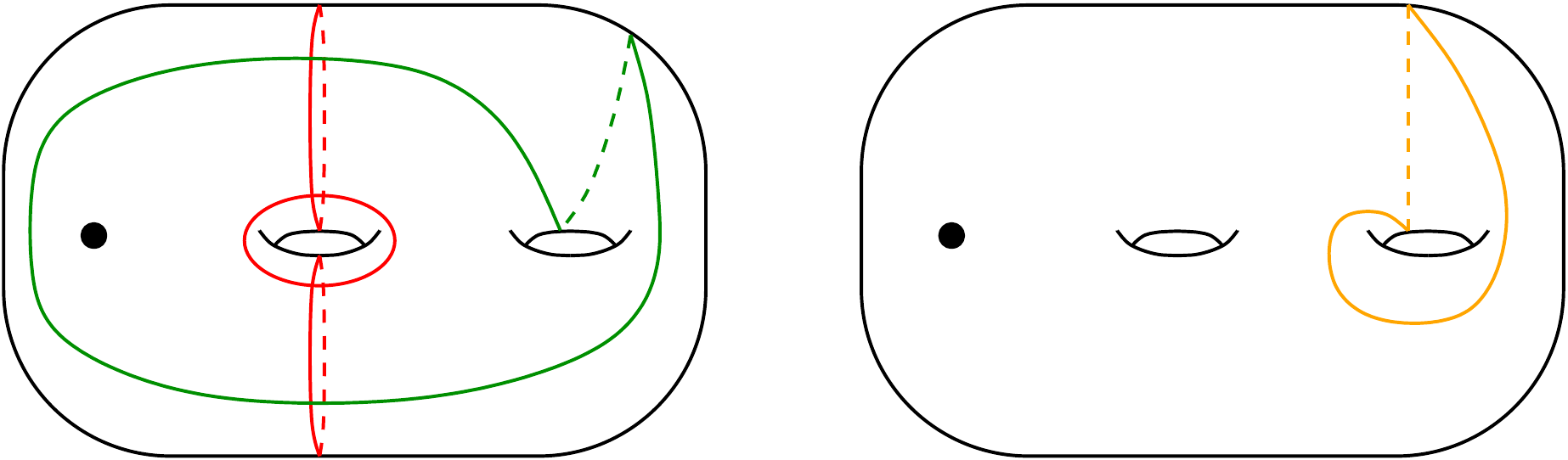}\\[5mm]
     \includegraphics[height=4cm]{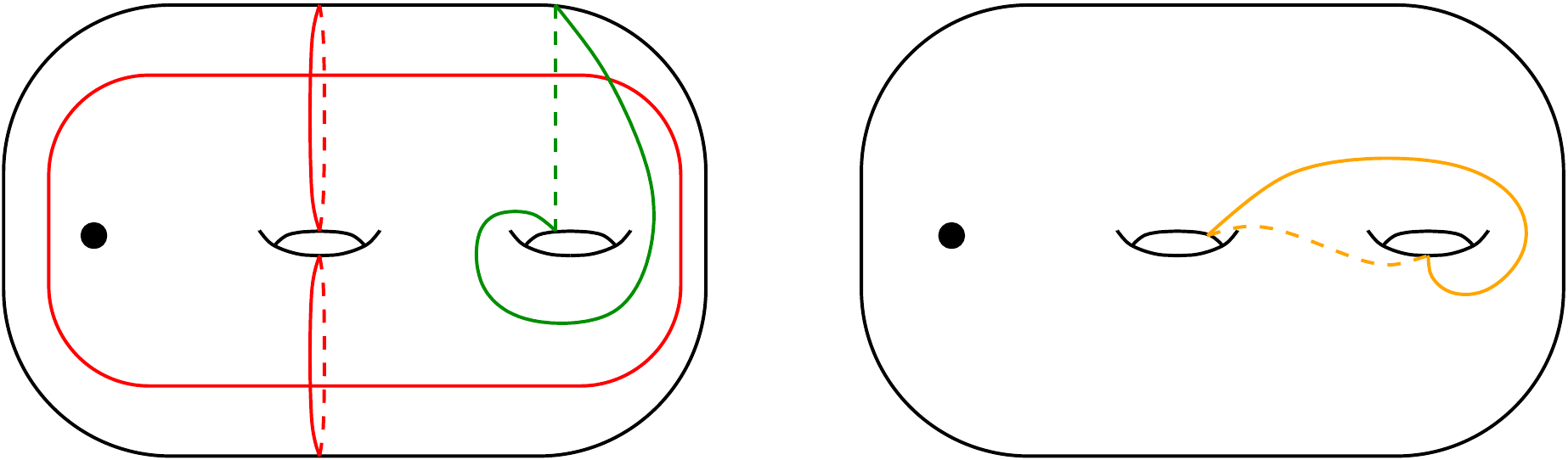}
     \caption{On the left, the curves needed to uniquely determine the curves $\tau^{-1}_{\alpha_{4}}(\alpha_{3})$ (top) and $\tau_{\alpha_{3}}(\alpha_{2})$ (bottom), on the right.}
     \label{fig:Sec4-5Fig3}
 \end{figure}
 
\subsubsection{Case $n \geq 2$}
 For this case we first define the following auxiliary curves. For let $\Delta_{1,n} = \Delta \in \Yf{S}^{2}$ (see Subsection \ref{subsec4-2}); then for $1 < i < n$ we define $\Delta_{1,i}$ and $\Delta_{i,n}$ as in Table \ref{tab:AuxDeltain} (see Figure \ref{fig:Sec4-5Fig4}).
 \begin{table}[ht]
     \centering
     \begin{tabular}{|c|l|}\hline
        $\Delta_{1,2}$  & $\delta_{1} \in \Cf^{1}$ \\\hline
        $\Delta_{n-1,n}$  & $\delta_{n-1} \in \Cf^{1}$\\\hline
        $\Delta_{1,i}$ & $\langle (\Cf \backslash \{\alpha_{0}^{1}, \ldots, \alpha_{0}^{i-1}\}) \cup \{\delta_{1}, \ldots \delta_{i-1}\}\rangle \in \Yf{S}^{2}$\hspace{2mm} for $2 < i < n-1$,\\\hline
        $\Delta_{i,n}$ & $\langle (\Cf \backslash \{\alpha_{0}^{i}, \ldots, \alpha_{0}^{n-1}\}) \cup \{\delta_{i}, \ldots \delta_{n-1}\}\rangle \in \Yf{S}^{2}$\hspace{2mm} for $2 < i < n-1$,\\\hline
     \end{tabular}
     \caption{The auxiliary curves $\Delta_{1,i}$ and $\Delta_{i,n}$.}
     \label{tab:AuxDeltain}
 \end{table}
 Note that $\Df_{1,i}, \Df_{i,n} \subset \Yf{S}^{2}$.
 
 \begin{figure}[ht]
     \centering
     \includegraphics[height=4cm]{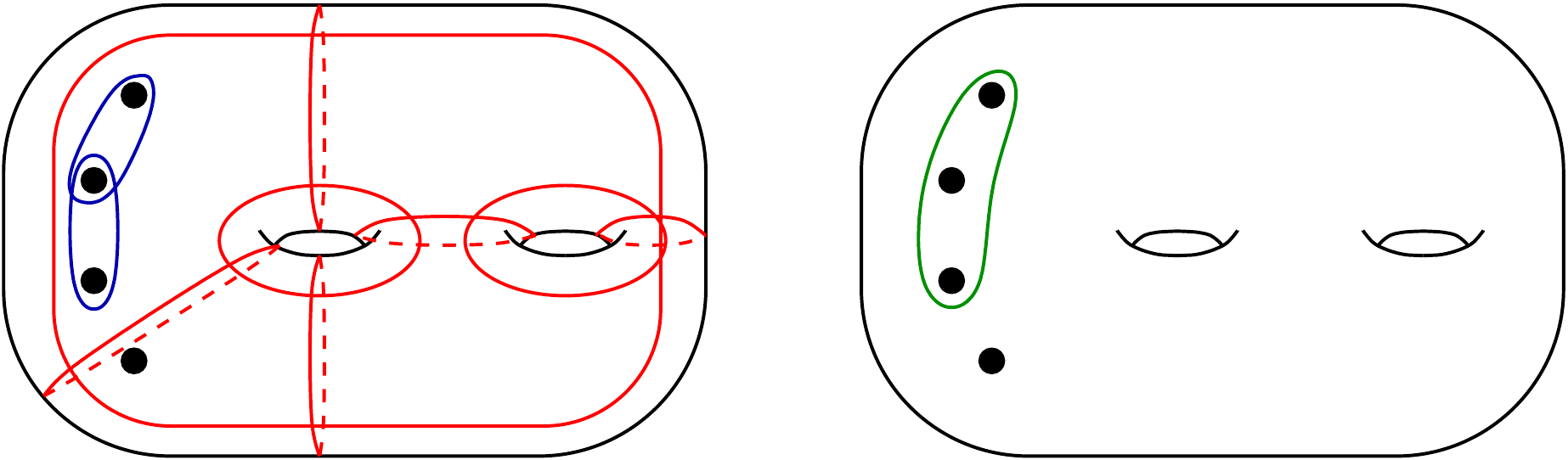}\\[5mm]
     \includegraphics[height=4cm]{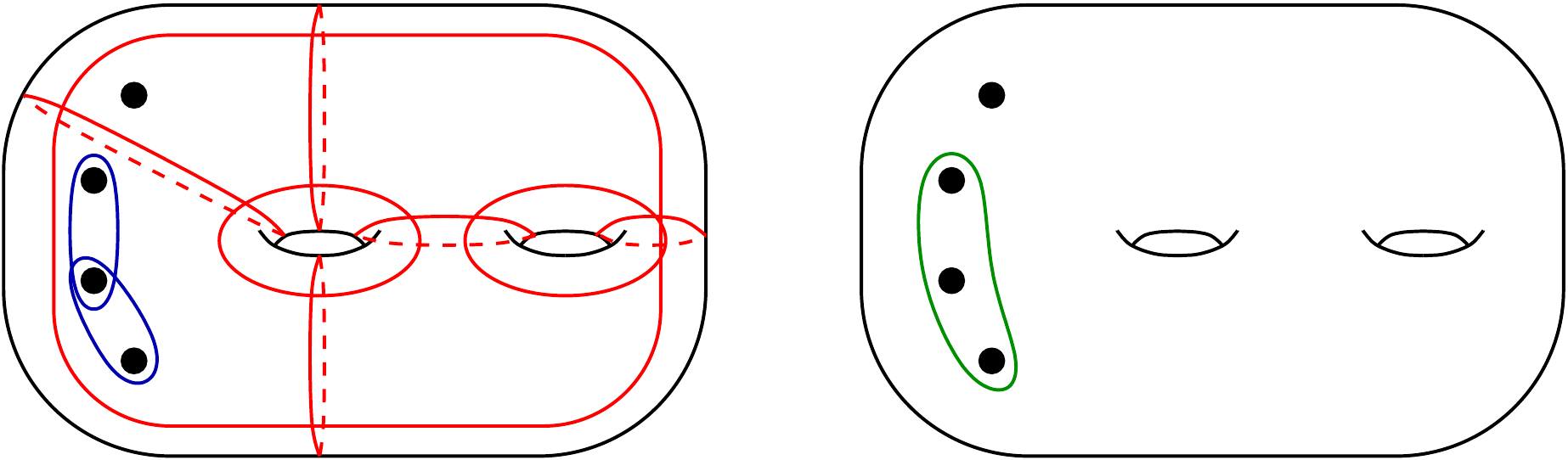}
     \caption{On the left, the curves needed to uniquely determine the curves $\Delta_{1,3}$ (top) and $\Delta_{2,4}$ (bottom).}
     \label{fig:Sec4-5Fig4}
 \end{figure}
 
 With these curves, we define the sets 
 $$\Df_{1,i} := \left\{
 \begin{array}{cl}
    \varnothing  & \text{if } i = 1, \\
    \{\delta_{1}\}  & \text{if } i = 2,\\
    \{\delta_{1}, \ldots, \delta_{i-1}\} \cup \{\Delta_{1,i}\} & \text{if } i \geq 3.
 \end{array}\right.$$
 $$\Df_{i,n} := \left\{
 \begin{array}{cl}
    \varnothing  & \text{if } i = n, \\
    \{\delta_{i-1}\}  & \text{if } i = n-1,\\
    \{\delta_{i}, \ldots, \delta_{n-1}\} \cup \{\Delta_{i,n}\} & \text{if } i \leq n-2.
 \end{array}\right.$$
 
 Using these sets, we define the following auxiliary curves (see Figure \ref{fig:Sec4-5Fig5}) $$\beta_{i}^{+} = \langle (\Cf \backslash \{\alpha_{0}^{0}, \ldots, \alpha_{0}^{i-1}, \alpha_{2}, \alpha_{4}\}) \cup \Df_{1,i} \rangle \in \Yf{S}^{3},$$ $$\beta_{i}^{-} = \langle (\Cf \backslash \{\alpha_{0}^{i}, \ldots, \alpha_{0}^{n}, \alpha_{2}, \alpha_{4}\}) \cup \Df_{i,n} \rangle \in \Yf{S}^{3}.$$
 
 \begin{figure}[ht]
     \centering
     \includegraphics[height=4cm]{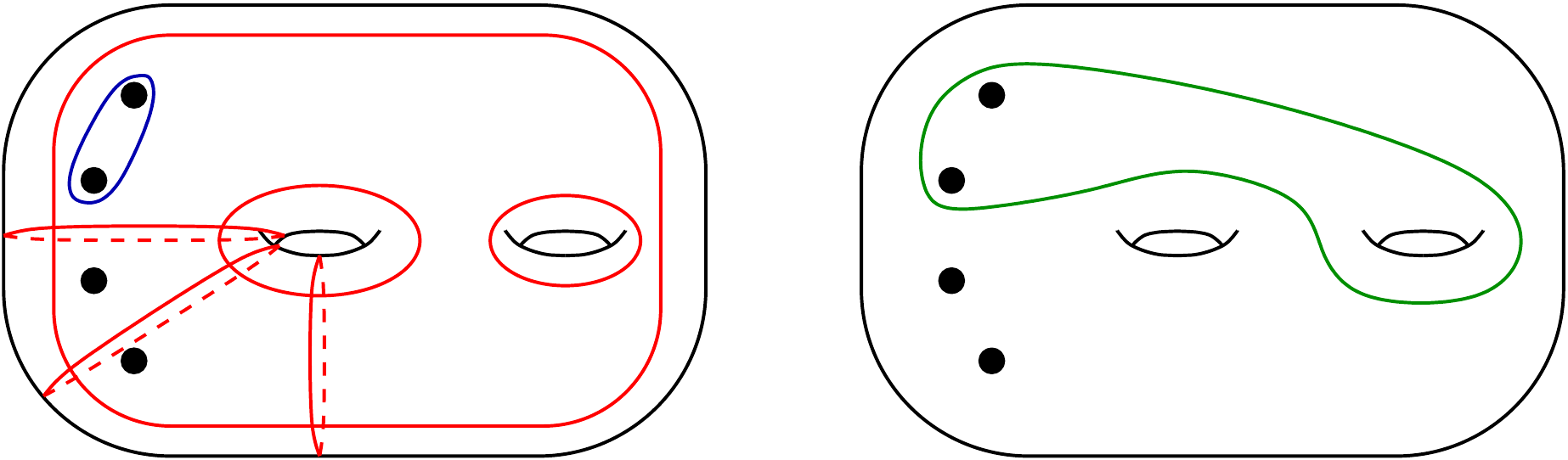}
     \caption{On the left, the curves needed to uniquely determine $\beta_{2}^{+}$, on the right.}
     \label{fig:Sec4-5Fig5}
 \end{figure}
 
 Now we can prove the claim for this case. Note that we need only substitute the set $\{\alpha_{0}^{0},\alpha_{0}^{1}\}$ for $\Cf_{f}$ in the unique determination of the curves $\gamma^{\pm}$ of Case $n=1$, to obtain analogous curves (again denoted by $\gamma^{\pm}$) for Case $n \geq 2$. Then, for the unique determination of $\tau_{\alpha_{5}}^{\pm 1}(\alpha_{4})$ we proceed as in Case $n=1$, adding the set $\Df$ and using these new $\gamma^{\pm}$. This implies that $\tau_{\alpha_{5}}^{\pm 1}(\alpha_{4}) \in \Yf{S}^{3}$
 
 Similarly, to uniquely determine $\tau_{\alpha_{4}}^{\pm 1}(\alpha_{3})$ and $\tau_{\alpha_{3}}^{\pm 1}(\alpha_{2})$ we also substitute in the set $A$ (see Table \ref{tab:ConjuntoA}) the set $\{\alpha_{0}^{0}, \alpha_{0}^{1}\}$ for $\Cf_{f}$. Thus, $\tau_{\alpha_{4}}^{\pm 1}(\alpha_{3}) \in \Yf{S}^{4}$ and $\tau_{\alpha_{3}}^{\pm 1}(\alpha_{2}) \in \Yf{S}^{5}$. For $\tau_{\alpha_{2}}^{\pm 1}(\alpha_{1})$ we add in $A$ the set $\Df$ and get $\tau_{\alpha_{2}}^{\pm 1}(\alpha_{1}) \in \Yf{S}^{6}$.
 
 For $\tau_{\alpha_{1}}^{\pm 1}(\alpha_{0}^{0})$ and $\tau_{\alpha_{1}}^{\pm 1}(\alpha_{0}^{n})$, we substitute in $A$ the curves $\beta_{C}^{+}$ for $\beta_{1}^{-}$ and $\beta_{n}^{+}$ respectively, and add the set $\Df$. Hence, $\tau_{\alpha_{1}}^{\pm 1}(\alpha_{0}^{0}), \tau_{\alpha_{1}}^{\pm 1}(\alpha_{0}^{n}) \in \Yf{S}^{7}$. The curves $\tau_{\alpha_{1}}^{\pm 1}(\alpha_{0}^{i})$ for $1 \leq i \leq n-1$ are a little trickier: we substitute $\beta_{C}^{+}$ for the set $\{\beta_{i}^{+}, \beta_{i+1}^{-}\}$ and add the sets $\Df_{1,i}$ and $\Df_{i+1,n}$. See Figure \ref{fig:Sec4-5Fig6} for an example. Then, $\tau_{\alpha_{1}}^{\pm 1}(\alpha_{0}^{i}) \in \Yf{S}^{7}$ for all $0 \leq i \leq n$.
 
 \begin{figure}[ht]
     \centering
     \includegraphics[height=4cm]{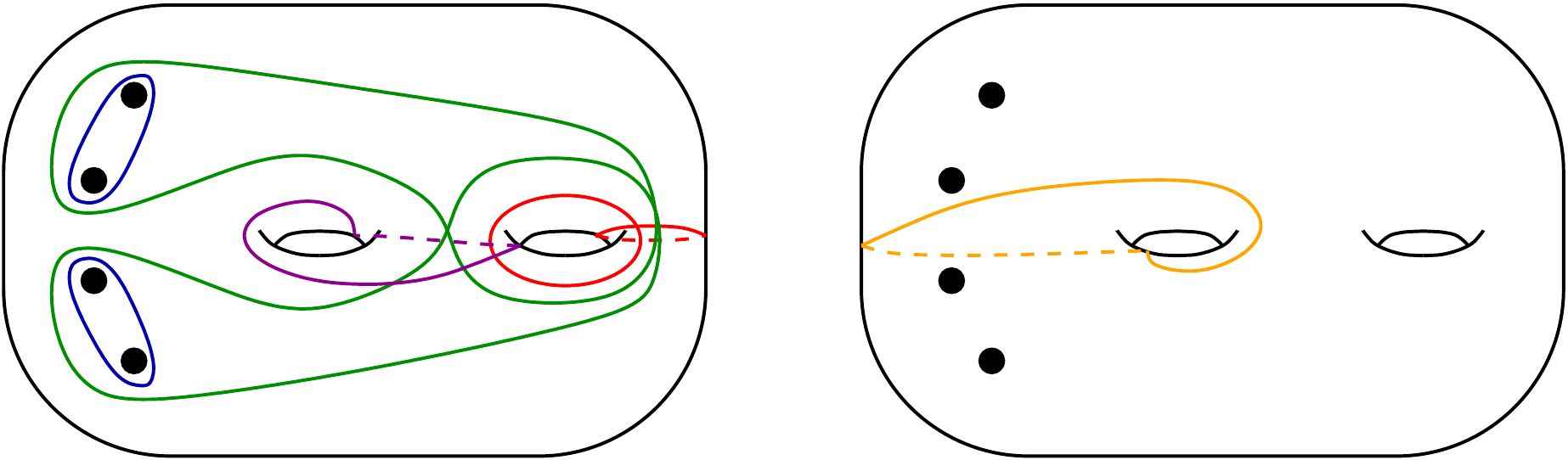}
     \caption{On the left, the curves needed to uniquely determine $\tau_{\alpha_{1}}(\alpha_{0}^{2})$, on the right.}
     \label{fig:Sec4-5Fig6}
 \end{figure}
 
 Proceeding as above, for the curves $\tau_{\alpha_{0}^{0}}^{\pm 1}(\alpha_{5})$ and $\tau_{\alpha_{0}^{n}}^{\pm 1}(\alpha_{5})$ we substitute for $\beta_{1}^{-}$ and $\beta_{n}^{+}$ respectively while adding $\Df$ to $A$ for all of them, and for the curves $\tau_{\alpha_{0}^{i}}^{\pm 1}(\alpha_{5})$ (with $0 < i < n)$ we substitute in $A$ the curve $\beta_{C}^{+}$ for the set $\{\beta_{i}^{+}, \beta_{i+1}^{-}\}$ and add the sets $\Df_{1,i}$ and $\Df_{i+1,n}$. See Figure \ref{fig:Sec4-5Fig7}. Therefore, $\tau_{\alpha_{0}^{i}}^{\pm 1}(\alpha_{5}) \in \Yf{S}^{8}$ for all $0 \leq i \leq n$.
 
 \begin{figure}[ht]
     \centering
     \includegraphics[height=4cm]{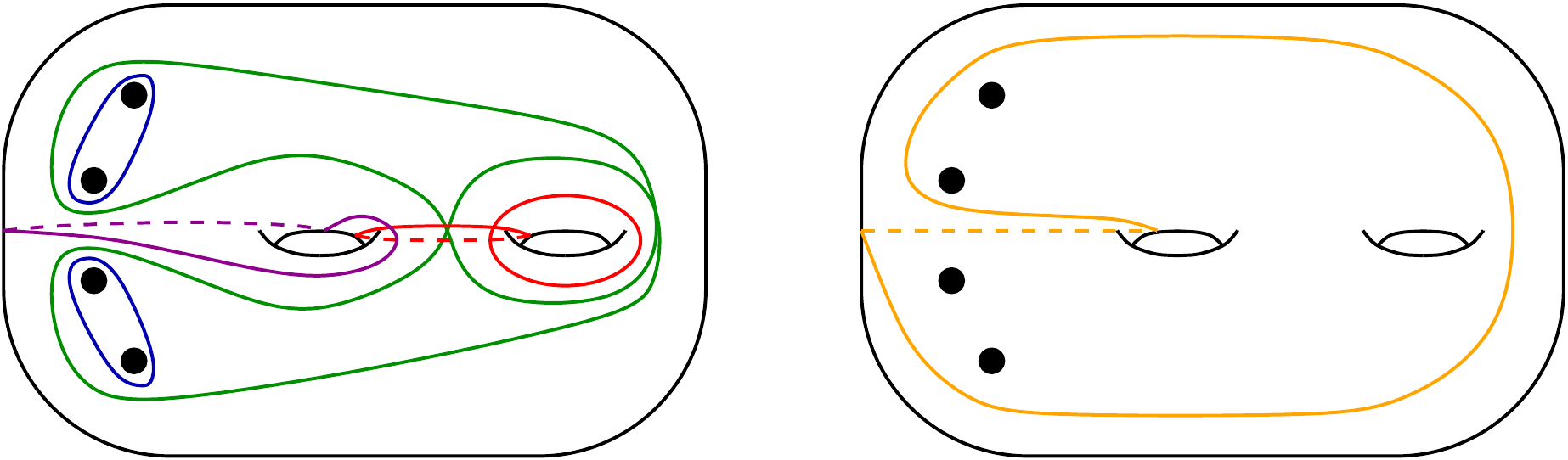}
     \caption{On the left, the curves needed to uniquely determine $\tau_{\alpha_{0}^{2}}(\alpha_{5})$, on the right.}
     \label{fig:Sec4-5Fig7}
 \end{figure}
 
\subsection{Proof of Claim 2: $\tau_{(\Cf \backslash \{\alpha_{5}\})}^{\pm 1}(\Bf) \subset \Yf{S}^{11}$}\label{subsec4-6}
 Similarly to Subsections \ref{subsec2-5} (first case), \ref{subsec2-6} and \ref{subsec3-5}, the proof of this claim follows from the results in Subsection \ref{subsec4-5}: Let $\beta \in \Bf$ and $\alpha \in \Cf \backslash \{\alpha_{5}\}$. Since $S$ has genus $2$, we can uniquely determine $\beta$ using sets $C \subset \Cf$ and $E \subset \Ef$. Then we have $\tau_{\alpha}^{\pm 1}(\beta) = \langle \tau_{\alpha}^{\pm 1}(C \cup E)\rangle$, which by the results in Subsection \ref{subsec4-5} and Remark \ref{Remg2-TCE} is an element in $\Yf{S}^{11}$.
 
 The existence of these sets can be easily inferred from the following examples (see Figure \ref{fig:Sec4-6Fig1}):
 \begin{itemize}
  \item For $C = \{\alpha_{1},\alpha_{2},\alpha_{3}\}$: Let $D = \varnothing$ if $n =1$ and $D= \{\out{j}\}_{j=1}^{n-1} \cup \{\vout{4}\}$ if $n \geq 2$. Then we have
  $$\beta_{C}^{+} = \langle \{\alpha_{1},\alpha_{2},\alpha_{3}\} \cup \{\alpha_{5}\} \cup D \rangle,$$
  $$\beta_{C}^{-} = \alpha_{5}.$$
  \item For $C = \{\alpha_{0}^{i},\alpha_{1},\alpha_{2}\}$ with $0 < i < n$: 
  $$\beta_{C}^{+} = \langle \{\alpha_{0}^{i},\alpha_{1},\alpha_{2}\} \cup \{\alpha_{0}^{j}: i < j\} \cup \{\alpha_{4}\} \cup \{\out{j}: j <i\} \cup \{\vout{4}\}\rangle,$$
  $$\beta_{C}^{-} = \langle \{\alpha_{0}^{i},\alpha_{1},\alpha_{2}\} \cup \{\alpha_{0}^{j}: j < i\} \cup \{\alpha_{4}\} \cup \{\out{j}: i <j\} \cup \{\vout{4}\}\rangle.$$
  \item For $C = \{\alpha_{0}^{n},\alpha_{1},\alpha_{2}\}$: 
  $$\beta_{C}^{+} = \langle \{\alpha_{0}^{n},\alpha_{1},\alpha_{2}\} \cup \{\alpha_{4}\} \cup \Df\rangle,$$
  $$\beta_{C}^{-} = \alpha_{4}.$$
  \item For $C = \{\alpha_{5},\alpha_{0}^{i},\alpha_{1}\}$ with $0 <i<n$: 
  $$\beta_{C}^{+} = \langle \{\alpha_{5},\alpha_{0}^{i},\alpha_{1}\} \cup \{\alpha_{3}\} \cup \{\out{j}: j \neq i\} \cup \{\vout{2},\vout{4}\}\rangle,$$
  $$\beta_{C}^{-} = \alpha_{3}.$$
  \item For $C = \{\alpha_{5},\alpha_{0}^{0},\alpha_{1}\}$:
  $$\beta_{C}^{+} = \langle \{\alpha_{5},\alpha_{0}^{0},\alpha_{1}\} \cup \{\alpha_{3}\} \cup \Df \rangle,$$
  $$\beta_{C}^{-} = \alpha_{3}.$$
  \item For $C = \{\alpha_{4},\alpha_{5},\alpha_{0}^{i}\}$ with $0 <i<n$: 
  $$\beta_{C}^{+} = \langle \{\alpha_{4},\alpha_{5},\alpha_{0}^{i}\} \cup \{\alpha_{0}^{j}: i <j\} \cup \{\alpha_{2}\} \cup \{\out{j}: j <i\} \cup \{\vout{2}\}\rangle,$$
  $$\beta_{C}^{-} = \langle \{\alpha_{4},\alpha_{5},\alpha_{0}^{i}\} \cup \{\alpha_{0}^{j}: j <i\} \cup \{\alpha_{2}\} \cup \{\out{j}: i <j\} \cup \{\vout{2}\}\rangle.$$
  \item For $C = \{\alpha_{4},\alpha_{5},\alpha_{0}^{n}\}$:
  $$\beta_{C}^{+} = \langle \{\alpha_{4},\alpha_{5},\alpha_{0}^{n}\} \cup \{\alpha_{2}\} \cup \Df\rangle,$$
  $$\beta_{C}^{-} = \alpha_{2}.$$
  \item For $C = \{\alpha_{3},\alpha_{4},\alpha_{5}\}$: $$\beta_{C}^{+} = \langle \{\alpha_{3},\alpha_{4},\alpha_{5}\} \cup \{\alpha_{1}\}  \cup \{\out{j}\}_{j=1}^{n-1} \cup \{\vout{2}\}\rangle,$$
  $$\beta_{C}^{-} = \alpha_{1}.$$
 \end{itemize}
 
 \begin{figure}[H]
     \centering
     \includegraphics[height=37mm]{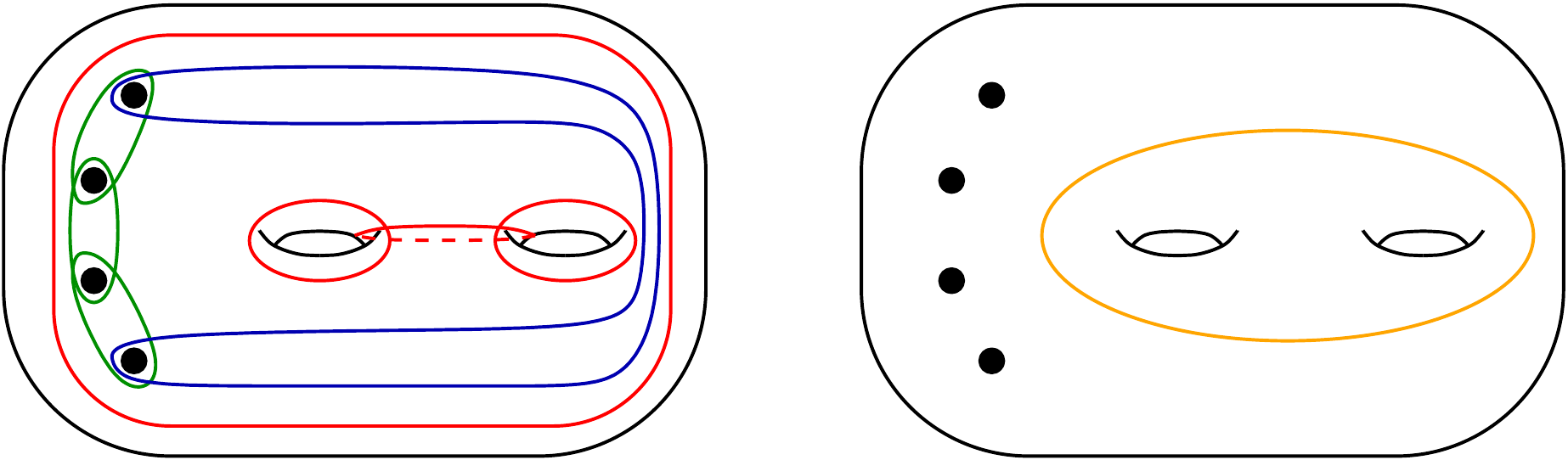}\\[2mm]
     \includegraphics[height=37mm]{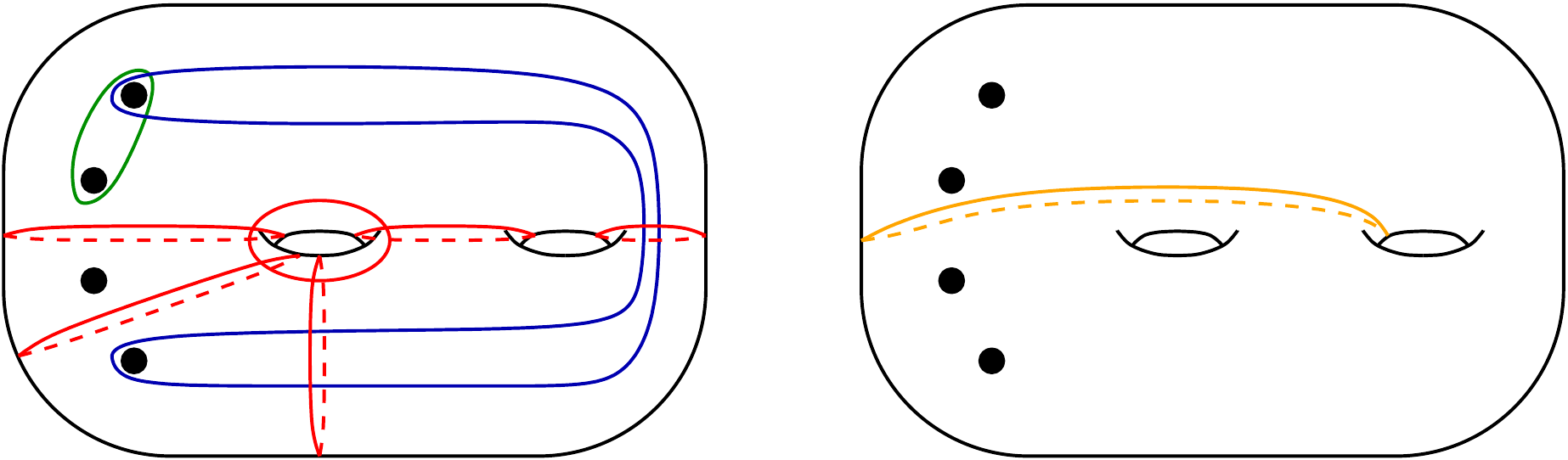}\\[2mm]
     \includegraphics[height=37mm]{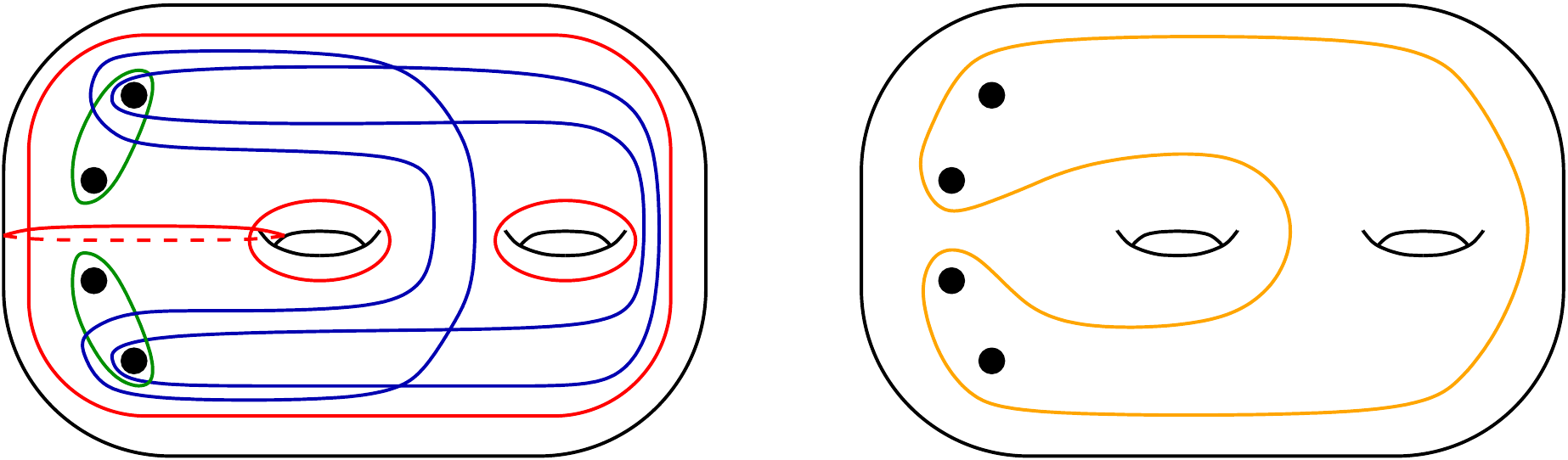}\\[2mm]
     \includegraphics[height=37mm]{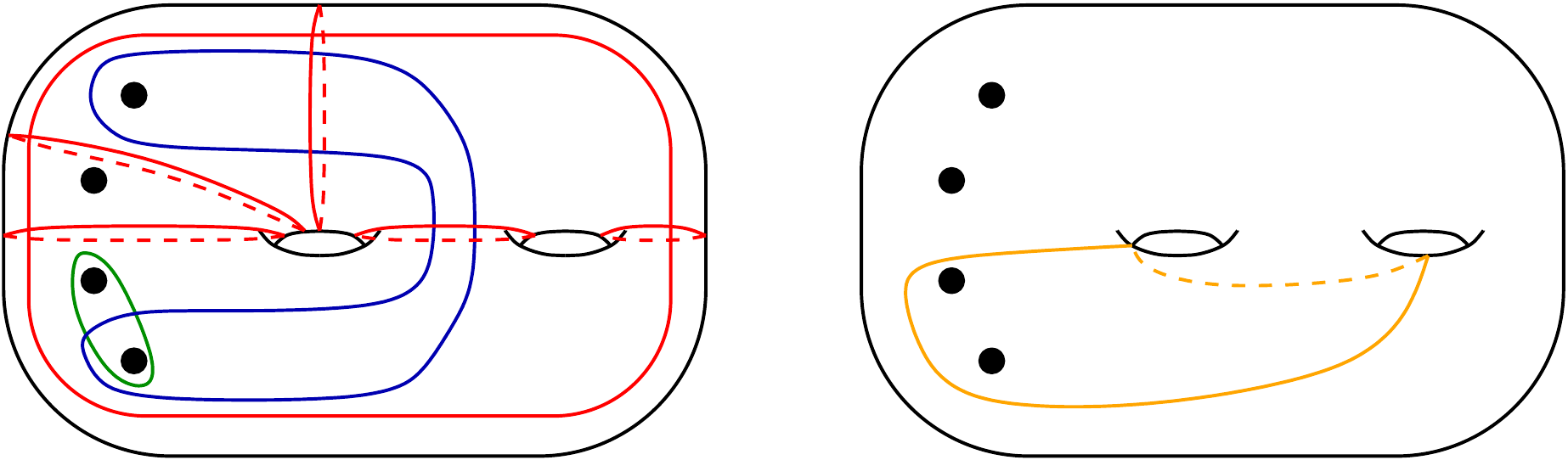}
     \caption{On the left, the curves needed to uniquely determine the curves $\beta_{C}^{+}$ for $C=\{\alpha_{1},\alpha_{2},\alpha_{3}\}$ (top), $C=\{\alpha_{0}^{2},\alpha_{1},\alpha_{1},\alpha_{2}\}$ (upper middle), $C=\{\alpha_{5},\alpha_{0}^{2},\alpha_{1}\}$ (lower middle) and $\beta_{C}^{-}$ for $C=\{\alpha_{4},\alpha_{5},\alpha_{0}^{2}\}$ (bottom), on the right.}
     \label{fig:Sec4-6Fig1}
 \end{figure}
\subsection{Proof of Claim 3: $\eta_{\{\out{i} : 0 < i <n\}}^{\pm 1}(\Cf) \subset \Yf{S}^{3}$}\label{subsec4-7}
 In this subsection we assume that $n \geq 2$ (otherwise there would be no outer curves).
 
 To prove this claim, first note that we only need to prove that $\eta_{\out{i}}^{\pm 1}(\alpha_{0}^{i}) \in \Yf{S}^{3}$ (all other curves in $\Cf$ are disjoint from $\out{i}$). To do so, we define the following auxiliar curves.
 
 For each $1 \leq i \leq n$, we define the following curves (see Figure \ref{fig:Sec4-7Fig1}): 
 $$\gamma_{+}^{i} = \langle \{\alpha_{0}^{j}: j <i\} \cup \{\alpha_{1},\alpha_{2}\} \cup \{\alpha_{4}\} \cup \{\zeta^{k}: i < k \leq n\}\rangle,$$
 $$\gamma_{-}^{i} = \langle \{\alpha_{0}^{j}: i < j\} \cup \{\alpha_{1}, \alpha_{2}\} \cup \{\alpha_{4}\} \cup \{\zeta^{k} : 1 \leq k \leq i\} \rangle.$$
 Note that for all $1 \leq i \leq n$ we have that $\gamma_{+}^{i}, \gamma_{-}^{i} \in \Yf{S}^{2}$.
 
 \begin{figure}[ht]
     \centering
     \includegraphics[height=4cm]{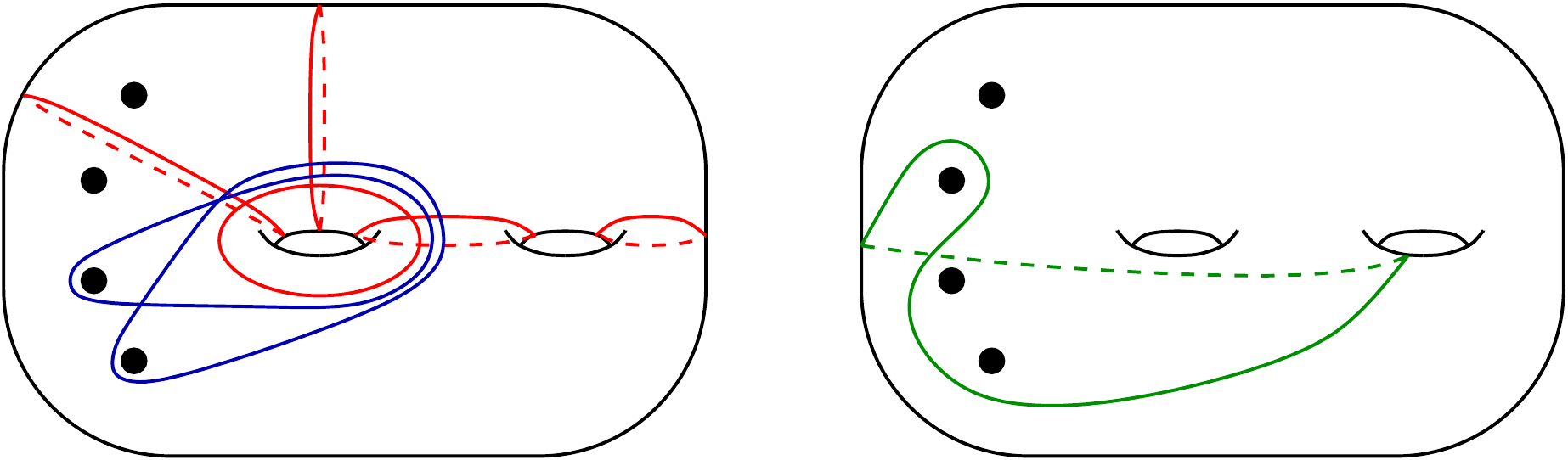}\\[5mm]
     \includegraphics[height=4cm]{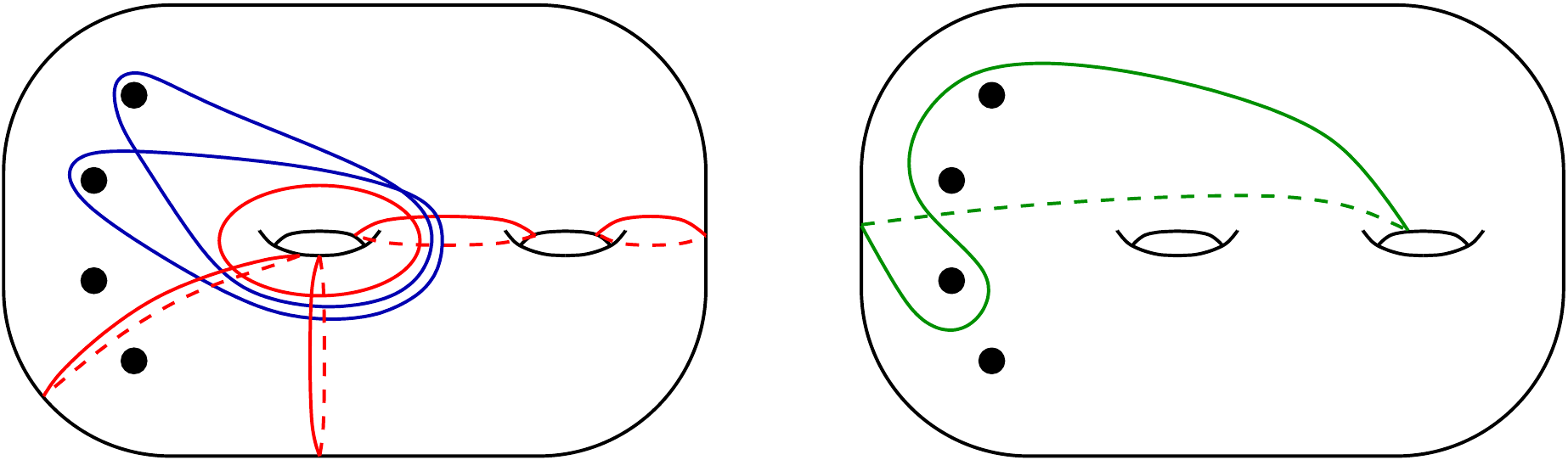}
     \caption{On the left, the curves needed to uniquely determine the curves $\gamma_{+}^{2}$ (top) and $\gamma_{-}^{2}$ (bottom), on the right.}
     \label{fig:Sec4-7Fig1}
 \end{figure}
 
 Finally, for each $1 \leq i <n$, we have that (see Figure \ref{fig:Sec4-7Fig2}): $$\eta_{\out{i}}(\alpha_{i}) = \langle (\Cf \backslash \{\alpha_{0}^{i}, \alpha_{1}, \alpha_{5}\}) \cup \{\gamma_{+}^{i}\}\rangle,$$
 $$\eta_{\out{i}}^{-1}(\alpha_{i}) = \langle (\Cf \backslash \{\alpha_{0}^{i}, \alpha_{1}, \alpha_{5}\}) \cup \{\gamma_{-}^{i}\}\rangle.$$
 
 Therefore, $\eta_{\{\out{i} : 0 < i <n\}}^{\pm 1}(\Cf) \subset \Yf{S}^{3}$
 
 \begin{figure}[ht]
     \centering
     \includegraphics[height=4cm]{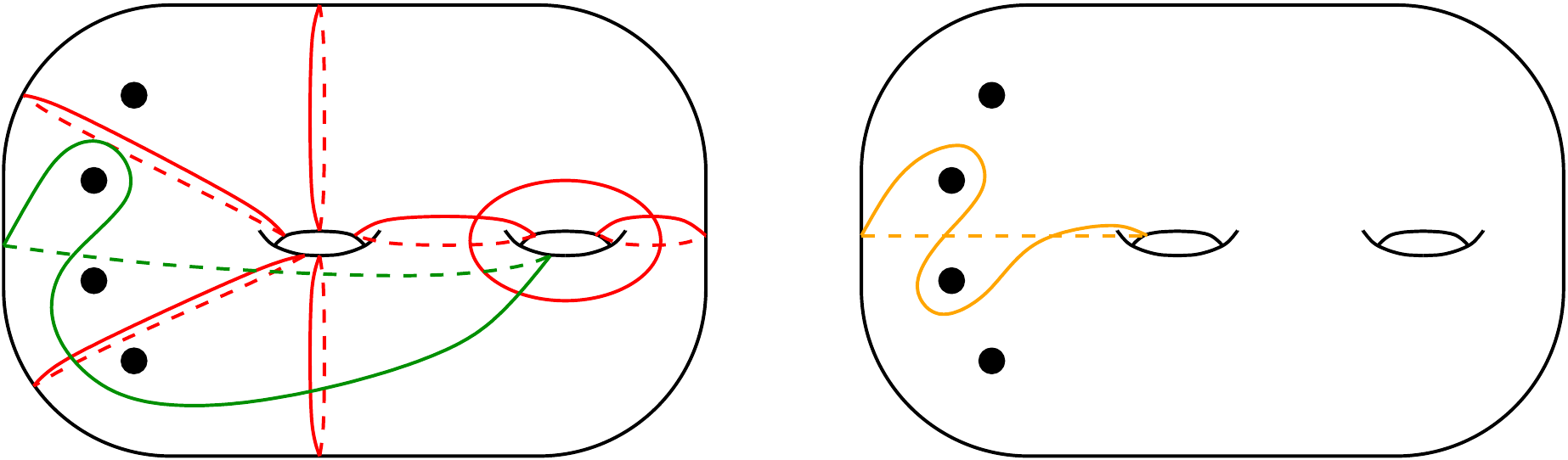}
     \caption{On the left, the curves needed to uniquely determine $\eta_{\out{i}}(\alpha_{2})$, on the right.}
     \label{fig:Sec4-7Fig2}
 \end{figure}
 
 \begin{Rem}\label{Remg2-HCE}
  Again (as in Subsection \ref{subsec4-5}), since $\Ef \subset \Cf^{2}$, we have that $$\eta_{\{\out{i} : 0 < i <n\}}^{\pm 1}(\Ef) \subset \Yf{S}^{5}.$$
 \end{Rem} 
\subsection{Proof of Claim 4: $\eta_{\{\out{i} : 0 < i <n\}}^{\pm 1}(\Bf) \subset \Yf{S}^{6}$}\label{subsec4-8}
 In this subsection we assume that $n \geq 2$ (otherwise there would be no outer curves).
 
 Analogously to Subsection \ref{subsec4-6}, this claim follows from the results in Subsection \ref{subsec4-7}: Let $\beta \in \Bf$ and $\delta \in \{\out{i} : 0 < i <n\}$; since $\beta = \langle C \cup E\rangle$ with $C \subset \Cf$ and $E \subset \Ef$ (as in Subsection \ref{subsec4-6}), we have that $\eta_{\delta}^{\pm 1}(\beta) = \langle \eta_{\delta}^{\pm 1}(C \cup E)\rangle$, which by the results in Subsection \ref{subsec4-7} and Remark \ref{Remg2-HCE} is contained in $\Yf{S}^{6}$.
 
 For the existence of these sets, we refer the reader to the examples and figures in Subsection \ref{subsec4-6}.
\subsection{Proof of Theorem \ref{TeoB}}\label{subsec4-9}
 Similarly to Subsections \ref{subsec2-7} and \ref{subsec3-8}, here we define the set $\Xf{S}$ as the set $\mathfrak{X}_{1}$ from Section 5 in \cite{Ara2}, which was proved to be rigid (see Lemma 5.2 in \cite{Ara2}).
 
 Let $X$ be the set defined as $\Cf$ union the set of all boundary curves of regular neighbourhoods of sets $A$ with one of the following forms:
 \begin{enumerate}
  \item $A = \{\alpha_{0}^{i},\alpha_{0}^{j},\alpha_{k}\}$ where $0 \leq i \leq j \leq n$ and $k = 1, 5$.
  \item $A = \{\alpha_{0}^{i}, \alpha_{0}^{j}, \alpha_{k}, \alpha_{k+1}\}$ where $0 \leq i \leq j \leq n$ and $k = 1, 4$.
  \item $A = \{\alpha_{i}, \alpha_{i+1}, \ldots, \alpha_{j}\}$, where $I = \{i, i+1, \ldots, j\}$ is an integer subinterval of $\{0, \ldots, 5\}$ modulo $6$ (We assume that $\alpha_{0}$ denotes $\alpha_{0}^{k}$ for some $k$). If the length of the interval $I$ is odd, then we require that both $i$ and $j$ are even.
 \end{enumerate}
 
 Then, we define $\Xf{S}$ as the union of $X^{1}$ and the set of all boundary curves of regular neighbourhoods of sets $A = \{\alpha_{i}, \alpha_{i+1}, \ldots, \alpha_{j}\}$, where $0 < i \leq j \leq 3$, and both $i$ and $j$ are odd.
 
 Finally, note that $\Yf{S} \subset \Xf{S}$, which immediatly implies Theorem \ref{TeoB}.
\section{Rigidity}\label{sec5}
In this section we prove Theorems \ref{TeoD} and \ref{TeoE}, i.e. we prove that using only topological restrictions on the surfaces, all graph morphisms between the corresponding curve graphs are geometric. The ideas and techniques used in this section are essentially those used in \cite{JHH2}.

To prove Theorems \ref{TeoD} and \ref{TeoE}, we first need the following definitions and lemmas.

Let $S = S_{g,n}$ be such that $\kappa(S) = 3g -3 +n \geq 2$. A \textit{multicurve} $M$ on $S$ is a set of curves on $S$ corresponding in $\ccomp{S}$ to a complete subgraph. Note that $|M| \leq \kappa(S)$. If the equality holds, then $M$ is called a \textit{pants decomposition of} $S$, since $S \backslash M$ would then be equal to the disjoint union of surfaces homeomorphic to $S_{0,3}$.

\begin{Lema}\label{Multicurves}
 Let $S_{1} = S_{g_{1},n_{1}}$ and $S_{2} = S_{g_{2},n_{2}}$ be such that $\kappa(S_{1}), \kappa(S_{2}) \geq 2$, and let $\varphi: \ccomp{S_{1}} \to \ccomp{S_{2}}$ be a graph morphism. Then $\varphi$ maps multicurves to multicurves of the same cardinality. Moreover, $\kappa(S_{1}) \leq \kappa(S_{2})$.
\end{Lema}
\begin{proof}
 The first part follows directly from $\varphi$ being a graph morphism, and a multicurve $M$ being a complete graph in $\ccomp{S_{1}}$. The second part follows from a pants decomposition $P$ of $S$ having maximal cardinality, and the first part of the lemma.
\end{proof}

\begin{Rem}\label{PantsToPants}
Note that this lemma implies that if $\kappa(S_{1}) \geq \kappa(S_{2})$, and the hypothesis of the lemma are satisfied, then $\kappa(S_{1}) = \kappa(S_{2})$. In particular, pants decompositions of $S_{1}$ would then be mapped to pants decompositions of $S_{2}$.
\end{Rem}

Let $\alpha$ and $\beta$ be curves on $S = S_{g,n}$. We say $\alpha$ and $\beta$ are \textit{Farey neighbours} if their open regular neighbourhood $N(\alpha,\beta)$ on $S$ has complexity one, i.e. is homeomorphic to either $S_{0,4}$ or $S_{1,1}$. In particular, we say $\alpha$ and $\beta$ are \textit{spherical} (respectively \textit{toroidal}) Farey neighbours if $N(\alpha,\beta)$ is homeomorphic to $S_{0,4}$ (respectively $S_{1,1}$). This terminology is inspired by the fact that both $\ccomp{S_{0,4}}$ and $\ccomp{S_{1,1}}$ (with a slightly modified definition, see \cite{FarbMar}) are isomorphic to the Farey graph (the 1-skeleton of the ideal triangulation of the hyperbolic plane); thus, if $\alpha$ and $\beta$ are Farey neighbours, then they are adjacent (neighbours) in $\ccomp{N(\alpha,\beta)}$.

Now, let $\varphi: \ccomp{S_{1}} \to \ccomp{S_{2}}$ be a graph morphism. We say $\varphi$ is \textit{spherical Farey} (respectively \textit{toroidal Farey}) if for all curves $\alpha$ and $\beta$ on $S_{1}$ that are spherical (respectively toroidal) Farey neighbours, they are mapped to intersecting curves, i.e. $i(\varphi(\alpha), \varphi(\beta)) \neq 0$. We say $\varphi$ is a \textit{Farey map} either if it is both spherical Farey and toroidal Farey when $g_{1} \geq 1$, or if it is sphereical Farey when $g_{1} = 0$. Note that being spherical/toroidal Farey (or both) is a generalisation of the concept of superinjectivity, see \cite{Irmak1}, \cite{Irmak2}, \cite{BehrMar} and \cite{Irmak3}.

\begin{Lema}\label{Farey}
 Let $S_{1} = S_{g_{1},n_{1}}$ and $S_{2} = S_{g_{2},n_{2}}$ be such that $\kappa(S_{1}) \geq \kappa(S_{2}) \geq 2$, and let $\varphi: \ccomp{S_{1}} \to \ccomp{S_{2}}$ be a graph morphism. If $S_{1}$ is not homeomorphic to $S_{1,2}$ nor $S_{2,0}$, then $\varphi$ is a Farey map.
\end{Lema}
\begin{proof}
 Let $\alpha$ and $\beta$ be curves on $S_{1}$ such that they are Farey neighbours. Then, there exists a multicurve $M$ on $S_{1}$ such that $M_{\alpha} = M \cup \{\alpha\}$ and $M_{\beta} = M \cup \{\beta\}$ are pants decompositions of $S_{1}$. By Lemma \ref{Multicurves} and Remark \ref{PantsToPants}, we know that $\kappa(S_{1}) = \kappa(S_{2})$, and that $\varphi(M_{\alpha})$ and $\varphi(M_{\beta})$ are pants decompositions of $S_{2}$. Since $\varphi(M_{\alpha})$ and $\varphi(M_{\beta})$ differ by exactly one curve, a complexity argument can be used to conclude that $\varphi(\alpha)$ and $\varphi(\beta)$ are contained in a subsurface $N$ of $S_{2}$ with $\kappa(N) = 1$.
 
 Thus, either $\varphi(\alpha) = \varphi(\beta)$ or $i(\varphi(\alpha),\varphi(\beta)) \neq 0$.
 
 To prove that $\varphi(\alpha) \neq \varphi(\beta)$ we use an auxiliary curve: Let $\gamma$ be a curve on $S_{1}$ such that $\alpha$ and $\gamma$ are disjoint, and $\beta$ is a Farey neighbour of $\gamma$. Using the same argument as above, we can deduce that either $\varphi(\beta) = \varphi(\gamma)$ or $i(\varphi(\beta),\varphi(\gamma)) \neq 0$. But none of these options can be true if $\varphi(\alpha) = \varphi(\gamma)$. Hence, $i(\varphi(\alpha), \varphi(\beta)) \neq 0$.
\end{proof}

Note that, if $S_{1}$ is homeomorphic to either $S_{1,2}$ or $S_{2,0}$, while we can prove that $\varphi$ is toroidal Farey using the same proof, in both cases there exist curves $\alpha$ and $\beta$ that are spherical Farey neighbours, such that any curve $\gamma$ disjoint from one of them, cannot be a Farey neighbour of the other. See Figure \ref{fig:Sec5-0Fig1}.

\begin{figure}[ht]
    \centering
    \labellist
    \pinlabel $\beta$ [b] at 222 218
    \pinlabel $\alpha$ [tr] at 185 175 
    \endlabellist
    \includegraphics[height=4cm]{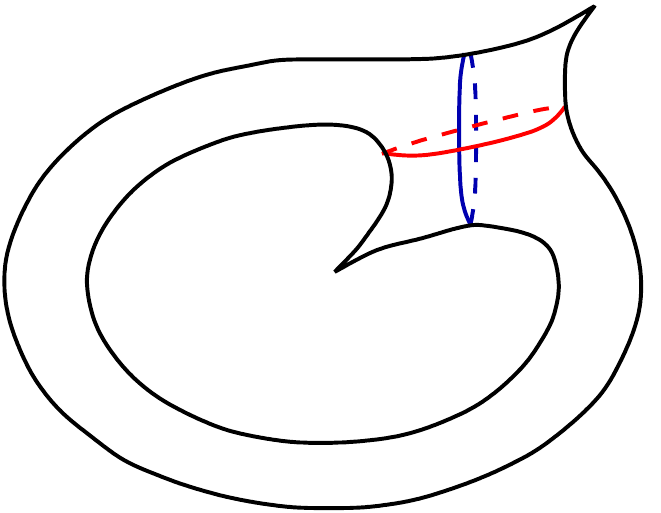} \hspace{3mm}
    \labellist
    \pinlabel $\beta$ [b] at 220 208
    \pinlabel $\alpha$ [tr] at 180 180
    \endlabellist
    \includegraphics[height=4cm]{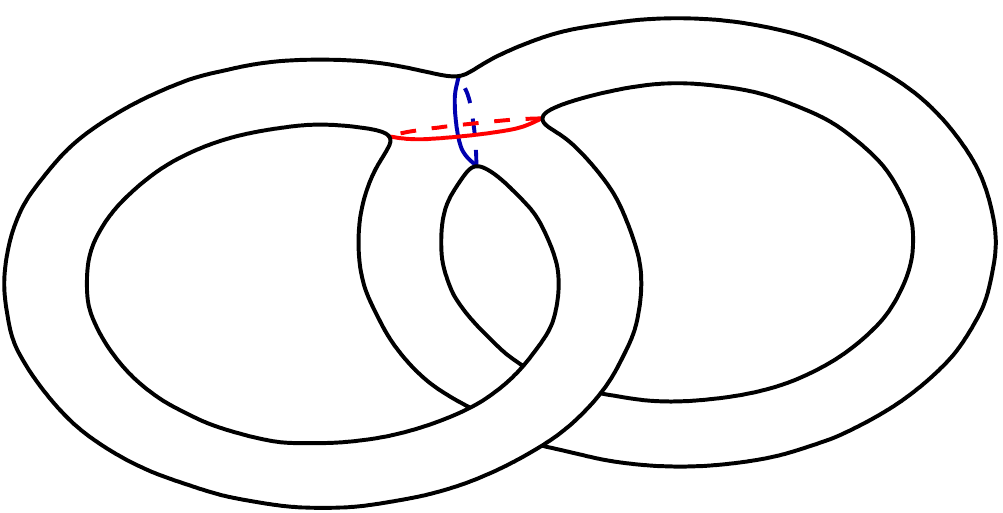}
    \caption{Examples of curves $\alpha$ and $\beta$ on $S_{1,2}$ (left) and $S_{2,2}$ (right) such that $\alpha$ and $\beta$ are spherical Farey neighbours, and any $\gamma$ disjoint from any of them cannot be a Farey neighbour of the other.}
    \label{fig:Sec5-0Fig1}
\end{figure}

Armed with Lemma \ref{Farey}, to prove Theorem \ref{TeoD} we only need Corollary C in \cite{JHH2}, which for the sake of completeness we rewrite here for the particular case of $\ccomp{S_{g,n}}$.
\begin{Corno}[C in \cite{JHH2}]
 Let $Y < \ccomp{S_{g,n}}$ be an rigid subgraph of $\ccomp{S_{g,n}}$, and $\phi: Y^{\omega} \to \ccomp{S_{g,n}}$ a graph morphism such that $\phi|_{Y}$ is locally injective. Then $\phi$ is the restriction to $Y^{\omega}$ of an automorphism of $\ccomp{S_{g,n}}$, unique up to the pointwise stabilizer of $Y$ in $\Aut{\ccomp{S_{g,n}}}$.
\end{Corno}

We recall Theorem \ref{TeoD} from the Introduction:
\begin{Teono}[\ref{TeoD}]
 Let $S = S_{g,n}$ with $\kappa(S) \geq 2$ and let $\varphi: \ccomp{S} \to \ccomp{S}$ be a graph morphism. Then $\varphi$ is an automorphism of $\ccomp{S}$.
\end{Teono}
\begin{proof}[\textbf{Proof of Theorem \ref{TeoD}}]
 As was mentioned in the Introduction, since Theorem \ref{TeoD} has already been proved by the author for the case $g \geq 3$, we assume that $g \leq 2$. For now, let $S$ be different from $S_{1,2}$ and $S_{2,0}$. Due to Theorem \ref{TeoB} (from this work) and Corollary C (from \cite{JHH2}), we simply need to prove that $\varphi|_{\Xf{S}}$ is locally injective.
 
 Let $\alpha$ and $\beta$ be different curves in $\Xf{S}$. By construction, one of the following situations holds:
 
 \begin{enumerate}
  \item $\alpha$ is disjoint from $\beta$.
  \item $\alpha$ and $\beta$ are Farey neighbours.
  \item There exists a curve $\gamma$ on $S$ such that $\alpha$ is disjoint from $\gamma$, and $\beta$ is a Farey neighbour of $\gamma$.
 \end{enumerate}
 
 Since $\varphi$ is a graph morphism, by Lemma \ref{Farey} $\varphi$ is a Farey map, and thus $\varphi(\alpha) \neq \varphi(\beta)$ if either (1) or (2) holds. If (3) holds, then $\varphi(\alpha)$ is disjoint from $\varphi(\gamma)$, and $i(\varphi(\beta),\varphi(\gamma)) \neq 0$, thus $\varphi(\alpha) \neq \varphi(\beta)$. This implies that $\varphi|_{\Xf{S}}$ is injective (and in particular locally injective) as desired.
 
 If $S$ is homeomorphic to either $S_{1,2}$ or $S_{2,0}$, let $\Sigma$ be homeomorphic to either $S_{0,5}$ or $S_{0,6}$, respectively. Then there exists an isomorphism $\psi: \ccomp{S} \to \ccomp{\Sigma}$ (induced by the hyperelliptic involution). Hence $\psi \circ \varphi \circ \psi^{-1}$ is a graph endomorphism of $\ccomp{\Sigma}$, and by the argument above is an automorphism of $\ccomp{\Sigma}$. Thus $\varphi$ is an automorphism of $\ccomp{S}$.
\end{proof}

Now, as was mentioned in the Introduction, to prove Theorem \ref{TeoE} we simply need to prove the prove the following theorem.

\begin{Teo}\label{TopRig}
 Let $S_{1} = S_{g_{1},n_{1}}$ and $S_{2} = S_{g_{2},n_{2}}$ be such that $\kappa(S_{1}) \geq \kappa(S_{2}) \geq 4$ with $(g_{1},n_{1}) \neq (0,7)$, and $\varphi: \ccomp{S_{1}} \to \ccomp{S_{2}}$ be a graph morphism. Then, $S_{1}$ is homeomorphic to $S_{2}$.
\end{Teo}

To prove this theorem, we are going to use (almost standard) techniques and geometric results. These techniques are heavily inspired by those in \cite{Shack} and are practically identical to those in \cite{JHH2} (to the point that in many cases we simply give a sketch of the proof and a reference).
\vspace{5mm}

Let $S = S_{g,n}$ with $\kappa(S) \geq 4$, and $P$ be a pants decomposition of $S$. We say a closed subsurface $\Sigma$ of $S$ is \textit{induced by} $P$ if all its boundary curves are pairwise non-isotopic, and all its boundary curves are elements of $P$. We say $\alpha, \beta \in P$ are \textit{adjacent with respect to} $P$ if there exists a closed subsurface $\Sigma$ induced by $P$, whose interior is homeomorphic to $S_{0,3}$ and such that $\alpha$ and $\beta$ are boundary curves of $\Sigma$. With this, we define the \textit{adjacency graph of} $P$, denoted by $\acomp{P}$, as the simplicial graph whose vertex set is $P$, and two vertices span an edge if they are adjacent to each other with respect to $P$.

Note that if $S_{1} = S_{g_{1},n_{1}}$ and $S_{2} = S_{g_{2},n_{2}}$ is such that $\kappa(S_{1}) \geq \kappa(S_{2}) \geq 4$, and $\varphi: \ccomp{S_{1}} \to \ccomp{S_{2}}$ is a graph morphism, then for every pants decomposition $P$ of $S_{1}$, by Remark \ref{PantsToPants}, there exists an induced map $\varphi_{P}: \acomp{P} \to \acomp{\varphi(P)}$. The following Lemma proves that $\varphi_{P}$ is a graph isomorphism.

\begin{Lema}[cf. Lemma 7 in \cite{Shack} and Lemma 3.6 in \cite{JHH2}]\label{Adjacency}
 Let $S_{1} = S_{g_{1},n_{1}}$ and $S_{2} = S_{g_{2},n_{2}}$ be such that $\kappa(S_{1}) \geq \kappa(S_{2}) \geq 4$, $\varphi: \ccomp{S_{1}} \to \ccomp{S_{2}}$ be a graph morphism, and $P$ be a pants decomposition of $S_{1}$. Then, for all $\alpha, \beta \in P$, $\alpha$ and $\beta$ are adjacent with respect to $P$ if and only if $\varphi(\alpha)$ and $\varphi(\beta)$ are adjacent with respect to $\varphi(P)$. In particular, the induced map $\varphi_{P}$ is a graph isomorphism.
\end{Lema}

\begin{proof}[Sketch of proof]
 Let $\gamma_{\alpha}$ and $\gamma_{\beta}$ be Farey neighbours of $\alpha$ and $\beta$, respectively, such that $P_{\alpha} = (P \backslash \{\alpha\}) \cup \{\gamma_{\alpha}\}$ and $P_{\beta} = (P \backslash \{\beta\}) \cup \{\gamma_{\beta}\}$ are pants decompositions. Then by Remark \ref{PantsToPants} $\varphi(P)$, $\varphi(P_{\alpha})$ and $\varphi(P_{\beta})$ are also pants decompositions.
 
 If $\alpha$ and $\beta$ are adjacent with respect to $P$, we can choose $\gamma_{\alpha}$ and $\gamma_{\beta}$ as Farey neighbours of each other. Then, $i(\varphi(\alpha),\varphi(\gamma_{\alpha})) \neq 0$, $i(\varphi(\beta),\varphi(\gamma_{\beta})) \neq 0$, and $i(\varphi(\gamma_{\alpha},\varphi(\gamma_{\beta}) \neq 0$ by Lemma \ref{Farey}. But this would be impossible if $\varphi(\alpha)$ and $\varphi(\beta)$ were not adjacent with respect to $\varphi(P)$. Thus, $\varphi(\alpha)$ and $\varphi(\beta)$ are adjacent with respect to $\varphi(P)$.
 
 If $\alpha$ and $\beta$ are not adjacent with respect to $P$, then $\gamma_{\alpha}$ and $\gamma_{\beta}$ are disjoint. Then, $i(\varphi(\alpha),\varphi(\gamma_{\alpha})) \neq 0$, $i(\varphi(\beta),\varphi(\gamma_{\beta})) \neq 0$, and $\varphi(\gamma_{\alpha})$ and $\varphi(\gamma_{\beta})$ are disjoint, by Lemma \ref{Farey} and $\varphi$ being a graph morphism. But this would be impossible if $\varphi(\alpha)$ and $\varphi(\beta)$ were adjacent with respect to $\varphi(P)$. Thus, $\varphi(\alpha)$ and $\varphi(\beta)$ are not adjacent with respect to $\varphi(P)$.
\end{proof}

Recall that an outer curve $\alpha$ is a separating curve on $S$ such that $S \backslash \{\alpha\}$ has a connected component homeomorphic to $S_{0,3}$.

\begin{Lema}[cd. Lemma 9 in \cite{Shack}, and Lemma 3.7 in \cite{JHH2}]\label{Non-outer}
 Let $S_{1} = S_{g_{1},n_{1}}$ and $S_{2} = S_{g_{2},n_{2}}$ be such that $\kappa(S_{1}) \geq \kappa(S_{2}) \geq 4$, and $\varphi: \ccomp{S_{1}} \to \ccomp{S_{2}}$ be a graph morphism. Then, $\varphi$ maps non-outer separating curves to non-outer separating curves.
\end{Lema}

\begin{proof}[Sketch of proof]
 This follows from the fact that $\varphi_{P}$ is an isomorphism and noting that if $P$ is a pants decomposition and $\alpha \in P$, then $\alpha$ is an non-outer separating curve if and only if the vertex corresponding to $\alpha$ in $\acomp{P}$ is a cut vertex.
\end{proof}

\begin{Lema}[cf. Lemma 10 in \cite{Shack}, and Lemma 3.8 in \cite{JHH2}]\label{Nonsep}
 Let $S_{1} = S_{g_{1},n_{1}}$ and $S_{2} = S_{g_{2},n_{2}}$ be such that $\kappa(S_{1}) \geq \kappa(S_{2}) \geq 4$, and $\varphi: \ccomp{S_{1}} \to \ccomp{S_{2}}$ be a graph morphism. Then, $\varphi$ maps non-separating curves to non-separating curves.
\end{Lema}

\begin{proof}[Sketch of proof]
 Again, this follows from Lemma \ref{Non-outer}, the fact that $\varphi_{P}$ is an isomorphism, and noting two things:
 \begin{enumerate}
  \item If $\alpha$ is an outer curve, then for all pants decompositions $P$ with $\alpha \in P$, $\alpha$ can adjacent to at most 2 different curves with respect to $P$.
  \item If $\alpha$ is a non-separating curve, twe can always choose a pants decomposition $P$ with $\alpha \in P$, such that $\alpha$ is adjacen to at least three different curves with respect to $P$.
 \end{enumerate}
\end{proof}

Let $S = S_{g,n}$ be such that $\kappa(S) \geq 4$, and $\alpha$ and $\beta$ be disjoint curves on $S$. The set $\{\alpha,\beta\}$ is a \textit{peripheral pair} if $\alpha$ and $\beta$ together bound a closed subsurface of $S$ homeomorphic to a once-punctured annulus. Let $\gamma$ be a curve disjoint from $\alpha$ and $\beta$; we say $\alpha$, $\beta$ and $\gamma$ \emph{bound a pair of pants} if they are the only boundary curves of a closed subsurface of $S$, whose interior is homeomorphic to $S_{0,3}$.

\begin{Rem}\label{Peripheral}
 Let $S = S_{g,n}$ be such that $\kappa(S) \geq 4$, $P$ be a pants decomposition of $S$, and $\alpha,\beta,\gamma \in P$. Note that we have the following:
 
 \begin{enumerate}
  \item  If $\alpha$ is adjacent with respect to $P$ to only $\beta$ and $\gamma$, and these three curves are non-separating, then $\{\alpha, \beta\}$ and $\{\alpha,\gamma\}$ are peripheral pairs.
  \item If $\alpha$, $\beta$ and $\gamma$ are adjacent to each other with respect to $P$, then either together they bound a pair of pants, or they are non-separating curves that look up to homeomorphism like Figure \ref{fig:Sec5-0Fig2}.
  \item (Using previous point) If $\alpha$, $\beta$ and $\gamma$ are adjacent to each other with respect to $P$, and two of them are separating curves, then the three of them are separating curves.
 \end{enumerate}
\end{Rem}

\begin{figure}[ht]
    \centering
    \labellist
    \pinlabel $\alpha$ [t] at 155 208
    \pinlabel $\beta$ [bl] at 65 85
    \pinlabel $\gamma$ [br] at 250 90
    \endlabellist
    \includegraphics[height=4cm]{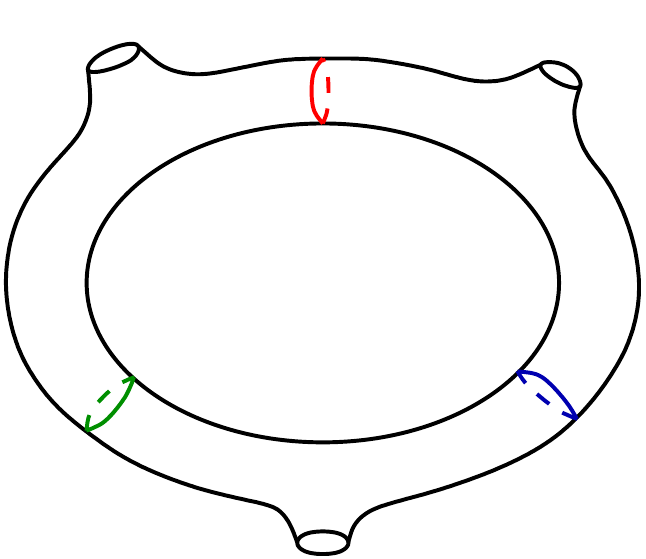}
    \caption{Curves $\alpha$, $\beta$, $\gamma$ adjacent to each other with respect to a pants decomposition, as in Remark \ref{Peripheral} (2).}
    \label{fig:Sec5-0Fig2}
\end{figure}

\begin{Lema}[cf. Lemma 11 in \cite{Shack}, and Lemma 3.10 in \cite{JHH2}]\label{Outer}
 Let $S_{1} = S_{g_{1},n_{1}}$ and $S_{2} = S_{g_{2},n_{2}}$ be such that $\kappa(S_{1}) \geq \kappa(S_{2}) \geq 4$, and $\varphi: \ccomp{S_{1}} \to \ccomp{S_{2}}$ be a graph morphism. Then, $\varphi$ maps outer curves to outer curves.
\end{Lema}

\begin{proof}
 Let $\alpha$ be an outer curve of $S_{1}$. By the argument in the proof of Lemma \ref{Non-outer}, we know that $\varphi(\alpha)$ cannot be a non-outer separating curve. This implies that if $\varphi(\alpha)$ were not an outer curve it would have to be non-separating. Now, we divide the proof into two cases:
 
 \textit{Case} $(g_{1},n_{1}) \neq (1,4)$: Let $\beta$ and $\gamma$ be non-outer separating curves such that $\alpha$, $\beta$ and $\gamma$ bound a pair of pants. Let also $P$ be a pants decomposition of $S_{1}$ with $\alpha,\beta,\gamma \in P$. By Lemma \ref{Non-outer} we know that $\varphi(\beta)$ and $\varphi(\gamma)$ are non-outer separating curves, and by Lemma \ref{Adjacency} we know that $\varphi(\alpha)$, $\varphi(\beta)$ and $\varphi(\gamma)$ are adjacent to each other with respect to $\varphi(P)$. Finally by Remark \ref{Peripheral} (3) we know that $\varphi(\alpha)$ has to be a separating curve. Thus, $\varphi(\alpha)$ has to be an outer curve.
 
 \textit{Case} $(g_{1},n_{1}) = (1,4)$: Let $P$ be a pants decomposition of $S_{1}$ such that $\alpha \in P$, $P \backslash \{\alpha\}$ is composed solely of non-separating curves, and $\alpha$ has exactly two curves adjacent to it with respect to $P$, say $\beta$ and $\gamma$. If we suppose that $\varphi(\alpha)$ is a non-separating curve, by Remark \ref{Peripheral} (1), we have that $\{\varphi(\alpha),\varphi(\beta)\}$ and $\{\varphi(\alpha),\varphi(\gamma)\}$ are peripheral pairs. In particular $\varphi(\beta)$ and $\varphi(\gamma)$ bound a twice-punctured annulus. But $\beta$ and $\gamma$ are adjacent to each other with respect to $P$, thus by Lemma \ref{Adjacency} we have that $\varphi(\beta)$ and $\varphi(\gamma)$ are adjacent with respect to $\varphi(P)$. This implies that there exists a closed subsurface $\Sigma$ with interior homeomorphic to $S_{0,3}$, induced by $\varphi(P)$ with $\varphi(\beta)$ and $\varphi(\gamma)$ as two of its boundary curves. Note that $\Sigma$ cannot be a once-punctured annulus (this would imply that $\kappa(S_{2}) = 3$), thus there exists a curve $\delta \in P \backslash \{\alpha,\beta,\gamma\}$ such that $\varphi(\delta)$ is the other boundary curve of $\Sigma$. However, this construction would imply that $\varphi(\delta)$ is a separating curve, which contradicts Lemma \ref{Nonsep}. Therefore, $\varphi(\alpha)$ has to be an outer curve.
\end{proof}

\begin{Lema}\label{CharacGenus1}
 Let $S=S_{g,n}$ be such that $\kappa(S) \geq 4$, and $P$ be a pants decomposition composed solely of non-separating curves. Then, $g = 1$ if and only if $\acomp{P}$ is a cycle graph.
\end{Lema}

\begin{proof}
 If $g = 1$, then $P = \{\alpha_{0}, \ldots, \alpha_{n-1}\}$. Then we can have a homeomorphism $f: S \to A/\sim$ where $A$ is an $n$-punctured annulus and $\sim$ identifies its boundaries, such that the image of the boundaries under the quotient is $\alpha_{0}$. This implies that every $\alpha_{i}$ with $i \neq 0$ is mapped to an essential curve in $A$ (that is not a boundary curve); moreover, since for all $i \in \mathbb{Z}/n\mathbb{Z}$ we have that $\alpha_{i}$ is a non-separating curve, it follows that $f(\alpha_{i})$ cannot be a separating curve of $A$. Thus, for all $i \neq 0$, $f(\alpha_{i})$ and ``$f(\alpha_{0})$'' bound a $k$-punctured annulus in $A$ for some $k > 0$. Hence (up to relabelling), for all $i \in \mathbb{Z}/n\mathbb{Z}$, $f(\alpha_{i})$ and $f(\alpha_{i+1})$ bound a once-punctured annulus. This implies that (up to relabelling) for all $i \in \mathbb{Z}/n\mathbb{Z}$, $\{\alpha_{i},\alpha_{i+1}\}$ is a peripheral pair, and then $P$ is (up to homeomorphism) the pants decomposition from Figure \ref{fig:Sec5-0Fig3}, for which $\acomp{P}$ is a cycle graph.
 
 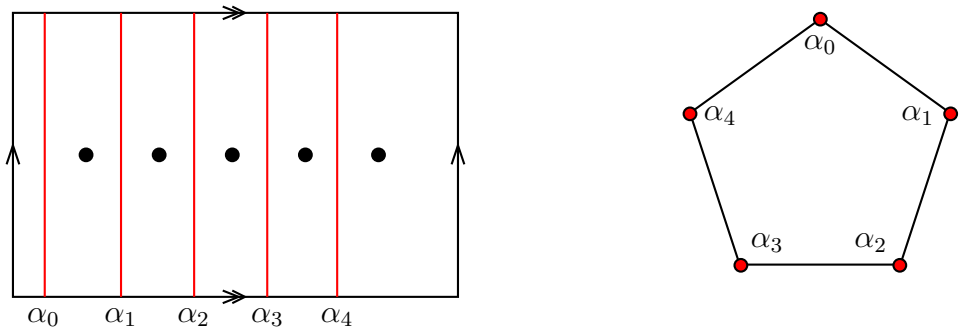
\begin{figure}[ht]
     \centering
     \labellist
     \pinlabel $\alpha_{0}$ [t] at 35 5
     \pinlabel $\alpha_{1}$ [t] at 104 5
     \pinlabel $\alpha_{2}$ [t] at 169 5
     \pinlabel $\alpha_{3}$ [t] at 235 5
     \pinlabel $\alpha_{4}$ [t] at 297 5
     \pinlabel $\alpha_{0}$ [t] at 728 250
     \pinlabel $\alpha_{1}$ [r] at 837 171
     \pinlabel $\alpha_{2}$ [br] at 793 43
     \pinlabel $\alpha_{3}$ [bl] at 661 43
     \pinlabel $\alpha_{4}$ [l] at  617 171
     \endlabellist
     \includegraphics[height=4cm]{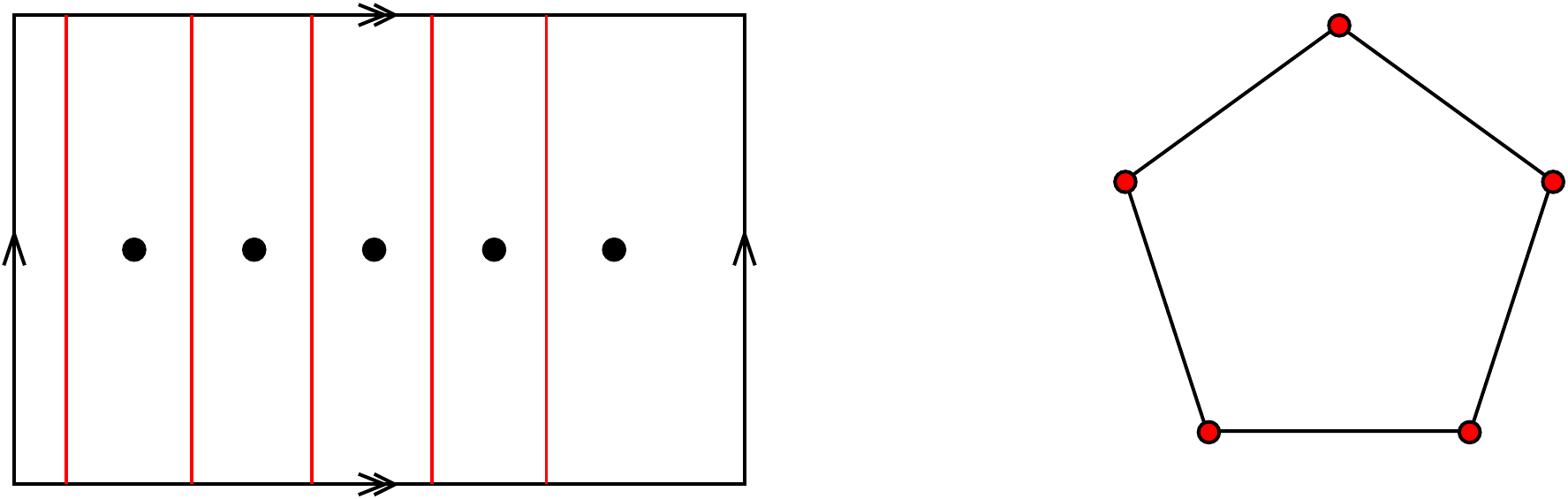}
     \caption{On the left, a punctured torus with a pants decomposition composed solely of non-separating curves; on the right its corresponding adjacency graph (a cycle graph).}
     \label{fig:Sec5-0Fig3}
 \end{figure}
 
 If $\acomp{P}$ is a cycle graph, we can label the curves in $P$ so that $P = \{\alpha_{0}, \ldots, \alpha_{\kappa}\}$ (with $\kappa = \kappa(S)$), and for all $i \in \mathbb{Z}/\kappa\mathbb{Z}$, the vertices corresponding to $\alpha_{i}$ and $\alpha_{i+1}$ are adjacent in $\acomp{P}$ (see the right side of Figure \ref{fig:Sec5-0Fig3}).  By Remark \ref{Peripheral} (1), this implies that for all $i \in \mathbb{Z}/\kappa\mathbb{Z}$, $\{\alpha_{i},\alpha_{i+1}\}$ is a peripheral pair, and thus bound a once-punctured annulus $A_{i}$ in $S$. The result follows from reconstructing $S$ using the annuli $A_{i}$.
\end{proof}

\begin{proof}[\textbf{Proof of Theorem \ref{TopRig}}]
 Since this theorem was already proved by the author in the case $g_{1} \geq 3$ (Lemma 3.16 in \cite{JHH2}), let $g_{1} \leq 2$. Then, we divide the proof into three cases depending on the genus of $S_{1}$:\\
 \textit{Case} $g_{1} = 0$: Let $P$ be a pants decomposition of $S_{1}$. Since $g_{1}=0$, this implies that $P$ is composed solely of separating curves. By Remark \ref{PantsToPants} and Lemmas \ref{Non-outer} and \ref{Outer}, we have that $\varphi(P)$ is a pants decomposition of $S_{2}$ composed solely of separating curves, which is only possible if $g_{2} = 0$. The result follows from the fact that $\kappa(S_{1}) = \kappa(S_{2})$ and the Classification of Surfaces.\\
 \textit{Case} $g_{1} = 1$: Let $P$ be a pants decomposition of $S_{1}$ composed solely of non-separating curves. By Lemma \ref{CharacGenus1} we have that $\acomp{P}$ is a cycle graph. This is preserved by $\varphi$ due to Lemma \ref{Adjacency}. Hence, by Remark \ref{PantsToPants} and Lemmas \ref{Adjacency} and \ref{Nonsep}, $\varphi(P)$ is a pants decomposition of $S_{2}$ composed solely of non-separating curves, such that $\acomp{\varphi(P)}$ is a cycle graph. Then, by Lemma \ref{CharacGenus1}, $g_{2} = 1$. As before, the result follows from the fact that $\kappa(S_{1}) = \kappa(S_{2})$ and the Classification of Surfaces.\\
 \textit{Case} $g_{1} = 2$: Let $P$ be a pants decomposition composed solely of non-separating curves, whose adjacency graph is not a cycle graph by Lemma \ref{CharacGenus1}. By Remark \ref{PantsToPants} and Lemmas \ref{Adjacency} and \ref{Nonsep}, we have that $\varphi(P)$ is a pants decomposition of $S_{2}$ composed solely of non-separating curves, with $\acomp{\varphi(P)}$ not a cycle graph. This implies that $g_{2} \geq 2 = g_{1}$. Since $\kappa(S_{1}) = \kappa(S_{2})$, this implies that $n_{1} \geq n_{2}$.
 
 If $n_{1} = 1$, the result then follows from $\kappa(S_{1}) = \kappa(S_{2})$ and the Classification of Surfaces.
 
 If $n_{1} > 1$, let $M$ be a multicurve of cardinality $\lfloor \frac{n_{1}}{2} \rfloor$, composed solely of outer curves. By Lemmas \ref{Multicurves} and \ref{Outer}, $\varphi(M)$ is a multicurve of $S_{2}$ with cardinality $\lfloor \frac{n_{1}}{2} \rfloor$ and composed solely of outer curves. This implies that $2 \lfloor \frac{n_{1}}{2} \rfloor \leq n_{2}$. Thus, either $n_{1} = n_{2}$, or $n_{1} = n_{2} + 1$; but the latter case leads us to a contradiction when substituting in $\kappa(S_{1}) = \kappa(S_{2})$. Therefore $n_{1} = n_{2}$ and the result follows from the fact that $\kappa(S_{1}) = \kappa(S_{2})$ and the Classification of Surfaces.
\end{proof}
\bibliographystyle{plain}
\bibliography{bibliography}
\vfill
\noindent Jesús Hernández Hernández\\
\begin{tabular}{l}
    \href{mailto:jhdez@matmor.unam.mx}{\texttt{jhdez@matmor.unam.mx}}\\
    \href{https://sites.google.com/site/jhdezhdez/}{\texttt{https://sites.google.com/site/jhdezhdez/}}\\
    Centro de Ciencias Matemáticas\\
    Universidad Nacional Autónoma de México\\
    Morelia, Mich. 58190\\
    México
\end{tabular}
\end{document}